\documentclass{amsart}
\usepackage{style}
\begin{document}
%%%%%%%%%%%%%%%%%%%%%%%
% TITLE AND AUTHORS INFO
%%%%%%%%%%%%%%%%%%%%%%%
\title[Symmetries captured by actions of weak Hopf algebras]{Symmetries of algebras captured by actions of weak Hopf algebras}
\author{Fabio Calder\'{o}n}
\author{Hongdi Huang}
\author{Elizabeth Wicks}
\author{Robert Won}
\address{Calder\'{o}n: Department of Mathematics, Universidad de Los Andes, Bogot\'{a}, DC, Colombia}
\email{f.calderonmateus@uniandes.edu.co}
\address{Huang: Department of Mathematics, Shanghai University, Shanghai 200444, China}
\email{hdhuang@shu.edu.cn}
\address{Wicks: Microsoft Corporation, Redmond, WA 98052, USA}
\email{elizabeth.wicks@microsoft.com}
\address{Won: Department of Mathematics, The George Washington University, Washington, DC 20052, USA}
\email{robertwon@gwu.edu}
%%%%%%%%%%%%%%%%%%%%%%%
% ABSTRACT
%%%%%%%%%%%%%%%%%%%%%%%
\begin{abstract} 
In this paper, we present a generalization of well-established results regarding  symmetries of $\Bbbk$-algebras, where $\Bbbk$ is a field. Traditionally, for a $\Bbbk$-algebra $A$, the group of $\Bbbk$-algebra automorphisms of $A$ captures the symmetries of $A$ via group actions. Similarly, the Lie algebra of derivations of $A$ captures the symmetries of $A$ via Lie algebra actions. In this paper, given a category $\mathcal{C}$ whose objects possess $\Bbbk$-linear monoidal categories of modules, we introduce an object $\operatorname{Sym}_{\mathcal{C}}(A)$ that captures the symmetries of $A$ via actions of objects in $\mathcal{C}$. Our study encompasses various categories whose objects include groupoids, Lie algebroids, and more generally, cocommutative weak Hopf algebras. Notably, we demonstrate that for a positively graded non-connected $\Bbbk$-algebra $A$, some of its symmetries are naturally captured within the weak Hopf framework. 
\end{abstract}

%%%%%%%%%%%%%%%%%%%%%%%
% KEYWORDS AND CLASSIFICATION
%%%%%%%%%%%%%%%%%%%%%%%
% Consult here: https://zbmath.org/static/msc2020.pdf
%%%%%%%%%%%%%%%%%%%%%%%
\maketitle
%%%%%%%%%%%%%%%%%%%%%%%
% TEMPORAL TOOLS FOR WRITTING
% Comment in the final version
%%%%%%%%%%%%%%%%%%%%%%%
%\noindent \fc{[FC's edits {\tt $\backslash$fc}]},
%\hh{[HH's edits {\tt $\backslash$hh}]},
%\ew{[EW's edits {\tt $\backslash$ew}]},
%\rw{[RW's edits {\tt $\backslash$rw}]},
%\blue{items to be filled in {\tt $\backslash$blue}},
%\red{math concerns in {\tt $\backslash$red}},
%[comments in brackets], strike-out %\stkout{text} with {\tt $\backslash$stkout}.

% Table of contents
%\setcounter{tocdepth}{2}
%\tableofcontents

%\setcounter{section}{-1} % Activate if want introduction as section 0.
%\footnote{[CW: TOC is just here so that we can see organization-- intend on removing]}

%%%%%%%%%%%%%%%%%%%%%%%
%%%%%%%%%%%%%%%%%%%%%%%
%%%%%%%%%%%%%%%%%%%%%%%

\section{Introduction}\label{sec:intro}
Let $\kk$ be an algebraically closed field of characteristic $0$. Symmetries of noncommutative algebras over $\kk$ have been extensively studied in the context of Hopf algebra actions and coactions, with notable success for connected $\kk$-algebras 
(see, e.g., \cite{Mo, K, CKWZ3, EW1}). However, extending these results to not-necessarily-connected $\kk$-algebras is not a straightforward task. In recent years, the study of actions and coactions of weak Hopf algebras has emerged as a promising avenue for investigating symmetries of such $\kk$-algebras, with several attempts made in this direction (see, e.g., \cite{BNS, C-DG, Nik02, HWWW20,WWW}). In this article, we contribute to this ongoing effort by providing further insights and results for (cocommutative) weak Hopf algebra actions on noncommutative, not necessarily connected $\kk$-algebras. 

Our main motivation for this article is the well-known correspondence between two action frameworks for groups and Lie algebras: the automorphism group $\AutA$ captures the group actions on $A$, while the Lie algebra of derivations $\Der(A)$ captures the Lie algebra actions on $A$. Namely, for a group $G$, a $G$-module algebra structure on a $\kk$-algebra $A$ yields a group morphism $G \to \AutA$, and conversely any group morphism $G \to \AutA$ induces a $G$-module algebra structure on $A$. Likewise, for a Lie algebra $\mf{g}$, a $\mf{g}$-module algebra structure on $A$ yields a Lie algebra morphism $\mf{g} \to \Der(A)$, and conversely any Lie algebra morphism $\mf{g} \to \Der(A)$ induces a $\mf{g}$-module algebra structure on $A$. 

Our aim is to study symmetries of $\kk$-algebras $A$, captured by actions of algebraic structures $H$ which resemble (cocommutative) Hopf algebras. To achieve this, we will generalize the classical correspondences to establish a connection between categorical and representation-theoretic frameworks for these $H$-actions on $A$.

Throughout this paper, we will focus on a specific set of categories.

\begin{notation}\label{not:cats}
 Let $X$ be a nonempty finite set. We consider the following categories:
 \begin{itemize}
 \item $\Group$, of \emph{groups} together with group morphisms,

 \item $\Lie$, of \emph{Lie algebras} together with Lie algebra morphisms,
 
 \item $\Hopf$, of \emph{Hopf algebras} together with Hopf algebra morphisms,

 \item  $\Grpd$, of \emph{groupoids} together with groupoid morphisms (i.e., functors),
  
 \item $\XGrpd$, the subcategory of $\Grpd$ whose objects are groupoids with object set $X$ (called \emph{$X$-groupoids}), together with $X$-preserving groupoid morphisms (see Definition~\ref{def:XGrpd}),
 
 \item $\XLie$, of \emph{$X$-Lie algebroids}, together with $X$-Lie algebroid morphisms (see Definition~\ref{def:XLie}),
 
 \item $\XWHA$, of weak Hopf algebras with a complete set of grouplike idempotents indexed by $X$ (called \emph{$X$-weak Hopf algebras}), together with $X$-preserving weak Hopf algebra morphisms (see Definition~\ref{def:xxx}).
 \end{itemize}
 
 We also consider the following full subcategories of $\Hopf$:
 \begin{itemize}
 \item $\GrAlg$, of \emph{group algebras},
 
 \item $\EnvLie$, of \emph{enveloping algebras of Lie algebras},
 
 \item $\CocomHopf$, of \emph{cocommutative Hopf algebras}.
 \end{itemize}
 
 As well as the following full subcategories of $\XWHA$:
 \begin{itemize}
 \item $\XGrAlg$, of \emph{$X$-groupoid algebras} (see Definition~\ref{def:xxx}),
 
 \item $\XEnvLie$, of \emph{$X$-enveloping algebras of $X$-Lie algebroids} (see Definition~\ref{def:envelopLiealgebroid}),
 
 \item $\XCocomWHA$, of \emph{cocommutative $X$-weak Hopf algebras}.
 \end{itemize}

Note that all the categories $\CC$ mentioned above have a common property: for any object $H$ in $\CC$, there is a notion of a category of left $H$-modules, which we denote $H\lmod$, and this category is a $\kk$-linear monoidal category (although $\otimes_{H\lmod}$ is not necessarily given by $\otimes_{\kk}$). We refer to any of the categories above as a \emph{category of Hopf-like structures} and provide further details in Remark~\ref{def:catHopflike}.
\end{notation}

The categories listed below consist of objects that will be acted on by Hopf-like structures.

\begin{notation}
\label{not:cats2}
Let $X$ be a nonempty finite set. We consider the following categories:
\begin{itemize}
 \item $\Set$, of \emph{sets} together with set-theoretic functions,
 \item $\Vec$, of \emph{$\kk$-vector spaces} together with linear maps,
 \item $\Algcat$, of \emph{$\kk$-algebras} together with $\kk$-algebra morphisms,
 \item $\XAlg$, the subcategory of $\Algcat$ whose objects are $\kk$-algebras with a complete set of nonzero orthogonal idempotents indexed by $X$ (called \emph{$X$-algebras}), together with $X$-preserving $\kk$-algebra morphisms (see Definition~\ref{def:xxx}).
  \end{itemize}
\end{notation}

Throughout, we assume the following hypothesis.
 
\begin{hypothesis*}
An unadorned $\otimes$ signifies $\otimes_\kk$. Unless otherwise stated, all $\kk$-algebras are associative and unital, and $\kk$-algebra morphisms are unital. If not explicitly specified, modules are assumed to be left-sided.
\end{hypothesis*}

Suppose $\CC$ is a category of Hopf-like structures. If $H$ is an object in $\CC$, we say that $H$ \emph{acts} on a $\kk$-algebra $A$ if $A$ is an \emph{algebra object} in the category $H\lmod$ (Definition~\ref{def:cataction}).

Let $A$ be an arbitrary $\kk$-algebra. In this paper, our main results identify a distinguished object in $\CC$, called the \emph{symmetry object} over $A$ in $\CC$ and denoted as $\Sym_{\CC}(A)$, which captures all actions on $A$ by objects of $\CC$. Specifically, we define $\Sym_{\CC}(A)$ as an object in $\CC$ acting on $A$ with the property that, for every object $H$ in $\CC$, there is a bijective correspondence between $H$-actions on $A$ and morphisms $H \to \Sym_{\CC}(A)$ in $\CC$ (see Definition~\ref{def:Sym}). Hence, the classical results mentioned above as our motivation can be concisely restated as follows: for a $\kk$-algebra $A$, $\Sym_{\Group}(A)=\AutA$ (Remark~\ref{rem:group}) and $\Sym_{\Lie}(A)=\Der(A)$ (Remark~\ref{rem:Lie}).

The first category we study in this paper is the category $\XGrpd$. 

\begin{definition*}[Definitions~\ref{def:XdecompA},~\ref{def:autgroupoid2}]
Let $X$ be a nonempty finite set.
\begin{enumerate}[label=\normalfont(\alph*)]
\item We say that a nonzero $\kk$-algebra $A$ is an {\it $X$-decomposable algebra} if $A$ has a $\kk$-algebra decomposition $A = \bigoplus_{x \in X} A_x$, where each $A_x$ is a unital $\kk$-algebra (some of which may be $0$). 
\item For an $X$-decomposable $\kk$-algebra $A = \bigoplus_{x \in X} A_x$, define the groupoid, $\XAutA$, to have object set $X$, and where $\Hom_{\XAutA}(x,y)$ is the set of unital $\kk$-algebra isomorphisms $A_x \to A_y$, for all $x,y \in X$.
\end{enumerate}
\end{definition*}

Our first main result identifies the symmetry object $\Sym_{\XGrpd}(A)$ over $X$-decomposable algebras in $\XGrpd$. Notice that, by taking $X$ to be the singleton set, we recover the classical result for groups mentioned above.

\begin{theorem*} [Proposition~\ref{prop:groupoidalg}]
Let $A = \bigoplus_{x\in X} A_x$ be an $X$-decomposable $\kk$-algebra. Then $A$ is an $\XAutA$-module algebra. Moreover, for any $X$-groupoid $\GG$, the $\GG$-module algebra structures on $A$ correspond bijectively to the $X$-groupoid morphisms $\GG \to \XAutA$. Hence, $\Sym_{\XGrpd}(A) = \XAutA$. 
\end{theorem*}

Our work on groupoid actions on $\kk$-algebras unifies previous work on so-called partial actions of groupoids on $\kk$-algebras (e.g., \cite{BP2012,PF13,PT2014}) with categorical methods.
The findings above can be linearized for actions of groupoid algebras via the result below, which generalizes well-known adjunctions for groups.

\begin{theorem*}[Theorem~\ref{thm:repkGrp}]
Let $X$ be a nonempty finite set.
\begin{enumerate}[label=\normalfont(\alph*)]
\item The groupoid algebra functor
\[
\kk(-): \XGrpd \longrightarrow \XAlg
\]
is well-defined. Moreover, it is left adjoint to the functor,
\[
(-)^\times_X:\XAlg \longrightarrow \XGrpd,
\]
that forms a groupoid of local units with object set $X$ \textnormal{(Definition~\ref{def:groupoidL})}.
\item The groupoid algebra functor 
\[
\kk(-): \XGrpd \longrightarrow \XWHA
\]
is well-defined. Moreover, it is left adjoint to the functor,
\[
\Gamma(-): \XWHA \longrightarrow \XGrpd,
\]
that forms a groupoid of grouplike elements with object set $X$ \textnormal{(Definition~\ref{def:groupoidGamma})}.
\end{enumerate}
\end{theorem*}

By applying this linearization, we calculate the symmetry object over $X$-decomposable algebras in the category of $X$-groupoid algebras $\XGrAlg$, and as a special case, the classical result for the category of group algebras (Remark~\ref{rem:kGrp}).

\begin{corollary*}[Proposition~\ref{prop:kGrpd}]
Let $A = \bigoplus_{x\in X} A_x$ be an $X$-decomposable $\kk$-algebra. Then $A$ is a $\kk(\XAutA)$-module algebra. Moreover, for any $X$-groupoid $\GG$, the $\kk\GG$-module algebra structures on $A$ correspond bijectively to the $X$-weak Hopf algebra morphisms $\kk \GG \to \kk(\XAutA)$. Hence, $\Sym_{\XGrAlg}(A) = \kk(\XAutA)$.
\end{corollary*}

As an application of the preceding findings, we establish a result that showcases the reduction of studying groupoid actions on the $\kk$-algebra $A$ to the study of group actions when $A$ is a domain. This result highlights how the study of groupoid actions represents a genuine extension of the classical theory of group actions on connected $\kk$-algebras.

\begin{corollary*}[Corollary~\ref{cor:grpoid-innerfaith}]
Suppose that $\GG$ is a finite groupoid and $A$ is a domain such that $A$ is an inner faithful $\kk \GG$-module algebra. Then $\GG$ is a disjoint union of groups, and at most one of the groups is nontrivial.
\end{corollary*}

Now, let us shift our focus to the category of $X$-Lie algebroids. Nikshych established a generalization of the Cartier--Gabriel--Kostant--Milnor--Moore theorem for cocommutative weak Hopf algebras $H$, which states that $H$ is isomorphic to $\left(\bigoplus_{x \in X}U(\mf{g}_x)\right) \# \kk \GG$, where $\GG$ is a groupoid acting by conjugation on a direct sum, indexed by a nonempty finite set $X$, of universal enveloping algebras of Lie algebras $\{\mf{g}_x\}_{x \in X}$ \cite[Theorem~3.2.4]{NDthesis}. Motivated by this result, we introduce the notion of an  \emph{$X$-Lie algebroid}, denoted by $\GLie$, which is defined as a direct sum of Lie algebras $\GLie = \bigoplus_{x \in X} \mf{g}_x$ equipped with a partially defined bracket (see Definition~\ref{def:Liealgebroid}). The \emph{$X$-universal enveloping algebra} of $\GLie$ is defined as $U_X(\GLie) = \bigoplus_{x \in X} U(\mf{g}_x)$ (Definition~\ref{def:envelopLiealgebroid}). We prove that actions of $X$-Lie algebroids correspond to actions of their $X$-universal enveloping algebras via the result below, which generalizes well-known adjunctions for Lie algebras.

\begin{theorem*}[Theorem~\ref{thm:adjLie}]
Let $X$ be a nonempty finite set.
\begin{enumerate}[label=\normalfont(\alph*)]
\item The $X$-universal enveloping algebra functor
\[
U_x(-): \XLie \longrightarrow \XAlg
\]
is well-defined. Moreover, it is left adjoint to the functor,
\[
\mathcal{L}_X(-):\XAlg \longrightarrow \XLie,
\]
that forms an associated $X$-Lie algebroid \textnormal{(Example~\ref{ex:LieXAlg})}.
\item The $X$-universal enveloping algebra functor 
\[
U_X(-): \XLie \longrightarrow \XWHA
\]
is well-defined. Moreover, it is left adjoint to the functor,
\[
P_X(-): \XWHA \longrightarrow \XLie,
\]
that forms an $X$-Lie algebroid of primitive elements \textnormal{(Example~\ref{ex:primitiveWHA})}.
\end{enumerate}
\end{theorem*}

We then calculate the symmetry object over $X$-decomposable algebras in the categories of $X$-Lie algebroids and $X$-universal enveloping algebras, obtaining a generalization of the classical result for Lie algebras and universal enveloping algebras of Lie algebras as follows. Here, for an $X$-decomposable $\kk$-algebra $A = \bigoplus_{x \in X}A_x$ we denote by $\Der_X(A)$ the $X$-Lie algebroid $\bigoplus_{x \in X} \Der(A_x)$ (Notation~\ref{not:GLLie2}).

\begin{theorem*}[Proposition~\ref{prop:Liealgebroidalg}] 
Let $A = \bigoplus_{x\in X} A_x$ be an $X$-decomposable $\kk$-algebra. Then $A$ is an $\Der_X(A)$-module algebra. Moreover, for any $X$-Lie algebroid $\GLie$, the $\GLie$-module algebra structures on $A$ correspond bijectively to the $X$-Lie algebroid morphisms $\GLie \to \Der_X(A)$. Hence, $\Sym_{\XLie}(A) = \Der_X(A)$.
\end{theorem*}

\begin{corollary*}[Proposition~\ref{prop:U(GLie)}]
Let $A = \bigoplus_{x\in X} A_x$ be an $X$-decomposable $\kk$-algebra. Then $A$ is an $U_X(\Der_X(A))$-module algebra. Moreover, for any $X$-Lie algebroid $\GLie$, the $U_X(\GLie)$-module algebra structures on $A$ correspond bijectively to the $X$-weak Hopf algebra morphisms $U_X(\GLie) \to U_X(\Der_X(A))$. Hence, $\Sym_{\XEnvLie}(A) = U_X(\Der_X(A))$.
\end{corollary*}

This brings us to the final main result of our paper, which calculates the symmetry object over $X$-decomposable algebras in the category of cocommutative $X$-weak Hopf algebras.
	 
\begin{theorem*}[Theorem~\ref{thm:cocomweakHopf}]
Let $A = \bigoplus_{x\in X} A_x$ be an $X$-decomposable $\kk$-algebra and define $K := U_X(\Der_X(A)) \# \kk(\XAutA)$. Then $A$ is an $K$-module algebra. Moreover, for any cocommutative $X$-weak Hopf algebra $H$, the $H$-module algebra structures on $A$ correspond bijectively to the $X$-weak Hopf algebra morphisms $H \to K$. Hence, $\Sym_{\XCocomWHA}(A) = K$.
\end{theorem*}

By taking $X$ to be the singleton set above, we obtain the following result for actions of cocommutative Hopf algebras.

\begin{theorem*}[\Cref{CocomHopf}] 
Let $A$ be a $\kk$-algebra and define $K := U(\Der(A)) \# \kk(\AutA)$. Then $A$ is a $K$-module algebra. Moreover, for any cocommutative Hopf algebra $H$, the $H$-module algebra structures on $A$ correspond bijectively to the Hopf algebra morphisms $H \to K$. Hence, $\Sym_{\CocomHopf}(A) = K$.
\end{theorem*}

The results presented in this paper reinforce the idea that weak Hopf algebra actions provide a promising framework for investigating the symmetries of noncommutative $\kk$-algebras, particularly in the non-connected case. We summarize our constructions of $\Sym_{\CC}(A)$ for the various categories $\CC$ of Notation~\ref{not:cats} in Table~\ref{tab:summary}.

\begin{table}[H]
\centering
{\footnotesize
\begin{tabular}{|c|c|c|}
\hline 
& &\\ [-1em]
\underline{Category $\CC$ of classical Hopf-like structures} & \underline{$\Sym_{\CC}(A)$ for algebra $A$} & \underline{Reference}\\ [.2em]

$\Group$ & $\AutA$ & Rem.~\ref{rem:group}\\ [.2em]

$\GrAlg$ & $\kk(\AutA)$ & Rem.~\ref{rem:kGrp}\\ [.2em]

$\Lie$ & $\Der(A)$ & Rem.~\ref{rem:Lie}\\ [.2em]

$\EnvLie$ & $U(\Der(A))$ & Rem.~\ref{rem:U(Lie)} \\ [.2em]

$\CocomHopf$ & $U(\Der(A)) \# \kk(\AutA)$ & Rem.~\ref{CocomHopf} \\ [.2em]
\hline
\hline &\\ [-1em]
\underline{Category $\CC$ of Hopf-like structures with base $X$} & \underline{$\Sym_{\CC}(A)$ for $X$-decomp. \hspace{-.05in} alg. \hspace{-.05in} $A$} & \underline{Reference}\\ [.2em]

$\XGrpd$ & $\XAutA$ & Prop.~\ref{prop:groupoidalg} \\ [.2em]

$\XGrAlg$ & $\kk(\XAutA)$ & Prop.~\ref{prop:kGrpd} \\[.2em]

$\XLie$ & $\Der_X(A)$ & Prop.~\ref{prop:Liealgebroidalg} \\[.2em]

$\XEnvLie$ & $U_X(\Der_X(A))$ & Prop.~\ref{prop:U(GLie)} \\[.2em]

$\XCocomWHA$ & $U_X(\Der_X(A)) \# \kk(\XAutA)$ & Thm.~\ref{thm:cocomweakHopf} \\[.2em]
\hline
\end{tabular}

\caption{The symmetry object $\Sym_{\CC}(A)$ in the category $\CC$ that captures symmetries of $\kk$-algebras $A$ by actions of objects in $\CC$.}\label{tab:summary}
}
\end{table}

The paper is structured as follows. In Section~\ref{sec:symdef}, we provide a precise definition of $\Sym_{\CC}(A)$, where $\CC$ is a category of Hopf-like structures and $A$ is a $\kk$-algebra. In Section~\ref{sec:groupoid}, we investigate groupoid actions and extend them linearly to study groupoid algebra actions in Section~\ref{sec:groupoidalg}. Section~\ref{sec:Liealgebroids} is dedicated to the study of $X$-Lie algebroids and their universal enveloping algebras. Leveraging Nikshych's generalization of the Gabriel--Kostant--Milnor--Moore theorem, we are able to study actions of cocommutative weak Hopf algebras in Section~\ref{sec:cocomweak}. Finally, in Section~\ref{sec:examples}, we illustrate our results for actions on polynomial algebras.

\section{Symmetries in module categories over Hopf-like structures}
\label{sec:symdef}

Recall that an additive category is called \emph{$\kk$-linear} if the hom-sets are $\kk$-vector spaces and the composition of morphisms is bilinear over $\kk$. A \emph{monoidal category} $(\DD,\otimes_{\DD},\unit_{\DD})$ is a category $\DD$ equipped with a bifunctor $\otimes_{\DD} : \DD \times \DD \to \DD$, an object $\unit_{\DD}$, and natural isomorphisms for associativity and unitality such that the pentagon and triangle axioms hold (see e.g. \cite[Ch.~2]{EGNO} for details).

While we do not provide a precise set of axioms that define a category $\CC$ as consisting of ``Hopf-like structures" beyond the given list in Notation~\ref{not:cats}, these categories do share several common properties, which we discuss in the following remark. However, we hope that the results presented in this paper can be extended to other categories, including those involving non-cocommutative Hopf-like structures. Exploring this extension is an open question for future research.

\begin{remark}\label{def:catHopflike}
Let $\CC$ be any of the categories of Hopf-like structures introduced in Notation~\ref{not:cats}. We will see in this paper that each of these categories possesses the following properties:
\begin{enumerate}[label=\normalfont(\alph*)]
    \item $\CC$ is a concrete category.
    \item For each object $H$ in $\CC$, there exists a notion of a quotient object $H/I$ in $\CC$ associated with certain subsets $I$ of $H$. The set $I$ may not necessarily correspond to a subobject of $H$ in $\CC$, nor does it necessarily need to be an object in $\CC$.
    \item One of the following cases occurs:
    \begin{enumerate}[label=\normalfont(\Roman*)]
        \item There exists a bifunctor $\boxtimes_{\CC}: \CC \times \Vec \to \Set$, so that the elements of $H \boxtimes_{\CC} V$ have the form $h \boxtimes_{\CC} v$, for certain $h \in H$ and $v \in V$; in this case we say $\CC$ is of \emph{type I}.
        \item There exists a bifunctor $\boxtimes_{\CC}: \CC \times \Vec \to \Vec$, so that the elements of $H \boxtimes_{\CC} V$ are sums of elements of the form $h \boxtimes_{\CC} v$, for certain $h \in H$ and $v \in V$; in this case we say $\CC$ is of \emph{type II}. 
    \end{enumerate}
    The  specifics of $\boxtimes_{\CC}$ depend on the category $\CC$, and details will be given for each specific category (see Example~\ref{ex:hopflike} below).
    \item For each object $H$ in $\CC$, there is a notion of endowing a $\kk$-vector space $V$ with a structure of $H$-module via a set-theoretic function $\ell_{H,V}: H \boxtimes_{\CC} V \to V$ (if $\CC$ is of type I) or a $\kk$-linear map $\ell_{H,V}: H \boxtimes_{\CC} V \to V$ (if $\CC$ is of type II). In either case, we use the notation $h \cdot v$ to denote $\ell_{H,V}(h \boxtimes v)$ and say that $H$ \emph{acts} on $V$.
       \item For each object $H$ in $\CC$, there is a notion of a morphism of $H$-modules, and thus there exists a  category $H\lmod$ of $H$-modules and their morphisms. This category can be given a $\kk$-linear monoidal structure $(H\lmod, \otimes_{H\lmod}, \unit_{H\lmod})$.
\end{enumerate}

In summary, the categories of Hopf-like structures introduced in this work are characterized by the ability to associate a $\kk$-linear monoidal category of modules to each object, with actions on these modules defined by specific functions (called \emph{action maps}). The following classical categories all satisfy the properties in the above remark.
\end{remark}

\begin{example}\label{ex:hopflike}
  \begin{itemize}
        \item When $\CC = \Group$, we may take $\boxtimes_{\Group}=\times$, the set-theoretic Cartesian product. In this case, if $G$ is a group, the action map $\ell_{G,V}: G \times V \to V$ is a $\kk$-linear group action of $G$ on a vector space  $V$. Then $(G\lmod, \otimes_{\kk}, \kk)$ is a $\kk$-linear monoidal category. If $V,W$ are $G$-modules, then $V \otimes W$ is a $G$-module via the action $g \cdot (v \otimes w) = (g \cdot v) \otimes (g \cdot w)$ for $g \in G$, $v \in V$, and $w \in W$.
        
        \item When $\CC = \Lie$, we may take $\boxtimes_{\Lie}=\times$, the set-theoretic Cartesian product. In this case, if $\mf{g}$ is a Lie algebra, the action map $\ell_{\mf{g},V}: \mf{g} \times V \to V$ is a $\kk$-bilinear map such that $[x,y]\cdot v = x\cdot (y\cdot v)-y\cdot(x\cdot v)$, for all $x,y\in \mf{g}$ and $v\in V$. Then $(\mf{g}\lmod, \otimes_{\kk}, \kk)$ is a $\kk$-linear monoidal category. If $V, W$ are $\mf{g}$-modules, then $V \otimes W$ is a $\mf{g}$-module via the action $x \cdot (v \otimes w) = (x \cdot v) \otimes w + v \otimes (x \cdot w)$.
        
        \item When $\CC = \Hopf$, we may take $\boxtimes_{\Hopf}=\otimes_{\kk}$. In this case, if $H$ is a Hopf algebra, the action map $\ell_{H,V}: H \otimes V \to V$ is a linear map that makes $V$ a left module over the $\kk$-algebra $H$ (see \cite[Section~1.6]{Mo} for details). Then $(H\lmod, \otimes_{\kk}, \kk)$ is a $\kk$-linear monoidal category. If $V, W$ are $H$-modules, then $V \otimes W$ is an $H$-module via the action $h \cdot (v \otimes w) = (h_1 \cdot v) \otimes (h_2 \cdot w)$, where $\Delta(h)=h_1 \otimes h_2 $ is the sumless notation for comultiplications. As consequence, $\GrAlg$, $\EnvLie$ and $\CocomHopf$ are also categories of Hopf-like structures.
    \end{itemize}
    Later we will show that the more general categories $\XGrpd$, $\XLie$, $\XWHA$, $\XGrAlg$, $\XEnvLie$ and $\XCocomWHA$ satisfy the properties in Remark~\ref{def:catHopflike} (see Remarks \ref{GGboxtimes}, \ref{GLieboxtimes}, and \ref{Hmodboxtimes}). 
\end{example} 

\begin{definition}[Inner-faithful action]
\label{def:innerfaithful}
   Let $\CC$ be a category of Hopf-like structures with bifunctor $\boxtimes_{\CC}: \CC \times \Vec \to \Set$ (if $\CC$ is of type I) or bifunctor $\boxtimes_{\CC}: \CC \times \Vec \to \Vec$ (if $\CC$ is of type II). If $H$ is an object in $\CC$ and $V$ is an $H$-module via action map $\ell_{H,V}$, we say that $H$ acts \emph{inner-faithfully} on $V$ if there is no proper quotient $H/I$ of $H$ in $\CC$ such that $H/I$ acts on $V$ via action map $\ell_{H/I, V}$, and the following diagram commutes (in $\Set$ or in $\Vec$):
\[
\begin{tikzcd}
H\boxtimes_{\CC} V \arrow[rrd, "{\ell_{H,V}}"] \arrow[d, "\pi\boxtimes_{\CC} \id"'] & &   \\
H/I \boxtimes_{\CC} V \arrow[rr, "{\ell_{H/I, V}}"']                  & & V
\end{tikzcd}
\]
Here, $\pi: H \to H/I$ is the natural projection.
\end{definition}

\begin{definition}[Algebra object, $H$-module algebra]\label{def:cataction} Recall that an \emph{algebra object} $(A,m,u)$ in a monoidal category $(\DD,\otimes_{\DD}, \unit_{\DD})$ is an object $A$ in $\DD$ equipped with morphisms $m: A \otimes_{\DD} A \to A$ and $u: \unit_{\DD} \to A$ in $\DD$ satisfying associativity and unitality axioms (see \cite[Section~7.8]{EGNO} for details).

Let $H$ be an object in a category of Hopf-like structures $\CC$. We say that $A$ is a \emph{left $H$-module algebra} if $A$ is an algebra object in the monoidal category $H\lmod$.
 \end{definition}

\begin{remark}
\label{rem:Hmodulealgebra}
It is important to note that although the objects of $H\lmod$ are $\kk$-vector spaces, the monoidal product $\otimes_{H\lmod}$ may not coincide with the $\kk$-tensor product $\otimes_{\kk}$. Therefore, a monoid object $A$ in $H\lmod$ does not necessarily induce a $\kk$-algebra structure on the underlying $\kk$-vector space $A$. And conversely, for a $\kk$-algebra $A$ that also belongs to $H\lmod$, even if the action map is compatible with the algebra operations, it does not guarantee that $A$ is an algebra object in $\CC$.

However, in the case of the classical Hopf-like structures listed in Example~\ref{ex:hopflike}, where $\otimes_{H\lmod}=\otimes_{\kk}$, it is well known that the algebra objects $A$ in $H\lmod$ correspond bijectively to the $\kk$-algebra structures on $A$ compatible with the action (see, e.g., \cite[Section~7.8]{EGNO}). This is the reason for no distinction between the two notions in the traditional case.

Furthermore, in all of the non-classical categories $\CC$ from Notation~\ref{not:cats}, even though the monoidal product differs from the $\kk$-tensor product, algebra objects in $H\lmod$ do correspond bijectively to $\kk$-algebras with compatible actions (as shown in Lemma~\ref{lem:GGModuleFormulas} for $\CC = \XGrpd$, Lemma~\ref{lem:XLieModuleFormulas} for $\CC = \XLie$, and Proposition~\ref{def:modulealgoverWHA} for $\CC = \XWHA$ or any of its subcategories). Therefore, we will use the convention of referring to such $\kk$-algebras with compatible actions also as \emph{left $H$-module algebras}.
\end{remark}

The symmetry object $\Sym_{\CC}(A)$ of a category of Hopf-like structures $\CC$, introduced here, is designed to capture the symmetries of any $\kk$-algebra $A$ arising from the actions of objects in $\CC$.

\begin{definition}[$\Sym_{\CC}(A)$]\label{def:Sym}
Let $A$ be a $\kk$-algebra and let $\CC$ be a category of Hopf-like structures. If it exists, we denote by $\Sym_{\CC}(A)$ an object in $\CC$ such that:
 \begin{enumerate}[label=\normalfont(\alph*)]
 \item $A$ is a $\Sym_{\CC}(A)$-module algebra via action map $\ell_{\Sym_{\CC}(A), A}: \Sym_{\CC}(A) \boxtimes_{\CC} A \to A$. If $f \boxtimes a \in \Sym_{\CC}(A) \boxtimes_{\CC} A$, we write $f \rt a$ to denote $\ell_{\Sym_{\CC}(A), A}(f \boxtimes a)$.

\item For each object $H$ in $\CC$, if $A$ is an $H$-module algebra via $\ell_{H,A}: H \boxtimes_{\CC} A \to A$, then there exists a unique morphism $\phi: H \to \Sym_{\CC}(A)$ in $\CC$ such that the following diagram commutes:
\[
\begin{tikzcd}
H \boxtimes_{\CC} A \arrow[rrd, "{\ell_{H,A}}"] \arrow[d, dashed, "\phi \boxtimes_{\CC} \id"'] & &   \\
\Sym_{\CC}(A) \boxtimes_{\CC} A \arrow[rr, "{\ell_{\Sym_{\CC}(A), A}}"']                  & & A
\end{tikzcd}
\]
Conversely, every morphism $\phi: H \to \Sym_{\CC}(A)$ in $\CC$ gives $A$ the structure of an $H$-module algebra via $\ell_{H,A}(h \boxtimes a) = \phi(h) \rt a$.
 \end{enumerate}
We call $\Sym_{\CC}(A)$ the \emph{symmetry object} over $A$ in $\CC$.
\end{definition}

\begin{remark}
Let $\CC$ be a category of Hopf-like structures and let $H$ be an object in $\CC$.
\begin{enumerate}[label=\normalfont(\alph*)]
\item Note that, while the definition of an $H$-module algebra is ``category-theoretic'', if for a $\kk$-algebra $A$ the symmetry object $\Sym_{\CC}(A)$ exists, then it provides a ``representation-theoretic" framework for actions on $\kk$-algebras: $A$ is an $H$-module algebra if and only if there is a morphism $H \to \Sym_{\CC}(A)$ in $\CC$.

\item It is not obvious that the symmetry object $\Sym_{\CC}(A)$ always exists in $\CC$. For example, consider the category of abelian groups $\CC = \AbGrp$ (which satisfies the conditions of Remark~\ref{def:catHopflike}) and let $A = \kk[x]$ be the polynomial algebra. Let $G = \langle g \rangle$ be the cyclic group of order $2$. Then $G$ acts on $A$ in two different ways: (1) via the action $\cdot$ defined by $g \cdot x := -x$, and (2) via the action $*$ defined by $g * x := -x + 1$. If $\Sym_{\AbGrp}(A)$ existed, then by definition we would obtain abelian group morphisms $\phi : G \to \Sym_{\AbGrp}(A)$ and $\psi : G \to \Sym_{\AbGrp}(A)$ so that $\phi(g) \rt x = -x$ and $\psi(g) \rt x = -x + 1$. But then the elements $\phi(g)$ and $\psi(g)$ would not commute in $\Sym_{\AbGrp}(A)$ (since they have different actions on $x$), contradicting the fact that $\Sym_{\AbGrp}(A)$ is an abelian group.
\end{enumerate}
\end{remark}

Assuming the existence of the symmetry object, we show two of its basic properties.

\begin{lemma}\label{lem:Symmproperties}
    Let $A$ be a $\kk$-algebra and let $\CC$ be a category of Hopf-like structures. If $\Sym_{\CC}(A)$ exists, then:
    \begin{enumerate}[label=\normalfont(\alph*)]
        \item $\Sym_{\CC}(A)$ is unique up to isomorphism in $\CC$.
        \item $\Sym_{\CC}(A)$ acts inner-faithfully on $A$.
    \end{enumerate}
\end{lemma}

\begin{proof}
Let $S: = \Sym_{\CC}(A)$.

(a) If $T$ is another object in $\CC$ satisfying the two conditions in Definition~\ref{def:Sym}, then we obtain unique morphisms $\phi: S \to T$ and $\psi: T \to S$ such that the following diagram commutes:
\[
\begin{tikzcd}
S \boxtimes_{\CC} A \arrow[rrd, "{\ell_{S,A}}"] \arrow[d, dashed, "\phi \boxtimes_{\CC} \id"'] & &   \\
T \boxtimes_{\CC} A \arrow[rr, "{\ell_{T, A}}"']  \arrow[d, dashed, "\psi \boxtimes_{\CC} \id"']                 & & A \\
S \boxtimes_{\CC} A \arrow[rru, "{\ell_{S,A}}"'] & &  
\end{tikzcd}
\]
Hence, the composition $(\psi \boxtimes_{\CC} \id) \circ (\phi \boxtimes_{\CC} \id) = \id_S \boxtimes_{\CC} \id$, as this is the unique morphism $S \boxtimes_{\CC} A \to  S\boxtimes_{\CC} A$ making the outer diagram commute. Hence, $\phi$ and $\psi$ are isomorphisms.

(b) Suppose that $S/I$ is a quotient object of $S$ which acts on $A$ such that the following diagram commutes:
\[
\begin{tikzcd}
S \boxtimes_{\CC} A \arrow[rrd, "{\ell_{S, A}}"] \arrow[d, "\pi\boxtimes_{\CC} \id"'] & &   \\
S/I \boxtimes_{\CC} A \arrow[rr, "{\ell_{S/I, A}}"']                   & & A
\end{tikzcd}
\]
By definition of of $\Sym_{\CC}(A)$, we obtain a morphism $\phi: S/I \to S$ such that the following commutes:
\[
\begin{tikzcd}
S \boxtimes_{\CC} A \arrow[rrd, "{\ell_{S, A}}"] \arrow[d, "\pi\boxtimes_{\CC} \id"'] & &   \\
S/I \boxtimes_{\CC} A \arrow[rr, "{\ell_{S/I, A}}"']  \arrow[d, dashed, "\phi \boxtimes_{\CC} \id"']                   & & A \\
S \boxtimes_{\CC} A \arrow[rru, "{\ell_{S,A}}"'] & &  
\end{tikzcd}
\]
Hence, the composition $(\phi \boxtimes_{\CC} \id) \circ (\pi \boxtimes_{\CC} \id) = \id_S \boxtimes_{\CC} \id$, as this is the unique morphism $S \boxtimes_{\CC} A \to  S\boxtimes_{\CC} A$ making the outer diagram commute. But this shows that $I$ is trivial so $S/I = S$, as desired.
\end{proof}

In the rest of
this paper, we will demonstrate the existence of the symmetry object $\Sym_{\CC}(A)$ in the categories $\CC$ mentioned in Notation~\ref{not:cats} and describe the object in these cases. It is worth noting that all these categories consist of objects that are (closely related to) weak Hopf algebras, while the algebras $A$ for which we calculate the symmetry object are not necessarily connected. Therefore, weak Hopf algebras can be seen as capturing the symmetries of not-necessarily-connected $\kk$-algebras.

\section{Actions of groupoids on algebras} \label{sec:groupoid}

In this section, we first study modules over groupoids and representations of groupoids. We generalize the correspondence between these two notions, which is well-known in the group setting (see Lemma~\ref{lem:groupoidrep}). Then, for a not necessarily connected $\kk$-algebra $A$, we identify a groupoid $\Sym_{\XGrpd}(A) = \Aut_{\XAlg}(A)$ which generalizes the automorphism group of a connected $\kk$-algebra and captures the symmetries of $A$ in the category $\XGrpd$. We begin by introducing the definition of a groupoid, which will be used consistently throughout this paper.

\begin{definition}[{$\GG$, $\GG_0$, $\GG_1$, $e_x$}]\label{def:groupoid}
	A \emph{groupoid} $\GG=(\GG_0, \GG_1)$ is a small category in which every morphism is an isomorphism. Here, $\GG_0$ (resp., $\GG_1$) is the set of objects (resp., morphisms) of $\GG$. If $g\in \GG_1$, then $s(g)$ (resp., $t(g)$) denotes the source (resp., target) object of $g$. By convention, we compose elements of a groupoid from right to left. For each $x \in \GG_0$, the identity morphism of $\Hom_{\GG}(x,x)$ is denoted by $e_x$. A \emph{groupoid morphism} is simply a functor between groupoids.
\end{definition}

A group $G$ can be viewed as a groupoid with only one object. Unless otherwise stated, we assume that for a groupoid $\GG$, the set of objects $\GG_0$ is nonempty and finite. If $\GG_1$ is also a finite set, then we call $\GG$ \emph{finite}. Throughout we will use the following notation.

\begin{notation}[$X$]\label{not:X}
	Henceforth, let $X$ be a nonempty finite set.
\end{notation}

\begin{definition}[$\XGrpd$]\label{def:XGrpd}
 Let $\XGrpd$ be the category defined as follows:
	 
\begin{itemize}
\item The objects are groupoids $\GG$ such that $\GG_0=X$; we call these \emph{$X$-groupoids}.
		 
\item The morphisms are groupoid morphisms leaving $X$ fixed, i.e., functors $\pi: \GG \to \GG'$ such that $\pi(x)=x$ for all $x\in X$; we call these functors \emph{$X$-groupoid morphisms}.
	\end{itemize}
\end{definition}

\begin{remark}\label{rem:permutation}
    In Definition~\ref{def:XGrpd}, one could allow morphisms of $X$-groupoids to permute the set of objects $X$ rather than requiring them to fix the set. However, we chose the latter for convenience in our proofs. Moreover, in the former case, since $X$ is finite, a suitable relabeling of objects in either $X$-groupoid would reduce to our case. Therefore, all our results apply up to a permutation of $X$ by considering a more relaxed definition of $X$-groupoid morphisms.
\end{remark}

Throughout this section, $\GG$ denotes an $X$-groupoid.

%%%%%%%%%%%%%%%%%

\subsection{Modules over groupoids} \label{sec:modgroupoid}

Some of the concepts here are partially adapted from \cite[Definition~7.1.7]{BHS2011} and \cite[p.~85]{PT2014}.

\begin{definition}[$X$-decomposable vector space]\label{def:XdecompV}
	A vector space $V$ is \emph{$X$-decomposable} if there exists a family $\{V_x\}_{x\in X}$ of subspaces of $V$ such that $V=\bigoplus_{x\in X} V_x$. We call the $\{V_x\}_{x\in X}$ the \emph{components} of $V$.
 
 If $V=\bigoplus_{x\in X} V_x$ and $W=\bigoplus_{x\in X} W_x$ are $X$-decomposable vector spaces, a linear map $f: V \to W$ is said to be \emph{$X$-decomposable} if there is a family $\{ f_x : V_x \to W_x \}_{x\in X}$ of linear maps such that $f|_{V_x} = f_x$ for all $x \in X$. In this case, we write $f = (f_x)_{x \in X}$.
\end{definition}

Every vector space $V$ has a trivial $X$-decomposition if $|X|=1$. Even if $|X| > 1$, $V$ has an $X$-decomposition where $V_x = V$ for one $x \in X$, and $V_y = 0$  for $y\in X, y\not=x$.

\begin{definition}[$\GG$-module]\label{def:groupoidmod}
An $X$-decomposable vector space $V=\bigoplus_{x\in X} V_x$ is said to be a \emph{left $\GG$-module} if it is equipped with, for each $x,y\in X$, a linear map $\Hom_{\GG}(x,y) \times V_x \to V_y$, denoted $(g,v) \mapsto g \cdot v$, such that
	\begin{itemize}
		\item $(gh) \cdot v = g \cdot (h \cdot v)$, for all $g,h\in \GG_1$ with $t(h)=s(g)$ and all $v\in V_{s(h)}$, and
		 
		\item $e_x \cdot v=v$, for all $x\in X$ and $v \in V_x$.
	\end{itemize}

Given two left $\GG$-modules $V=\bigoplus_{x\in X} V_x$ and $W=\bigoplus_{x\in X} W_x$, a \emph{$\GG$-module morphism} $f = (f_x)_{x \in X}: V \to W$ is an $X$-decomposable linear map 
 such that
\begin{equation}\label{eq:mormodgroupoid}
g \cdot f_{s(g)}(v)= f_{t(g)}(g\cdot v), \qquad \text{for all $g\in \GG_1$ and $v\in V_{s(g)}$}.
\end{equation}
For two $\GG$-module morphisms $f=(f_x)_{x \in X} : V \to W$ and $f'= (f'_x)_{x \in X} : W \to Z$, their {\it composition} is defined as the $\GG$-module morphism $f'f= (f'_x f_x)_{x \in X} : V \to Z$.
\end{definition}

\begin{remark}\label{rem:groupoidmod}
	Let $V=\bigoplus_{x\in X} V_x$	be a left $\GG$-module. Notice that the action is defined locally,  meaning that if $g \in \GG_1$, then $g \cdot v$ is well-defined only when $v\in V_{s(g)}$. For this reason, in the literature, the linear maps mentioned in Definition~\ref{def:groupoidmod} are called a \emph{partial groupoid action} (see e.g., \cite[Section~1]{BP2012}). Also, when $|X| = 1$ (that is, when $\GG$ is a group), Definition~\ref{def:groupoidmod} recovers the classical notion of a module over a group.
\end{remark}

The following result is an adapted version of \cite[Proposition~9.3]{IR19}.

\begin{lemma}
\label{lem.Gmodulegroupoid}
	Let $V=\bigoplus_{x\in X} V_x$	be a left $\GG$-module. Then, for each $g\in \GG_1$, the linear map $\nu_g : V_{s(g)} \to V_{t(g)}$ given by $v \mapsto g \cdot v$ is an isomorphism. In particular, $( \{V_x\}_{x\in X}, \{ \nu_g \}_{g\in \GG_1} ) $ is a groupoid.
\end{lemma}

\begin{proof}
	If $g: x\to y$ is a morphism in $\GG_1$, then $\nu_{g}^{-1}=\nu_{g^{-1}}$. Indeed, for every $v\in V_x$ we have $\nu_{g^{-1}}\nu_g(v)=g^{-1} \cdot (g \cdot v) = (g^{-1}g) \cdot v=e_x \cdot v = v$. Similarly, for any $v'\in V_y$ we have $\nu_g \nu_{g^{-1}} (v') = e_y \cdot v' = v' $. In particular, $\nu_{e_x}=\id_{V_x}$, for every $x \in X$.
\end{proof}

\begin{notation}[$\nu_g, \omega_g, \alpha_g$]
\label{not.struciso}
The maps $\{\nu_g\}_{g\in \GG_1}$ above are called the \emph{structure isomorphisms} of the $\GG$-module $V=\bigoplus_{x\in X} V_x$. In the  remainder of this section we will use the Greek letters $\{\nu_g\}_{g\in\GG_1}$, $\{\omega_g\}_{g\in\GG_1}$, $\{\alpha_g\}_{g\in\GG_1}$, etc., to denote the  structure isomorphisms of $\GG$-modules $V=\bigoplus_{x\in X} V_x$,  $W=\bigoplus_{x\in X} W_x$,  $A=\bigoplus_{x\in X} A_x$, respectively.
\end{notation}

For instance, if $f = (f_x)_{x \in X}: V \to W$ is a $\GG$-morphism, then condition~\eqref{eq:mormodgroupoid} can be restated as $\omega_g f_{s(g)} = f_{t(g)} \nu_g$, for all $g\in \GG_1$. The category of left $\GG$-modules can be endowed with a monoidal structure. The proof of the following result is straightforward.

\begin{lemma}[$\GGmod$]\label{lem:groupoid-mod}
	Let $\GGmod$ be the category of left $\GG$-modules. This category admits a monoidal structure as follows. 
	\begin{itemize}
		\item If $V=\bigoplus_{x\in X} V_x, W=\bigoplus_{x\in X} W_x \in \GGmod$, then $V \otimes_{\GGmod} W = \bigoplus_{x\in X} V_x \otimes W_x$, and the structure isomorphisms are $\{\nu_g \otimes \omega_g\}_{g\in\GG_1}$, 
		 
	\item $\unit_{\GGmod} = \bigoplus_{x\in X} \kk$ with the structure isomorphisms $\{\kappa_g = \id_{\kk}\}_{g\in\GG_1}$.
\end{itemize}
If $f = (f_x)_{x\in X}:V \to W$ and $f' = (f'_x)_{x\in X}: V' \to W'$ are two $\GG$-module morphisms, then $f \otimes_{\GGmod} f' = (f_x \otimes f'_x)_{x\in X}$. The associativity constraint is that induced by the tensor product of components, and the left/right unital constraints
	\begin{gather*}
		l_{V} : \unit_{\GGmod} \otimes_{\GGmod} V \longrightarrow V, \quad \quad \quad \quad
		r_{V}: V \otimes_{\GGmod} \unit_{\GGmod}\longrightarrow V,
	\end{gather*}
	 are given by scalar multiplication, that is, for each $x\in X$,
\[
(l_{V})_x : \kk \otimes V_x \rightarrow V_x, \; \; k \otimes v \mapsto k v \quad \quad (r_{V})_x : V_x \otimes \kk \rightarrow V_x, \; \; v \otimes k \mapsto k v.
\]
\end{lemma}

For a group $G$, Lemma~\ref{lem:groupoid-mod} implies the well-known result that $G\lmod$ is a monoidal category under the usual tensor product $\otimes_{\kk}$.

\begin{remark}
\label{GGboxtimes}
For an $X$-groupoid $\GG$ and an $X$-decomposable vector space $V = \bigoplus_{x \in X} V_x$, in the language of Remark~\ref{def:catHopflike}, we define
\[
\GG \boxtimes_{\XGrpd} V = \{(f,v) \in \GG \times V \mid f \in \GG, v \in V_{s(f)}\}.
\]
Then a left action of $\GG$ on $V$ can be viewed as a map $\GG \boxtimes_{\XGrpd} V \to V$. By the above lemma, $(\GG\lmod, \otimes_{\GG\lmod}, \kk^{|X|})$ is a monoidal category.

Recall that a \emph{subgroupoid} is a subcategory closed under taking inverses. Since the morphisms in $\XGrpd$ fix the set $X$, the subobjects of $\GG$ in $\XGrpd$ are all \emph{wide} subgroupoids $\mathcal{H}$ (that is, $\mathcal{H}_0 = X$). A wide subgroupoid is called \emph{normal} if $ghg\inv \in \mathcal{H}$ for all $h \in \mathcal{H}_1$ and all $g \in \GG_1$ such that $s(g) = t(h)$. If $\mathcal{H}$ is a normal subgroupoid, then the set of cosets of $\mathcal{H}$ in $\mathcal{G}$ forms a quotient groupoid $\mathcal{G}/\mathcal{H} \in \XGrpd$.
\end{remark}

Now we provide specific examples of modules over a groupoid.

\begin{example}\label{exa:groupoid}
	Consider the following groupoid:
	\begin{equation}\label{eq:GG}
	\GG =	\begin{tikzcd}%[sep=scriptsize]
			\bullet_x \arrow[r,bend left,"g"] \arrow["e_x"', loop, distance=2em, in=215, out=145] & \bullet_y \arrow[l,bend left,"g^{-1}"] \arrow["e_y"', loop, distance=2em, in=35, out=325]
		\end{tikzcd} 
	\end{equation}
\begin{enumerate}[label=\normalfont(\alph*)]
	\item\label{it:exgroupoid1} Let $V_x=V_y=\kk^2$ and define for all $a,b\in\kk$,
	\begin{gather*}
		g \cdot (a,b)=g^{-1} \cdot (a,b) := (b,a), \quad \quad	e_x \cdot (a,b) = e_y \cdot (a,b) := (a,b).
	\end{gather*}
Then $V=V_x \oplus V_y$ is a left $\GG$-module; $\nu_g,\nu_{g^{-1}}: \kk^2 \to \kk^2$ are given by $(a,b) \mapsto (b,a)$.

\item\label{it:exgroupoid2} Let $W_x=W_y=\kk^3$ and define for all $a,b,c\in\kk$,
\begin{gather*}
\quad \quad	\quad g \cdot (a,b,c)=g^{-1} \cdot (a,b,c) := (-a,c,b),
	\quad \quad e_x \cdot (a,b,c) = e_y \cdot (a,b,c) := (a,b,c).
\end{gather*}
Then $W=W_x \oplus W_y$ is a left $\GG$-module. Also, $\omega_g,\omega_{g^{-1}}: \kk^3 \to \kk^3$ are both given by $(a,b,c)\mapsto (-a,c,b)$. Furthermore, the maps $f_x,f_y:\kk^2 \to \kk^3$ both defined as $(a,b)\mapsto(0,a,b)$ make $f= (f_x, f_y)$ a $\GG$-module morphism from $V$ to $W$.

\item\label{it:exgroupoid3} Let $Z_x=Z_y=\kk[t,t^{-1}]$ and define for every $r\in \kk[t,t^{-1}]$,
\begin{gather*}
	g \cdot r:= tr, \qquad g^{-1} \cdot r := t^{-1}r, \qquad	e_x \cdot r = e_y \cdot r := r.
\end{gather*}
Then $Z=Z_x \oplus Z_y$ is an infinite-dimensional left $\GG$-module. The structure isomorphisms $\xi_g,\xi_{g^{-1}}: \kk[t,t^{-1}] \to \kk[t,t^{-1}]$ are given by $\xi_g(r)=tr$ and $\xi_{g^{-1}}(r)=t^{-1}r$, for every $r\in\kk[t,t^{-1}]$.

\item\label{it:exgroupoid4} Let $A_x = A_y = \kk[t, t\inv]$ and let $\sigma$ be the automorphism of $\kk[t,t\inv]$ define by $\sigma(t) = t\inv$. For $r \in \kk[t,t\inv]$, define
\[
g \cdot r = \sigma(r) = g\inv \cdot r, \quad e_x \cdot r = e_y \cdot r = r. 
\]
Then $A = A_x \oplus A_y$ is an infinite-dimensional left $\GG$-module with structure isomorphisms $\alpha_{g} = \alpha_{g\inv} = \sigma$.
\end{enumerate}
\end{example}

%%%%%%%%%%%%%%%%%

\subsection{Representations of groupoids} \label{sec:repgroupoid}

In this subsection we focus our study on representations of groupoids by adapting terminology present in \cite[Section~9.3]{IR19}, \cite[Section~3.1]{IRPaper19} and \cite[Section~2.3]{PF13}. Recall that $\GG$ denotes an arbitrary $X$-groupoid. First, we introduce a generalization of the general linear group, $\GL(V)$, over a vector space $V$.

\begin{definition}[$\GL_X(V), \GL_{(d_1,\dots,d_n)}(\kk)$]\label{def:autgroupoid1}
	Let $V=\bigoplus_{x\in X} V_x$ be an $X$-decomposable vector space. We define the $X$-groupoid $\GL_X(V)$, which we call the \emph{$X$-general linear groupoid} of $V$,	as follows: 
		\begin{itemize}
			\item the object set is $X$,
 
			\item for any $x,y\in X$, 
 $\Hom_{\GL_X(V)}(x,y)$ is the space of linear isomorphisms between the vector spaces $V_x$ and $V_y$.
		\end{itemize}
If $X = \{1, \dots, n\}$ and $V_i$ has dimension $d_i$, then we also denote $\GL_X(V)$ by $\GL_{(d_1, \dots, d_n)}(\kk)$ for $d_1 \leq d_2 \leq \dots \leq d_n$.
\end{definition}

This generalizes the classical notation $\GL_d(\kk) = \GL(V)$ when $V$ has dimension $d$. Note that, in general, this $X$-groupoid is not finite (but always has finitely many objects).

\begin{example}\label{exa:xdecomp1a}
	If $X=\{x,y\}$, then $\kk^4$ is $X$-decomposable by taking $(\kk^4)_x:=(\kk,\kk,0,0) $ and $(\kk^4)_y:=(0,0,\kk,\kk)$. Moreover, we have the following:
	\begin{equation*}
		 \GL_X(\kk^4) = \GL_{(2,2)}(\kk) =	\begin{tikzcd}%[sep=scriptsize]
			x
 \arrow[r,bend left,dashed, rightarrow] \arrow[dashed, rightarrow, loop, distance=2em, in=215, out=145] & y \arrow[l,bend left,dashed, rightarrow] \arrow[dashed, rightarrow, loop, distance=2em, in=35, out=325]
		\end{tikzcd} \cong \begin{tikzcd}%[sep=scriptsize]
			\kk^2 \arrow[r,bend left,dashed, rightarrow] \arrow[dashed, rightarrow, loop, distance=2em, in=215, out=145] & \kk^2\arrow[l,bend left,dashed, rightarrow] \arrow[dashed, rightarrow, loop, distance=2em, in=35, out=325]
		\end{tikzcd}
	\end{equation*}
	Here, the dashed arrows can be identified with the spaces 
 $\GL_2(\kk)$ of linear isomorphisms.
\end{example}

\begin{definition}[Representation of $\GG$]\label{def:repgroupoid}
A \emph{representation of $\GG$} is an $X$-decomposable vector space $V=\bigoplus_{x\in X} V_x$ equipped with a $X$-groupoid morphism $\pi: \GG \to \GL_{X}(V)$, called the \emph{representation map} of $V$. We often denote this as $(V,\pi)$. Given two $\GG$-representations $(V,\pi)$, $(W,\eta)$, a \emph{morphism of $\GG$-representations} is a $\kk$-linear natural transformation $\varphi:\pi \Rightarrow \eta$, that is, a family of linear maps $\varphi=\{ \varphi_x: V_x \to W_x \}_{x\in X}$ such that $\varphi_{t(g)}\pi(g)=\eta(g)\varphi_{s(g)}$ for every $g\in \GG_1$.
\end{definition}

The $\GG$-representations, together with their morphisms, form a category that possesses a monoidal structure.

\begin{lemma}[$\GGrep$]\label{lem:groupoid-rep}
    Let $\GGrep$ be the category of $\GG$-represen\-tations. This category admits a monoidal structure as follows. 
	\begin{itemize}
	\item If $(V,\pi),(W,\eta) \in \GGrep$, then $V \otimes_{\GGrep} W = \bigoplus_{x\in X} V_x \otimes W_x$; the representation map $\GG \to \GL_X(V \otimes_{\GGrep} W)$ is given by $g \mapsto \pi(g) \otimes \eta(g)$, for all $g\in \GG_1$,
	\item $\unit_{\GGrep} = \bigoplus_{x\in X} \kk$; the representation map $\GG \to \GL_X(\unit_{\GGrep})$ is given by $g \mapsto \id_{\kk}$, for all $g\in \GG_1$.
\end{itemize}
If $\varphi=\{ \varphi_x: V_x \to W_x \}_{x\in X}$ and $\varphi'=\{ \varphi'_x: V'_x \to W'_x \}_{x\in X}$ are morphisms of $\GG$-representations, then $\varphi \otimes_{\GGrep} \varphi' = (\varphi_x \otimes \varphi'_x)_{x\in X}$. The associativity constraint is that induced by the tensor product of components, and the left/right unital constraints
	\begin{gather*}
		l_{V} : \unit_{\GGrep} \otimes_{\GGrep} V \longrightarrow V, \quad \quad \quad \quad
		r_{V}: V \otimes_{\GGrep} \unit_{\GGrep}\longrightarrow V,
	\end{gather*}
	 are given by scalar multiplication, that is, for each $x\in X$,
\[
(l_{V})_x : \kk \otimes V_x \rightarrow V_x, \; \; k \otimes v \mapsto k v \quad \quad (r_{V})_x : V_x \otimes \kk \rightarrow V_x, \; \; v \otimes k \mapsto k v.
\]
\end{lemma}

\begin{remark}
In \cite[Definition~2]{IRPaper19}, a $\GG$-representation is defined as a functor from $\GG$ to $\Vec$; our definition is a repackaging of the information carried by such a functor. One benefit of this repackaging is that in the case that $|X| = 1$ (that is, when $\GG$ is a group), Definition~\ref{def:repgroupoid} recovers the classical notion of a group representation, i.e., a vector space $V$ equipped with a group morphism $\pi: G \to \GL(V)$.
\end{remark}

\begin{example}\label{exa:xdecomp1rep}
Let $\GG$ be as in \eqref{eq:GG}. Then the $\{x,y\}$-decomposable vector space $\kk^4$ of Example~\ref{exa:xdecomp1a} is a representation of $\GG$ by taking $\pi: \GG \to \GL_X(\kk^4)= \GL_{(2,2)}(\kk)$ as
\begin{align*}
	\pi(g): (\kk,\kk,0,0) & \longrightarrow (0,0,\kk,\kk) & \pi(g^{-1}) : (0,0,\kk,\kk) & \longrightarrow (\kk,\kk,0,0) \\
	(a,b,0,0) &\mapsto (0,0,b,a), & (0,0,a,b) & \mapsto (b,a,0,0).
\end{align*}
\end{example}

We aim to establish a correspondence between modules over groupoids and representations of groupoids. To achieve this, recall that a functor $F:\DD\to \DD'$ between two monoidal categories is called \emph{strong monoidal} if it is equipped with a natural isomorphism $F_{X, Y}:F(X)\otimes_{\DD'} F(Y)\to F(X\otimes_{\DD}Y)$, and an isomorphism $F_0:\unit_{\DD'}\to F(\unit_{\DD})$ in $\DD'$, such that associative and unital constraints are satisfied (see e.g. \cite[Definition~2.4.1]{EGNO} or \cite[Definition~3.3]{WWW} for details). Moreover, two monoidal categories are said to be \textit{monoidally isomorphic} if there exists a strong monoidal functor between them that also is an isomorphism of ordinary categories.

The following result reconciles various notions of (linear) groupoid actions in the literature and reduces to a well-known result for groups when $\GG$ has one object; see, e.g. \cite[Section~7.1.ii]{BHS2011}, \cite[Section~1]{BP2012}, \cite[Section~9.3]{IR19}, \cite[Section~3.1]{IRPaper19}, \cite[Section~2.3]{PF13}, or \cite[Section~3]{PT2014}.

\begin{lemma} \label{lem:groupoidrep}
The categories $\GGmod$ and $\GGrep$ are monoidally isomorphic.
\end{lemma}

\begin{proof}
    Consider the functor $F:\GGmod \to \GGrep$ that sends a left $\GG$-module $V=\bigoplus_{x\in X} V_x$, equipped with structure isomorphisms $\{\nu_{g}\}_{g\in\GG_1}$, to the representation $V=\bigoplus_{x\in X} V_x$ with representation map $\pi: \GG \to \GL_{X}(V)$ defined as $\pi(g):=\nu_g$ for all $g\in \GG_1$. Here, $\pi(x) = x$ holds for all $x \in X$ by Definition~\ref{def:repgroupoid}. Moreover, if $f = (f_x)_{x\in X}:V \to W$ is a $\GG$-module morphism, then $F(f)=f$. Clearly, $F$ is an isomorphism of categories.

    For any two left $\GG$-modules $V,W$ define $F_{V, W}:F(V)\otimes_{\GGrep} F(W)\to F(V\otimes_{\GGmod}W)$ as $(F_{V, W})_x = \id_{V_x\otimes W_x}$ for all $x\in X$. The collection $\{F_{V,W}\}_{V,W\in\GGmod}$ forms a natural isomorphism. Additionally, consider the morphism $F_0:\unit_{\GGrep}\to F(\unit_{\GGmod})$ in $\GGrep$ given by $(F_0)_x=\id_{\kk}: \kk \to \kk$ for all $x\in X$, which is invertible. Through straightforward verification, it can be checked that $F$ satisfies the associativity and unit constraints, proving it is a strong monoidal functor.
\end{proof}

For a group $G$, this results recovers the classical monoidal isomorphism between $G\lmod$ and $\rep(G)$, given by $V \mapsto V$ and $g \cdot v = \pi(g)(v)$, for all $g \in G$ and $v \in V$.

\begin{example}
The left $\GG$-module $\kk^2 \oplus \kk^2$ of Example~\ref{exa:groupoid}(a) corresponds to the $\GG$-representation $\kk^4 \cong \kk^2 \oplus \kk^2$ of Example~\ref{exa:xdecomp1rep} through the correspondence provided by Lemma~\ref{lem:groupoidrep}.
\end{example}

%%%%%%%%%%%%%%%%%

\subsection{Module algebras over groupoids} \label{sec:modalggroupoid}

Next, we study module algebras over groupoids.

\begin{definition}[$X$-decomposable algebra] \label{def:XdecompA}
 Let $A$ be a $\kk$-algebra. We say that $A$ is an \emph{$X$-decomposable $\kk$-algebra} if there exists a family $\{A_x\}_{x\in X}$ of unital $\kk$-algebras (some of which may be $0$) such that $A=\bigoplus_{x\in X} A_x$ as $\kk$-algebras.
 
If $A=\bigoplus_{x\in X} A_x$ and $B=\bigoplus_{x\in X} B_x$ are $X$-decomposable $\kk$-algebras, a $\kk$-linear map $f: A \to B$ is said to be a \emph{morphism of $X$-decomposable algebras} if there is a family $\{ f_x : A_x \to B_x \}_{x\in X}$ of $\kk$-algebra maps such that $f|_{A_x} = f_x$ for all $x \in X$. In this case, we write $f = (f_x)_{x \in X}$.
\end{definition}

In other words, $X$-decomposable $\kk$-algebras are simply direct sums of $|X|$ unital $\kk$-algebras with the canonical (unital) $\kk$-algebra structure. This is a stronger condition than the $\kk$-algebra $A$ being an $X$-decomposable vector space, in the sense of Definition~\ref{def:XdecompV}, since we also require that the decomposition respects the $\kk$-algebra structure of $A$.

\begin{remark}\label{rem:localid}
Let $A=\bigoplus_{x\in X} A_x$ be an $X$-decomposable $\kk$-algebra.
\begin{enumerate}[label=\normalfont(\alph*)]
    \item If for each $x\in X$ we denote the multiplicative identity element of $A_x$ by $1_x:=1_{A_x}$, then $1_A=\sum_{x\in X} 1_x$. We refer to the elements of the set $\{1_x \mid x \in X \text{ and $1_x \neq 0$}\}$ as the {\it local identities} of $A$. The local identities form a complete set of orthogonal central idempotents of $A$.
    \item A morphism of $X$-decomposable algebras $ f: A \to B $ is simply a $\kk$-algebra homomorphism that respects the components, meaning $ f(A_x) \subseteq B_x $ for each $ x \in X $. This also implies that $ f $ preserves the local identities, i.e., $ f(1_{A_x}) = 1_{B_x} $ for all $ x \in X $.
\end{enumerate}
\end{remark}

However, the local identities could be a sum of nonzero central orthogonal idempotents (i.e., non-primitive), as the next example shows. Thus, an $X$-decomposition of a $\kk$-algebra is not unique.

\begin{example}\label{exa:xdecomp1b}
	The $\{x,y\}$-decomposable vector space $\kk^4$ of Example~\ref{exa:xdecomp1a} is, in fact, a $\{x,y\}$-decomposable $\kk$-algebra when considering component-wise operations. Here, the local identities are $1_x=(1,1,0,0)$ and $1_y=(0,0,1,1)$, which are not primitive.
\end{example}

We point out an example of an $X$-decomposable $\kk$-algebra which is used in \cite{CF2009} for their classification of weak Hopf algebra structures derived from $U_q(sl_2)$. 

\begin{example}\label{exa:xdecomp3}
Take $q\in \kk \backslash \{ 0, \pm 1 \}$. We follow \cite{CF2009} for the following construction. The \emph{weak quantum group $A:=wsl_q2$} is the unital $\kk$-algebra generated by indeterminates $E,F,K,\bar{K}$ subject to the following relations:
\begin{gather*}
	K E = q^2E K, \qquad \bar{K} E =q^{-2} E \bar{K}, \qquad K F =q^{-2} F K, \qquad \bar{K} F = q^2 F \bar{K},\\
	K\bar{K} = \bar{K} K, \qquad K \bar{K} K =K, \qquad \bar{K} K \bar{K} = \bar{K}, \qquad E F - F E = (K -\bar{K})/(q-q^{-1}).
		\end{gather*}
	For $X=\{x,y\}$, let $A_x=A 1_x$, $A_y=A 1_y$, with $1_x=K \bar{K}$, and $1_y=1_A-K \bar{K}$. Thus, $A = A_x \oplus A_y$ is an $X$-decomposable $\kk$-algebra; in fact, $A_x \cong U_q(\mf{sl}_2)$ and $A_y \cong \kk[t_1,t_2]$ as $\kk$-algebras \cite[Theorems~2.3, 2.5, 2.7]{CF2009}.
\end{example}

Now, we define two remarkable $X$-decomposable linear maps that arise when an $X$-decomposable $\kk$-algebra has the structure of module over an $X$-groupoid $\GG$.

\begin{remark}\label{rem:GGModuleMaps}
If $A=\bigoplus_{x\in X} A_x$ is an $X$-decomposable $\kk$-algebra, the multiplication map $m_A: A \otimes A \to A$ and unit map $u_A : \kk \to A$ immediately decompose into the respective multiplication map $m_x: A_x \otimes A_x \to A_x$ and unit map $u_x: \kk \to A_x$ of each $\kk$-algebra $A_x$, for all $x\in X$. If additionally $A$ is a $\GG$-module then, using the notation of Lemma~\ref{lem:groupoid-mod}, $m_A$ and $u_A$ induce $X$-decomposable linear maps $\underline{m}_A:=(m_x)_{x\in X} : A \otimes_{\GGmod} A \to A$ and $\underline{u}_A:=(u_x)_{x\in X} : \unit_{\GGmod} \to A$, which we call the \emph{monoidal multiplication} and \emph{monoidal unit} of $A$, respectively. It is clear that these maps satisfy associativity and unital condition. However, in general, these maps are not necessarily $\GG$-module morphisms.
\end{remark} 

By definition, it follows immediately that the monoidal multiplication and monoidal unit maps are $\GG$-module morphisms precisely when $A$ is an algebra in $\GG\lmod$ (i.e., $A$ is a $\GG$-module algebra).
Conversely, if $A$ is $\GG$-module algebra, then it comes equipped with maps in $\GGmod$, $\underline{m}_A = (m_x)_{x \in X} A \otimes_{\GGmod} A \to A$ and $\underline{u}_A = (u_x)_{x \in X}: \unit_{\GGmod} \to A$. These morphisms extend naturally to linear maps $m_A: A \otimes A \to A$ (where if $a \in A_x$ and $b \in A_y$ for $x \neq y$, we define $m_A(a \otimes b) = 0$) and $u_A : \kk \to A$ (defined by $\kk\to \unit_{\GGmod} \to A$ where $\kk \to \unit_{\GGmod}$ maps $1_{\kk}$ to $(1,1, \ldots, 1)$). Hence, we have proved the following result.

\begin{lemma}\label{lem:GGModuleFormulas}
Let $A=\bigoplus_{x\in X} A_x$ be an $X$-decomposable $\kk$-algebra and let $\GG$ be an $X$-groupoid.
 Then the following statements are equivalent.
	\begin{enumerate}[label=\normalfont(\alph*)]
		\item $A$ is a $\GG$-module algebra via the monoidal product $\underline{m}_A: A \otimes_{\GGmod} A \to A$ and monoidal unit $\underline{u}_A: \unit_{\GGmod} \to A$ of Remark~\ref{rem:GGModuleMaps}.
		 
		\item $A$ is a $\GG$-module such that
		\begin{gather}
			g \cdot (ab)=(g \cdot a)(g \cdot b),\label{eq:groupoidmodalg1}\\
			g \cdot 1_{s(g)} = 1_{t(g)},\label{eq:groupoidmodalg2}
		\end{gather}
		for all $g\in \GG_1$ and $a,b\in A_{s(g)}$.
	\end{enumerate}
\end{lemma}

By abuse of notation, we also refer to an $X$-decomposable $\kk$-algebra satisfying the conditions in Lemma~\ref{lem:GGModuleFormulas}(b) as a \emph{$\GG$-module algebra}. Next, we generalize the algebra automorphism group, $\AutA$, of a $\kk$-algebra $A$.

\begin{definition}[$\XAutA$]\label{def:autgroupoid2}
 Let $A=\bigoplus_{x\in X} A_x$ be an $X$-decomposable $\kk$-algebra. We define $\XAutA$, the \emph{$X$-algebra automorphism groupoid} of $A$, as follows: 
\begin{itemize}
 \item the object set is $X$,
 \item for any $x,y\in X$, $\Hom_{\XAutA}(x,y)$ %\sout{$\Hom_{\XAutA}(A_x,A_y)$} 
 is the space of unital $\kk$-algebra isomorphisms between the unital $\kk$-algebras $A_x$ and $A_y$. The composition of morphisms is determined by the composition of the corresponding $\kk$-algebra morphisms.
\end{itemize}
\end{definition}

\begin{remark}\label{rem:autxal}
The $X$-algebra automorphism groupoid $\XAutA$ is an $X$-groupoid. Additionally, $\XAutA$ is a subgroupoid of $\GL_X(A)$  (see Definition~\ref{def:autgroupoid1}), and it is usually proper in the sense that the morphisms of $\XAutA$ are a proper subset of the morphisms of $\GL_X(A)$. 
In the case $|X| = 1$, this reduces to the statement that $\AutA$ is a proper subgroup of $\GL(A)$ (which is simply the statement that the $\kk$-algebra automorphisms of $A$ form a proper subset of the $\kk$-vector space automorphisms of the vector space $A$).
\end{remark}

Now, we give a generalization of the equivalence between the categorical and representation theoretic notions of action on $\kk$-algebras for groups. In other words, we explicitly calculate the symmetry object $\Sym_{\XGrpd}(A)$ over any $X$-decomposable $\kk$-algebra $A$ in the category $\XGrpd$ (recall Definition~\ref{def:Sym}).

\begin{proposition}\label{prop:groupoidalg}
Let $A = \bigoplus_{x\in X} A_x$ be an $X$-decomposable $\kk$-algebra and let $\GG$ be an $X$-groupoid. Then:
\begin{enumerate}[label=\normalfont(\alph*)]
 \item $A$ is an $\XAutA$-module algebra via $\ell_{\XAutA, A}(f \boxtimes a) = f(a)$ for all $f \boxtimes a \in \XAutA \boxtimes_{\XGrpd} A$ (that is, for all $f\in \XAutA_1$ and $a\in A_{s(f)}$). We denote $\ell_{\XAutA, A}(f \boxtimes a)$ by $f \rt a$.

\item Suppose that $A$ is a $\GG$-module algebra via $\ell_{\GG, A}$, and denote $g \cdot a := \ell_{\GG,A}(g \boxtimes a)$ for all $g \boxtimes a \in \GG \boxtimes_{\XGrpd} A$. Then there is a unique $X$-groupoid morphism $\pi: \GG \to \XAutA$ such that $g \cdot a = \pi(g) \rt a$ for all $g \boxtimes a \in \GG \boxtimes_{\XGrpd} A$.

\item Every $X$-groupoid morphism $\pi: \GG \to \XAutA$ gives $A$ the structure of a $\GG$-module algebra via $\ell_{\GG,A}(g \boxtimes a) = \ell_{\XAutA, A}(\pi(g) \boxtimes a)$ for all $g \boxtimes a \in \GG \boxtimes_{\XGrpd} A$. 
 \end{enumerate}
Hence $\Sym_{\XGrpd}(A) = \XAutA$.
\end{proposition}

\begin{proof}
(a) For all $f\in \XAutA_1$ and $a,b\in A_{s(f)}$ we have
 \begin{gather*}
 f \rt (ab) = f(ab) = f(a) f(b) = (f \rt a)(f \rt b),\\ \quad 
 f \rt 1_{s(f)} = f(1_{s(f)})=1_{t(f)}
 \end{gather*}
 so by Lemma~\ref{lem:GGModuleFormulas} it follows that $A$ is an $\XAutA$-module algebra.

 (b) If $A$ is a $\GG$-module via $\cdot$, then by Lemma~\ref{lem.Gmodulegroupoid}, there is an associated subgroupoid of $\GL_X(A)$ with object set $X$ and its structure isomorphisms $\alpha_g$ (see Notation~\ref{not.struciso}). This induces a unique $X$-groupoid morphism $\pi: \GG \to \GL_X(A)$ such that and $\pi(g) = \alpha_g$ for all $x \in X$ and $g\in\GG_1$ (see Lemma~\ref{lem:groupoidrep}). Moreover, if $A$ is a $\GG$-module algebra, then it satisfies \eqref{eq:groupoidmodalg1} and \eqref{eq:groupoidmodalg2}, which are equivalent to map $\pi(g)$ being a unital $\kk$-algebra map for all $g \in G$. Hence, the image of the groupoid morphism $\pi$ is contained in the subgroupoid $\XAutA$ of $\GL_X(A)$. Corestricting $\pi$ to a map $\GG \to \XAutA$ gives the result.

(c) If $\pi: \GG \to \XAutA$ is any $X$-groupoid morphism, then we can define a map $\ell_{\GG, A}: \GG \boxtimes_{\XGrpd} A \to A$ by $\ell_{\GG,A}(g \boxtimes a) = \ell_{\XAutA, A}(\pi(g), a)$ for all $g \boxtimes a \in \GG \boxtimes_{\XGrpd} A$. Since $\pi$ is an $X$-groupoid morphism, we see that if $g, h \in \GG_1$ with $t(h) = s(g)$ and $a \in A_{s(h)}$, then
\[
(gh)\cdot a = \pi(gh) \rt a = [\pi(g)\pi(h)] \rt a = \pi(g) \rt (\pi(h) \rt a) = g \cdot (h \cdot a),
\]
and for all $a \in A_x$
\[
e_x \cdot a = \pi(e_x) \rt a = a.
\]
Hence, $\cdot$ makes $A$ a $\GG$-module. Further for all $g \in \GG_1$ and $a,b \in A_{s(g)}$, we have
\[
g \cdot (ab) = \pi(g) \rt (ab) = (\pi(g) \rt a)(\pi(g) \rt b) = (g \cdot a)(g \cdot b)
\]
and
\[
g \cdot 1_{s(g)} = \pi(g) \rt 1_{s(g)} = 1_{t(\pi(g))} = 1_{t(g)}
\]
and so $A$ is a $\GG$-module algebra.
\end{proof}

\begin{example}
	In Example~\ref{exa:groupoid}\ref{it:exgroupoid1}, the $\GG$-module $\kk^4 \cong \kk^2 \oplus \kk^2$ is a $\GG$-module algebra with functor $\pi: \GG \to \Aut_{\XAlg}(\kk^4)$ induced from $\pi$ in Example~\ref{exa:xdecomp1rep}.
In Example~\ref{exa:groupoid}\ref{it:exgroupoid4}, the $\GG$-module $A = \kk[t, t\inv] \oplus \kk[t, t\inv]$ is a $\GG$-module algebra with the functor $\pi: \GG \to \GGAutA$.
	However, the $\GG$-module in Example~\ref{exa:groupoid}\ref{it:exgroupoid2} is not an example of a $\GG$-module algebra since the structure isomorphisms are not unital. Similarly, the induced $X$-decomposition of the $\GG$-module of Example~\ref{exa:groupoid}\ref{it:exgroupoid3} does not make it into an $\GG$-module algebra.
\end{example}

We end this section with a well-known result (see \cite[Proposition~1.2]{CM1984}), which is the reinterpretation of Proposition~\ref{prop:groupoidalg} in the case that $|X| = 1$ (that is, when $\GG$ is a group).

\begin{remark} \label{rem:group}
Let $A$ be a $\kk$-algebra. Then $A$ is an $\AutA$-module algebra and for a group $G$, the following are equivalent.
\begin{enumerate}[label=\normalfont(\alph*)]
 \item $A$ is a $G$-module algebra.
 \item There exists a group morphism $\pi: G \to \Aut_{\Algcat}(A)$. 
\end{enumerate}
Hence, $\Sym_{\Group}(A) = \AutA$.
\end{remark}

%%%%%%%%%%%%%%%%%
%%%%%%%%%%%%%%%%%
%%%%%%%%%%%%%%%%%
\section{Actions of groupoid algebras on algebras}
\label{sec:groupoidalg}

It is well-known that modules over a group correspond to modules over its group algebra. In this section, we extend this correspondence to the groupoid setting (see Theorem~\ref{thm:repkGrp}). This allows us to obtain results about module algebras over groupoid algebras, as given in Proposition~\ref{prop:kGrpd}. Along the way, we also demonstrate that finite groupoid actions on domains factor through group actions (see Corollary~\ref{cor:grpoid-innerfaith}), highlighting that the study of groupoid actions represents a natural generalization of group actions on domains.

%%%%%%%%%%%%%%%%%%%%%%%%%%%

\subsection{Weak Hopf algebras}
\label{sec:weakbk}

While a group algebra is a Hopf algebra, a groupoid algebra is a \emph{weak} Hopf algebra. Therefore, we start by providing the definition of a weak bialgebra and recalling some fundamental properties. For more details, we refer the reader to \cite{BNS,NV}.

\begin{definition}[Weak bialgebra] \label{def:wba} 
A \textit{weak bialgebra} is a quintuple $(H,m,u,\Delta, \varepsilon)$ such that
\begin{enumerate}[label=\normalfont(\alph*)]
 \item $(H,m,u)$ is a $\kk$-algebra,
 
 \item $(H, \Delta, \varepsilon)$ is a $\kk$-coalgebra, 
 \item \label{def:wba3} $\Delta(ab)=\Delta(a)\Delta(b)$ for all $a,b \in H$, 
 \item \label{def:wba4} $\varepsilon(abc)=\varepsilon(ab_1)\varepsilon(b_2c)=\varepsilon(ab_2)\varepsilon(b_1c)$ for all $a,b,c \in H$, 
 \item \label{def:wba5} $\Delta^2(1_H)=(\Delta(1_H) \otimes 1_H)(1_H \otimes \Delta(1_H))=(1_H \otimes \Delta(1_H))(\Delta(1_H) \otimes 1_H)$. 
\end{enumerate}
Here, we use \emph{sumless Sweedler notation}, which means that we write $\Delta(h):= h_1 \otimes h_2$ for any $h \in H$. A \emph{weak bialgebra morphism} between two weak bialgebras is a linear map that is both a $\kk$-algebra and a $\kk$-coalgebra map.
\end{definition}

\begin{definition}[$\varepsilon_s$, $\varepsilon_t$, $H_s$, $H_t$] \label{def:eps}
Let $(H, m, u, \Delta, \varepsilon)$ be a weak bialgebra. We define the {\it source and target counital maps}, respectively as follows:
\[
\begin{array}{ll}
 \varepsilon_s: H \longrightarrow H, & x \mapsto 1_1\;\varepsilon(x1_2) \\
 \varepsilon_t: H \longrightarrow H, & x \mapsto \varepsilon(1_1x)\;1_2.
\end{array}
\]
We denote the images of these maps by $H_s:=\varepsilon_s(H)$ and $H_t:=\varepsilon_t(H)$, which are called the \emph{source} and \emph{target counital subalgebras} of $H$, respectively.
\end{definition}

\begin{remark}[{\cite{BNS,BCJ}}]
\label{rem:HsHt}
Since the counital subalgebras are separable and Frobenius, they are finite-dimensional. Moreover, $H_t^{\text{op}} \cong H_s$ as $\kk$-algebras.
\end{remark}

\begin{definition}[Weak Hopf algebra] \label{def:weak}
A \textit{weak Hopf algebra} is a sextuple $(H,m,u,\Delta,\varepsilon, S)$, where the quintuple $(H,m,u,\Delta,\varepsilon)$ is a weak bialgebra and $S: H \to H$ is a linear map, called the \textit{antipode}, that satisfies the following properties for all $h\in H$:
 $$S(h_1)h_2=\varepsilon_s(h), \quad \quad
h_1S(h_2)=\varepsilon_t(h), \quad \quad
S(h_1)h_2S(h_3)=S(h).$$
A {\it weak Hopf algebra morphism} between two weak Hopf algebras is a weak bialgebra map $f:H \to H'$ such that $S_{H'} \circ f = f \circ S_H$.

A \emph{(weak) Hopf ideal} $I$ of a weak Hopf algebra $H$ is a biideal of $H$ such that $S(I) \subseteq I$.
\end{definition}

\begin{remark}[\cite{BNS, NV}] \label{rem:antipodeprops}
It follows from \cref{def:weak} that $S$ is anti-multiplicative with respect to $m$, and anti-comultiplicative with respect to $\Delta$. Moreover, the following conditions are equivalent: $\Delta(1_H)=1_H\otimes 1_H$;\; $\ep(xy)=\ep(x)\ep(y)$ for all $x,y\in H$;\; $S(x_1)x_2=\ep(x)1_H$ for all $x \in H$; \; $x_1S(x_2)=\ep(x)1_H$ for all $x \in H$.
In this case, $H$ is a Hopf algebra.
\end{remark}

\begin{remark}
\label{Hmodboxtimes}
Since a weak Hopf algebra $H$ is a $\kk$-algebra, it has a category of $H$-modules, $H\lmod$, whose objects are vector spaces $V$ endowed with an action map $H \otimes_{\kk} V \to V$ in the usual sense (that is, $(hk)\cdot v = h \cdot (k \cdot v)$ and $1_H \cdot v = v$ for all $h,k\in H$ and $v \in V$). Hence, in the language of Remark~\ref{def:catHopflike}, if $\CC$ is any of the categories of weak Hopf algebras in Notation~\ref{not:cats}, we may take $\boxtimes_{\CC}$ to be $\otimes_{\kk}$.

The category $H\lmod$ is monoidal, where if $V$ and $W$ are $H$-modules then
\[
V \otimes_{H \lmod} W:=\Delta(1_H)\cdot(V \otimes W) \subseteq V \otimes_{\kk} W
\]
and $\unit_{H\lmod} = H_t$ \cite[Section~2.1]{BSK}.
\end{remark}

Since, for a weak Hopf algebra $H$, $H\lmod$ is a monoidal category, we may define an algebra object $A$ in $H\lmod$. However, as mentioned in Remark~\ref{rem:Hmodulealgebra}, since the monoidal product $\otimes_{H\lmod}$ is not $\otimes_{\kk}$, it is not immediately obvious that such an algebra object has a $\kk$-algebra structure. Nevertheless, by work of the first author and Reyes \cite{CR}, it is possible to construct an isomorphism of categories between the algebra objects in $H\lmod$ and a certain category consisting of $\kk$-algebras that are $H$-modules.

\begin{proposition}[\cite{CR}, see also \cite{WWW}] \label{def:modulealgoverWHA}
Let $H$ be a weak Hopf algebra. Then the following hold.

\begin{enumerate}[label=\normalfont(\alph*)]
\item If $(A, \underline{m}_A: A \otimes_{H \lmod} A \to A, \underline{u}_A: H_t = \unit_{H \lmod} \to A)$ is an algebra object in $H\lmod$, then $A$ can be given the structure of a $\kk$-algebra by defining $m_A: A \otimes_{\kk} A \to A$ and $u_A: \kk \to A$ by $m_A(x \otimes y) = \underline{m}_A(1_1 \cdot x \otimes_{H \lmod} 1_2 \cdot y)$ and $u_A(1_{\kk}) = \underline{u}_A (1_{H_t})$. This $H$-action on the $\kk$-algebra $A$ satisfies
\begin{equation}
\label{eq.Hmodalgebra}
h\cdot ab=(h_1\cdot a)(h_2\cdot b) \quad \text{and} \quad h\cdot 1_A=\varepsilon_t(h)\cdot 1_A \quad \text{for all $a,b \in A$ and $h \in H$.}
\end{equation}

\item Conversely, if a $\kk$-algebra $(A, m_A: A \otimes_{\kk} A \to A, u_A: \kk \to A)$ is an $H$-module which satisfies equation \eqref{eq.Hmodalgebra}, then the $A$ can be viewed as an algebra object in $H\lmod$ $(A, \underline{m}_A: A \otimes_{H \lmod} A \to A, \underline{u}_A: H_t = \unit_{H \lmod} \to A)$ by defining $\underline{m}_A(x \otimes_{H\lmod} y) = (1_1 \cdot x)(1_2 \cdot y)$ and $\underline{u}_A(h) = h \cdot 1_A$.
 
\end{enumerate}
\end{proposition}
Hence, if $A$ is a $\kk$-algebra and an $H$-module which satisfies equation~\eqref{eq.Hmodalgebra}, then we call $A$ an \emph{$H$-module algebra}. The proof of the following proposition is straightforward.

\begin{proposition} \label{prop:weak-direct} 
Let $\{H_x\}_{x \in X}$ be a finite collection of a weak Hopf algebras.
\begin{enumerate}[label=\normalfont(\alph*)]
\item Then $H = \bigoplus_{x \in X} H_x$ is a weak Hopf algebra. The algebra operations and antipode are defined component-wise, while the coalgebra structure is given by the sum of the component-wise operations.
Moreover, $H_s = \bigoplus_{x \in X} (H_x)_s$ and $H_t = \bigoplus_{x \in X} (H_x)_t$.
 
\item If $A_x$ is a $H_x$-module algebra for each $x \in X$, then $\bigoplus_{x\in X} A_x$ is an $H$-module algebra with $H$-action defined component-wise.   
\end{enumerate}
\end{proposition}

%%%%%%%%%%%%%%%%%%%

\subsection{Modules and representations over groupoid algebras}
\label{sec:grpdrep}

Now, we shift our focus to groupoid algebras and explore the relationship between their modules and representations and those of the base groupoid. As before, throughout this section, $\GG$ denotes an arbitrary $X$-groupoid.

\begin{definition}[$\kk\GG$] \label{def:grpdalg}
The \emph{groupoid algebra} $\kk\GG$ is defined as the $\kk$-algebra with a vector space basis given by $\GG_1$, and the product of two morphisms being their composition if defined and $0$ otherwise, extended linearly. Here, $1_{\kk\GG} = \sum_{x\in X} e_x$. Note that $\kk \GG$ is a weak Hopf algebra with $\Delta(g) = g \otimes g$, $\ep(g) = 1_\kk$, and $S(g) = g^{-1}$, for all $g \in \GG_1$. Here, $(\kk \GG)_s = (\kk \GG)_t = \bigoplus_{x \in X} \kk e_x \; \cong \; \kk^{|X|}$.
\end{definition}

For a group $G$, the previous construction coincides with the group algebra $\kk G$ and its Hopf algebra structure.

Consider the category $\kk\GG\lmod$ of \emph{left $\kk\GG$-modules}. Since $\kk\GG$ is a weak Hopf algebra, it follows from Remark~\ref{Hmodboxtimes} that $\kk\GG\lmod$ is a monoidal category. The next result showcases the connection between $\kk\GG\lmod$ and the previously introduced monoidal category $\GGmod$ of left $\GG$-modules (recall Lemma~\ref{lem:groupoid-mod}).

\begin{lemma} \label{lem:kGGmod}
The categories $\GGmod$ and $\kk\GG\lmod$ are monoidally isomorphic.
\end{lemma}

\begin{proof}
    Consider the functor $F:\kk\GG\lmod \to \GGmod$ that sends a left $\kk\GG$-module $V$, with action $\cdot$, to the left $\GG$-module $V=\bigoplus_{x\in X} V_x$, where $V_x = e_x \cdot V$ for all $x \in X$. The $\GG$-action is a restriction of the $\kk\GG$-action and can be denoted similarly as $\cdot$. Moreover, if $f: V \to W$ is a $\kk\GG$-module morphism, then $F(f)=(f_x)_{x\in X}$, where $f_x: V_x \to W_x$ is defined by $f_x(e_x\cdot v)=e_x\cdot f(v)$ for all $x\in X$. Then $F$ is an isomorphism of categories, with its inverse obtained by linearizing the $\GG$-actions.

    From Definition~\ref{def:grpdalg}, $\Delta(1_{\kk\GG})=\sum_{x\in X} e_x \otimes e_x$. Therefore, if $V,W$ are left $\kk\GG$-modules, we have $V \otimes_{\kk\GG\lmod} W= \bigoplus_{x\in X} V_x \otimes W_x$. Using the monoidal structure of the category $\GGmod$ (Lemma~\ref{lem:groupoid-mod}), define $F_{V, W}:F(V)\otimes_{\GGmod} F(W)\to F(V\otimes_{\kk\GG\lmod}W)$ as $(F_{V, W})_x = \id_{V_x\otimes W_x}$ for all $x\in X$. Then $\{F_{V,W}\}_{V,W\in\kk\GG\lmod}$ forms a natural isomorphism. Also, consider the morphism $F_0:\unit_{\GGmod}\to F(\unit_{\kk\GG\lmod})$ in $\GGmod$ given by $(F_0)_x=\id_{\kk}: \kk \to \kk$ for all $x\in X$, which is clearly invertible. Straightforward verification confirms that $F$ satisfies the associativity and unit constraints, proving it is a strong monoidal functor.
\end{proof}

For a group $G$, this result recovers the classical monoidal isomorphism between $\kk G\lmod$ and $G\lmod$, which is obtained by restricting the group algebra action to a group action.

For a vector space $V$, we let $\End(V) = \Hom_{\Vec}(V, V)$ denote the usual endomorphism ring of $V$. If $V = \bigoplus_{x \in X} V_x$ is $X$-decomposable with each $V_x \neq 0$, then for each $x \in X$, we write $\pi_x: V \to V_x$ to denote the canonical projection and $\iota_x: V_x \to V$ to denote the canonical inclusion. Note that $\iota_x \pi_x$ is an idempotent of $\End(V)$ whose restriction to $V_x$ is $\id_{V_x}$.

\begin{definition}[Representation of $\kk\GG$]\label{def:kGGmod}
A \emph{representation of $\kk\GG$} is an $X$-decomposable vector space $V=\bigoplus_{x\in X} V_x$ equipped with a $\kk$-algebra morphism $\widetilde{\pi}: \kk \GG \rightarrow \End(V)$, such that $\widetilde{\pi}(e_x) = \iota_x \pi_x$, for all $x \in X$, called the \emph{representation map} of $V$. We often denote this as $(V,\widetilde{\pi})$. A morphism between $\kk\GG$-representations $(V, \widetilde{\pi})$ and $(W, \widetilde{\eta})$ is a linear map $\varphi: V \to W$ with $\varphi(V_x) \subset W_x$ for all $x \in X$, such that $\varphi \circ \widetilde{\pi}(g) = \widetilde{\eta}(g) \circ \varphi$ for $g \in \GG_1$.
\end{definition}

In a $\kk\GG$-representation $(V,\widetilde{\pi})$ we have $\widetilde{\pi}(g)=\widetilde{\pi}(ge_{s(g)})=\widetilde{\pi}(g)\iota_{s(g)}\pi_{s(g)}$ for every $g\in\GG_1$. Similarly, $\widetilde{\pi}(g)=\widetilde{\pi}(e_{t(g)}g)=\iota_{t(g)}\pi_{t(g)}\widetilde{\pi}(g)$. This implies that $\widetilde{\pi}(g)(v)\in V_{t(g)}$ if $v \in V_{s(g)}$, and the result is zero otherwise. Hence, we can consider $\widetilde{\pi}(g)$ as a map from $V_{s(g)}$ to $V_{t(g)}$.

The $\kk\GG$-representations, together with their morphisms, form a category that possesses a monoidal structure.

\begin{lemma}[$\rep(\kk\GG)$]\label{lem:groupoidalg-rep}
    Let $\rep(\kk\GG)$ be the category of $\kk\GG$-representations. This category admits a monoidal structure as follows.
	\begin{itemize}
	\item If $(V,\widetilde{\pi}),(W,\widetilde{\eta}) \in \rep(\kk\GG)$, then $V \otimes_{\rep(\kk\GG)} W = \bigoplus_{x\in X} V_x \otimes W_x$; the representation map $\kk\GG \to \End(V \otimes_{\rep(\kk\GG)} W)$ is given by the linearization of $g \mapsto \widetilde{\pi}(g) \otimes \widetilde{\eta}(g)$, for all $g\in \GG_1$,
	\item $\unit_{\rep(\kk\GG)} = \bigoplus_{x\in X} \kk$; the representation map $\kk\GG \to \End(\unit_{\rep(\kk\GG)})$ is given by the linearization of $g \mapsto \id_{\kk}$, for all $g\in \GG_1$.
\end{itemize}
If $\varphi:V \to W$ and $\varphi':V'\to W'$ are morphisms of $\kk\GG$-representations, then $\varphi \otimes_{\rep(\kk\GG)} \varphi' = \bigoplus_{x\in X}(\varphi|_{V_x} \otimes \varphi'|_{V'_x})$. The associativity constraint is that induced by the tensor product of components, and the left/right unital constraints are given by scalar multiplication.
\end{lemma}

In Corollary~\ref{cor:repiso} we describe the connection between $\rep(\kk\GG)$ and $\rep(\GG)$.

%%%%%%%%%%%%%%%%%%%

\subsection{Local units of algebras} \label{sec:locunit}

Next, we introduce the notion of local units of a $\kk$-algebra. This will serve the purpose of generalizing the group of units functor to the groupoid setting; see Theorem~\ref{thm:repkGrp}(a) below.

\begin{definition}[$\widehat{a}$] Let $A$ be a $\kk$-algebra and $\{ e_x \}_{x\in X}$ be a set of nonzero orthogonal idempotents of $A$. For $x,y \in X$, an element $a \in e_y A e_x$ is called a \emph{local unit} if there exists an element $b \in e_x A e_y$ such that $ab = e_y$ and $ba = e_x$. The element $b$ is called a {\it local inverse} of $a$, which we denote by $\widehat{a} := b$.
\end{definition}

The notions above are well-defined due to the following straightforward lemma.

\begin{lemma}\label{lem:localunits}
	Let $A$ be a $\kk$-algebra and let $\{ e_x \}_{x\in X}$ be a set of nonzero orthogonal idempotents of $A$. Suppose that $a \in e_y A e_x$ is a local unit of $A$ for some $x,y \in X$. Then:
	\begin{enumerate}[label=\normalfont(\alph*)]
	\item $a$ is nonzero. 
	 
	\item If $a \in e_{y'} A e_{x'}$ for $x',y'\in X$, then $x=x'$ and $y=y'$. 
	 
	\item If there exist $b, b' \in e_x A e_y$ such that $ab = ab'=e_y$ and $ba =b'a=e_x$, then $b=b'$.
	 
 \item For $x,y,z\in X$, if $a_1\in e_y A e_x$ and $a_2 \in e_z A e_y$ are local units of $A$, then the product $a_2a_1 \in e_z A e_x$ is also a local unit of $A$.   
	\end{enumerate}
\end{lemma}

\begin{remark}\label{rem:localunits}
	In contrast with the notion of local identities introduced in Remark~\ref{rem:localid} for $X$-decomposable $\kk$-algebras, the idempotents in Lemma~\ref{lem:localunits} are not required to be central, nor do we require the condition $1_A=\sum_{x\in X} e_x$. This allows us to work with a broader scope of $\kk$-algebras towards our main result, Theorem~\ref{thm:repkGrp}. For example, for an $X$-groupoid $\GG$, the groupoid algebra $\kk \GG$ has the set of nonzero orthogonal idempotents $\{e_x\}_{x\in X}$, and even though $1_{\kk\GG} = \sum_{x\in X} e_x$, the idempotents are $e_x$ are not central, in general. Hence, unless $\GG$ is totally disconnected, $\kk \GG$ is not a $X$-decomposable $\kk$-algebra via these elements.
\end{remark}

\begin{definition}[$(-)^{\times}_X$]\label{def:groupoidL}
 Let $A$ be a $\kk$-algebra and $\{ e_x \}_{x\in X}$ be a set of nonzero orthogonal idempotents of $A$. The \emph{groupoid of local units} of $A$, denoted by $A^{\times}_X$, is defined as follows.
 \begin{itemize}
 \item The objects are the elements $x\in X$.
 \item For each $x,y\in X$, $\Hom_{A^{\times}_X}(x,y) = \{a \in e_y A e_x \mid \text{$a$ is a local unit of $A$} \}$. In this case we write $a: x \to y$.
 \item Given two morphisms $a: x\to y$ and $a': y \to z$, we define their composition as the product $a'a: x\to z$.
 \item The inverse of a morphism $a: x \to y$ is given by its local inverse $\widehat{a}: y \to x$.
 \end{itemize}
\end{definition}

\begin{remark}\label{rem:groupoidL}
\begin{enumerate}[label=\normalfont(\alph*)]
	\item Lemma~\ref{lem:localunits} guarantees that the construction above is well-defined and indeed yields an $X$-groupoid. 
    \item For a $\kk$-algebra $A$, the construction of $A^{\times}_X$ strongly depends on the chosen family of idempotents $\{e_x \mid x \in X\}$. As a result, $A$ may have multiple different associated groupoids of local units. It would be more natural, notationally, for $X$ to denote the set of idempotents rather than an index set. However, in this paper, we generally begin by choosing an idempotent decomposition of $1 = \sum_{x \in X} e_x$, which then determines our index set $X$. In fact, we frequently use $X$ interchangeably to refer both to the set of idempotents and to the index set. Since different choices of idempotents are never assigned the same index set, this should not cause any confusion.
 \end{enumerate}
\end{remark}

It would be interesting to investigate how the groupoid of local units varies with different choices of idempotent decompositions, though we do not explore this here. We do, however, provide an example below that demonstrates two different choices of idempotent decomposition.

\begin{example}\label{exa:localunits}
	\begin{enumerate}[label=\normalfont(\alph*)]
 \item Let $A = \kk[x] \oplus \kk[y] \oplus \kk[z^{\pm 1}]$, with identity elements in each component denoted by $e_1$, $e_2$, and $e_3$, respectively.  Consider the two distinct sets of orthogonal idempotents $X = \{e_1 + e_2, e_3\}$ and $Y = \{e_1, e_2 + e_3\}$ in $A$. Both choices yield groupoids of local units, $A_X^\times$ and $A_Y^\times$, each with two objects and no morphisms between different objects. However, the groups of morphisms at the objects differ between the two cases. In $A_X^\times$, the groups of morphisms are $\kk^\times \times \kk^\times$ and $\kk[z^{\pm 1}]$, whereas in $A_Y^\times$, they are $\kk^\times$ and $\kk^\times \times \kk[z^{\pm 1}]$.
 
\item If $|X|=1$, then for any $\kk$-algebra $A$ we can take $\{1_A\}$ as the set of idempotents indexed by $X$. In this case, $A^{\times}_X$ is precisely the group of units $A^\times$.
		
\item If $\GG$ is an $X$-groupoid and $A= \kk \GG$, then the identity morphisms $\{e_x\}_{x\in X}$ of $\GG$ form a set of orthogonal nonzero idempotents in $\kk \GG$. Hence, $\GG$ is a subgroupoid of $(\kk \GG)^{\times}_{X}$ since $gg^{-1} = e_{t(g)}$ and $g^{-1}g = e_{s(g)}$ for all $g\in \GG_1$. If $|X| = 1$, then this reduces to the fact that $\GG$ is a subgroup of the group $(\kk \GG)^\times$. Note that, in general, $(\kk \GG)_X^\times$ is not equal to $\kk \GG$.

\item[(c)] Let $A=\bigoplus_{x\in X} A_x$ be an $X$-decomposable $\kk$-algebra. Then $A^{\times}_X$ is a disjoint collection of groups due to the fact that $\Hom_{A^{\times}_X}(x,x) \cong (A_x)^\times$ as multiplicative groups and $\Hom_{A^{\times}_X}(x,y) = 0$, for every $x,y\in X$ with $x \neq y$.

\item[(d)] Let $V= \bigoplus_{x\in X} V_x$ be an $X$-decomposable vector space with $V_x \neq 0$ for all $x \in X$. Consider $A = \End(V)$ and recall the notation introduced in the discussion before Definition~\ref{def:kGGmod}.
For all $x \in X$, we define $e_x = \iota_x \pi_x$. Then, $e_y A e_x \cong \Hom_{\Vec}(V_x,V_y)$, and so $f \in e_y A e_x$ is a local unit if and only if it is an isomorphism of vector spaces between $V_x$ and $V_y$. Hence, $A^{\times}_X \cong \GL_X(V)$ (recall Definition~\ref{def:autgroupoid1}). If $|X|=1$, then we recover that $(\End(V))^\times = \GL(V)$.
	\end{enumerate}
\end{example}

\subsection{Groupoids of grouplike elements}
\label{sec:gpdelt}

Now, we study the groupoid of grouplike elements $\Gamma(H)$ of a weak Hopf algebra $H$, which extends the classical group formed by the grouplike elements of a Hopf algebra. In a weak Hopf algebra $H$, we call an element $h \in H$ \emph{grouplike} if $\Delta(h) = h \otimes h$ and $\epsilon(h) = 1_{\kk}$. Following \cite[Corollary~6.6, Proposition~6.8, Theorem~8.4]{BGL}, we see that the grouplike elements of $H$ naturally form a groupoid.

\begin{definition}[$\Gamma(H)$] \label{def:groupoidGamma}
Let $H$ be a weak Hopf algebra. The \emph{groupoid of grouplike elements} of $H$, denoted $\Gamma(H)$, is defined as follows. The morphisms of $\Gamma(H)$ are the elements in
\[
\Gamma(H)_1 = \{ h \in H : \Delta(h) = h \otimes h, \; \ep(h) = 1_\kk\}.
\]
The object set $\Gamma(H)_0$ is $\Gamma(H)_1 \cap H_s$. 
Each $h \in \Gamma(H)_1$ is an element in $\Hom_{\Gamma(H)}(\ep_s(h), \ep_t(h))$ with composition defined by restriction of the product in $H$. The inverse of a morphism $h \in \Hom_{\Gamma(H)}(\ep_s(h), \ep_t(h))$ is given by $S(h) \in \Hom_{\Gamma(H)}(\ep_t(h), \ep_s(h))$. 
\end{definition}

If $H$ is a Hopf algebra, then  $\Gamma(H)_1$ gives the classical group of grouplike elements of $H$, and so we use the same notation $\Gamma(H)$ whether $H$ is a Hopf algebra or a weak Hopf algebra.
We remark that $\Gamma(H)_0$ is a finite set since $H_t$ is finite-dimensional and grouplike elements are linearly independent.
However, $\Gamma(H)_0$ may be empty (see Example~\ref{ex:hay} below).

We also remark that the definition of grouplike element that we use here is different to the one given by Nikshych, where an element $h \in H$ is called grouplike if $\Delta(h) = 1_1h \otimes 1_2 h = h 1_1 \otimes h 1_2$ \cite[Definition 4.1]{Nik02}. With Nikshych's  definition, the grouplike elements in a weak Hopf algebra form a group, rather than a groupoid. For the present paper, the groupoid notion will be more useful.

We now seek to understand $\Gamma(H)_0$ in two special cases: when $H_s \cap H_t \cong \kk$ and when $H_s = H_t$.

The next result shows that it is possible to understand the object set $\Gamma(H)_0$ in terms of the \emph{minimal weak Hopf subalgebra} $\Hmin := H_s H_t$, which was studied in \cite{Nik02, Nik04}. Note that, since the $\kk$-algebras $H_s$ and $H_t$ are finite-dimensional and commute with each other, $\Hmin$ is indeed a finite-dimensional weak Hopf subalgebra of $H$.

\begin{lemma}
\label{lem:GammaH3}
Let $H$ be a weak Hopf algebra. Then 
\[
\Gamma(H)_0 = \{p \in H_s\cap H_t ~\mid~ p \text{ is a nonzero idempotent and } \dim(\Hmin p) = 1 \}.
\]
\end{lemma}

\begin{proof}
For any grouplike element $ p \in \Gamma(H)_0 = \Gamma(H)_1 \cap H_s $, since $p \in H_s$, we have $p = \varepsilon_s(p)$. From \cite[Lemma 6.5(iii) and (v)]{BGL}, it follows that
\begin{equation}
    \label{eq:epS}
    p = \varepsilon_s(p) = \varepsilon_t(S(p)) \in H_t.
\end{equation}
Therefore, $\Gamma(H)_0 \subseteq H_s \cap H_t$. Note that we may identify $p$ with the identity morphism at object $p$, and so $p^2 = p$. Since $\ep(p) = 1$, this shows that $p$ is a nonzero idempotent.

From \cite[Proposition 3.4 and Remark 3.6(ii)]{Nik02} we conclude that $p$ is a central idempotent of $\Hmin$, and thus $\Hmin p$ is a $\kk$-algebra with unit $p$. Clearly, $\Delta(\Hmin p) = \Delta(\Hmin) \Delta(p) \subseteq \Hmin p \otimes \Hmin p$. 
Moreover, since $p$ is identified with the identity morphism at $p$, we have $S(p) = p$. Hence,
for any element $xp \in \Hmin p$, we have
\[
S(xp) = S(p)S(x) = pS(x) = S(x)p \in \Hmin p,
\]
since $p \in H_s \cap H_t$ commutes with $\Hmin$.
Thus, $\Hmin p$ is a weak Hopf algebra. In fact, it is a Hopf algebra, since $p$ is the identity of $\Hmin p$ and $\Delta(p) = p \otimes p$.

Now we will show that $\Hmin p = (\Hmin p)_{\rm min}$. As an intermediate step will we show that $(H_s H_t p)_s = H_s p$. (A similar argument shows that $(H_s H_t p)_t = H_t p$.) Suppose first that we have an arbitrary element of $H_s p$, which we will write as $xp$ for some $x \in H_s$. Then $xp \in (H_s H_t p)_s$, since $xp \in H_s H_t p$ and  $xp \in H_s$ so $\ep_s(xp)=xp$. For the reverse containment, now suppose that we have an arbitrary element of $(H_s H_t p)_s$, which we will write as $\ep_s(xyp)$ for some $x \in H_s, y\in H_t$. Then we have
\begin{equation*}
\ep_s(xyp) = \ep_s(xy\ep_s(p))\overset{*}{=}\ep_s(xy)\ep_s(p)=\ep_s(xy)p \in H_s p,
\end{equation*}
where $*$ follows from \cite[Lemma 2.5, Eq. 2.13b]{BNS}. Then we have
\begin{equation*}
    (\Hmin p)_{\rm min} = (\Hmin p)_s (\Hmin p)_t = (H_s H_t p)_s (H_s H_t p)_t = (H_s p) (H_t p) = H_s H_t p = \Hmin p,
\end{equation*}
where the second to last equality is justified since $p\in H_s$ which commutes with $H_t$ and $p$ is idempotent.

Since $\Hmin p = (\Hmin p)_{\rm min} = \kk p$, we conclude that $\dim(\Hmin p) = 1$.
 
Conversely, suppose that $p \in H_s\cap H_t$ is a nonzero idempotent such that $\Hmin p$ has dimension 1. Because $\Hmin p$ has a weak Hopf algebra structure, $\Delta(\Hmin p)\subset \Hmin p\otimes \Hmin p$. Write $\Delta(p) = \alpha p \otimes p$, for $\alpha \in \kk$. Since $\Delta(p) = \Delta(p^k) = \alpha^k p\otimes p$ for any positive integer $k$, $\alpha^k = \alpha$ and so $\alpha = 1$. So, $\Delta(p) = p \otimes p$. Next, $p = (\varepsilon \otimes \id)\Delta(p) = \varepsilon(p)p$; thus, $\varepsilon(p) = 1_\kk$. Hence, $p \in \Gamma(H)_0$, as desired. 
\end{proof}

Next, we consider the case when $H_s \cap H_t \cong \kk$.

\begin{proposition}
\label{prop:GammaH1}
If $H$ is a weak Hopf algebra with $H_s \cap H_t = \kk$, then $H$ is either a Hopf algebra or $\Gamma(H)_0 = \emptyset$.
\end{proposition}

\begin{proof}
By Lemma~\ref{lem:GammaH3}, we have that $|\Gamma(H)_0| \leq 1$. If $|\Gamma(H)_0| = 1$, then there exists some scalar $\alpha \in \kk$ such that $\Gamma(H)_0 = \alpha 1_H$. Since $\alpha 1_H$ is grouplike, $\Delta(1_H) = \alpha(1_H \otimes 1_H)$. By \cite[Proposition 2.2.2]{NV}, $H_t = H_s = \kk 1_H$ and hence $H$ is a Hopf algebra. Otherwise, $|\Gamma(H)_0| = 0$, whence $\Gamma(H)_0 = \emptyset$.
\end{proof}
\begin{example}
\label{ex:hay}
Let $N \geq 2$ be an integer and $\epsilon \in \{1, -1\}$. Consider the face algebra $H = \mf{S}(A_{N-1}; t)_{\epsilon}$ introduced by Hayashi \cite[Example 2.1]{Hayashi99b}. Then $H$ is a weak Hopf algebra with $H_s \cap H_t = \kk 1_H$. So by the result above, we have that $\Gamma(H)_0 = \emptyset$.
\end{example}

We now consider the case when $H_s = H_t$. In particular, this holds if $H$ is a cocommutative weak Hopf algebra.

\begin{proposition}
\label{prop:hs=ht}
Let $H$ be a weak Hopf algebra with $H_s = H_t$. Let $A$ be an $H$-module algebra, in the sense of Proposition~\ref{def:modulealgoverWHA}. Then the following statements hold.
\begin{enumerate}[label=\normalfont(\alph*)]
\item Let $\{e_1, \dots, e_n\}$ be a complete set of primitive orthogonal idempotents of $H$. Then $\Gamma(H)_0 = \{e_1, \dots, e_n\}$. In particular, $e_i$ is grouplike for all $i$.
\item Let $\{e_1, \dots, e_n\}$ be a complete set of primitive idempotents of $H$. Then $A = \bigoplus_{i=1}^n A_i$ is a $\{e_1, \dots, e_n\}$-decomposable $\kk$-algebra.
The local identities of $A$ are given by the family of orthogonal idempotents $\{e_i \cdot 1_A \mid 1\leq i\leq n\}$.
\item Suppose, further, that $H = \bigoplus_{x \in X} H_x$ is a direct sum of weak Hopf algebras $H_x$. Then $A = \bigoplus_{x \in X} A_x$ is $X$-decomposable, and for each $x \in X$, $A_x$ is an $H_x$-module algebra obtained by restricting the action of $H$ on $A$. Moreover, $H_y\cdot A_x =0$ if $x\ne y \in X$.
\end{enumerate}
\end{proposition} 

\begin{proof}
(a) If $H_s=H_t$, then $H_s=H_t=\Hmin$ is a commutative semisimple $\kk$-algebra by Remark~\ref{rem:HsHt}. Hence, $H_t=\bigoplus_{i=1}^n H_t e_i$, where $\{e_1, \dots, e_n\}$ is a complete set of primitive orthogonal idempotents. Since $\kk$ is algebraically closed, $H_te_i\cong \kk e_i$ for $i$. Then, the statement follows from Lemma~\ref{lem:GammaH3}.

(b) Let $Y := \{e_1, \dots, e_n\}$ as in part (a). Next define the linear map $\eps^{\prime}: H \to H$ by $\eps^{\prime}(h)= 1_1\ep(1_2h)$. It follows from a left-sided version of \cite[Proposition 4.15]{C-DG} that
\begin{equation}
 \label{eq:hmodule}
 (h\cdot 1_A)a=\ept(h) \cdot a \quad
 \textnormal{and} \quad 
 a(h\cdot 1_A)=\eps'(h) \cdot a, 
\end{equation}
for any $a\in A$ and $h\in H$.
Note that for $h \in H$, we have
\[
\eps'(h) = 1_1 \ep(1_2 h) = \textstyle \sum_{i=1}^n e_i \ep(e_i h) = \sum_{i=1}^n \ep(e_i h)e_i = \ep(1_1 h)1_2 = \ept(h).
\] 
Therefore, by \eqref{eq:hmodule} we see that
\begin{equation}
 \label{eq:hmodule2}
 (h\cdot 1_A)a=a(h\cdot 1_A)
\end{equation}
for any $a\in A$ and $h\in H$. 
By the computation above, we also see that for any $1 \leq j \leq n$, we have
\begin{align}
 \label{eq:epsej}
 \eps'(e_j)=e_j.
\end{align}

By Definition~\ref{def:XdecompA}, to show that $A$ is an $Y$-decomposable $\kk$-algebra it suffices to check that $\{e_i\cdot 1_A ~\mid~ 1\leq i\leq n\}$ is a complete set of orthogonal central idempotents of $A$. 
Indeed, 
\begin{equation*}
 (e_i\cdot 1_A)(e_j\cdot 1_A) \overset{\eqref{eq:hmodule}}{=}\eps'(e_j) \cdot(e_i\cdot 1_A)\overset{\eqref{eq:epsej}}{=}e_j\cdot (e_i\cdot 1_A)=\delta_{i,j}e_i\cdot 1_A.
\end{equation*}
So $e_i \cdot 1_A$ and $e_j \cdot 1_A$ are mutually orthogonal idempotents. It follows from $\sum_{i=1}^n (e_i\cdot 1_A) = (\sum_{i=1}^n e_i)\cdot 1_A=1_A$ that $\{e_i\cdot 1_A\mid 1 \leq i \leq n\}$ is a complete set of idempotents. Letting $h = e_i$ in \eqref{eq:hmodule2}, we get $(e_i\cdot 1_A)a = a(e_i\cdot 1_A)$ for all $a \in A$. Hence each $e_i \cdot 1_A$ is central in $A$.

(c) We give the proof in the case that $|X| = 2$. The proof easily generalizes to any finite set $X$. Let $H_x$ and $H_y$ be weak Hopf algebras and suppose that $A_x$ is an $H_x$-module algebra and $A_y$ is an $H_y$-module algebra. By part (b), we know that $A$ is $Y$-decomposable, where $Y = \{e_1, \dots, e_n\}$ is a complete set of primitive idempotents of $H$, and the $\kk$-algebra decomposition of $A$ is given by
\[
A = \textstyle \bigoplus_{i=1}^n (e_i \cdot 1_A) A.
\]
Suppose that $1_{H_x} = e_1 + \dots + e_k$ and $1_{H_y} = e_{k+1} + \dots + e_n$. Then
\[
A = \textstyle \bigoplus_{i=1}^k (e_i \cdot 1_A) A \oplus \bigoplus_{i=k+1}^n (e_i \cdot 1_A) A = (1_{H_x} \cdot 1_A)A \oplus (1_{H_y} \cdot 1_A) A.
\]
Hence, letting $A_x = (1_{H_x} \cdot 1_A) A$ and $A_y = (1_{H_y} \cdot 1_A) A$, we see that $A = A_x \oplus A_y$.

We claim that $H_x\cdot A_x \subseteq A_x$, but $H_x \cdot A_y = 0$ (and the statement holds by symmetry with $x$ and $y$ reversed). 
Choosing $h_x \in H_x$ and $a_y = (1_{H_y}\cdot 1_A)a \in A_y$ for some $a\in A$, we compute:
\[
h_x\cdot a_y = h_x\cdot((1_{H_y}\cdot 1_A)a) = ((h_x)_{1}\cdot(1_{H_y}\cdot 1_A))((h_x)_{2}\cdot a) = (((h_x)_{1} 1_{H_y})\cdot 1_A)((h_x)_{2}\cdot a)=0.
\] 
Without loss of generality, we assume that $h_x \cdot a_x = b_x + b_y$ where $h_x \in H_x$, $a_x, b_x \in A_x$, and $b_y \in A_y$. Then, 
$h_x \cdot a_x = 1_{H_x}\cdot b_x + 1_{H_x}\cdot b_y = b_x$.
So, $H_x\cdot A_x\subseteq A_x$. Likewise, $H_y\cdot A_y\subseteq A_y$ as claimed.
\end{proof}

As a corollary to the previous proposition, we obtain that groupoid actions on domains factor through group actions (see similar results for inner faithful Hopf actions on various domains factoring through actions of groups \cite{EW1, CEW1, CEW2, EW2, EW3}).

\begin{corollary} \label{cor:grpoid-innerfaith}
Suppose that $\GG$ is an $X$-groupoid and $A$ is a domain such that $A$ is an inner faithful $\kk \GG$-module algebra. Then $\GG$ is a disjoint union of groups, and at most one of the groups is nontrivial.
\end{corollary}

\begin{proof}
By Proposition~\ref{prop:hs=ht}, if $A$ is a $\kk \GG$-module algebra, then $A$ is $X$-decomposable where $X = \{1, \dots, n\}$ is the set of objects of $\GG$, that is, $A = \bigoplus_{i =1}^n A_i$. Since $A$ is a domain, this implies that exactly one of the $A_i$ is nonzero. Without loss of generality, suppose $A_1 \neq 0$.
Now consider the ideal of $\kk \GG$ defined by
\[
I = \left\langle g - e_i \mid g \in \GG \text{ such that } s(g) = t(g) = i,\; 2 \leq i \leq n \right\rangle.
\]
It is straightforward to check that $I$ is in fact a weak Hopf ideal of $\kk \GG$.
Further, observe that if $a \in A_1$ and $g - e_i$ is a generator of $I$ (so $i \neq 1$), then we have $(g - e_i) \cdot a = 0$. Hence, $I \cdot A_1 = 0$, whence $I \cdot A$ = 0. Since $A$ is inner faithful, this implies that for $2 \leq i \leq n$, the only element $g \in \GG$ satisfying $s(g) = t(g) = i$ is the trivial path $e_i$. This shows that $\GG$ is a disjoint union of trivial groups, together with a possibly nontrivial group at vertex 1.
\end{proof}

%%%%%%%%%%%%%%%%%%%%%%%%%%%

\subsection{Key categories}
\label{sec:keycat}

Next, we define in detail some categories that we first mentioned in Notation~\ref{not:cats} and \ref{not:cats2}.

\begin{definition}[$\XAlg$, $\XWHA$, $\XGrAlg$]\label{def:xxx}
We define the following categories.

\begin{enumerate}[label=\normalfont(\alph*)]
\item Let $\XAlg$ be the category defined as follows:
		 
\begin{itemize}
\item The objects are unital $\kk$-algebras $A$ with a given set $\{e_x^A\}_{x\in X}$ of nonzero orthogonal idempotents such that $1_A = \sum_{x\in X} e_x^A$; we call these \emph{$X$-algebras}.
	 
\item The morphisms are unital $\kk$-algebra maps $f: A \to B$ such that $f(e_x^A) = e_x^B$ for every $x\in X$; we call these maps \emph{$X$-algebra morphisms}.
\end{itemize}
		 
\item Let $\XWHA$ be the category defined as follows:
	 
\begin{itemize}
\item The objects are weak Hopf algebras $H$ with a given set $\{e_x^H\}_{x \in X}$ of nonzero orthogonal
idempotents such that $1_H = \sum_{x \in X} e_x^H$, $\Delta(e_x^H) = e_x^H \otimes e_x^H$, and $\ep(e_x^H) = 1_\kk$ for all $x \in X$; we call these \emph{$X$-weak Hopf algebras}.
	 
\item The morphisms are unital weak Hopf algebra morphisms $f: H \to H'$ such that $f(e_x^H) = e_x^{H'}$ for all $x \in X$; we call these maps \emph{$X$-weak Hopf algebra morphisms}.
\end{itemize}	

\item Let $\XGrAlg$ be the category of $X$-groupoid algebras, that is, the full subcategory of $\XWHA$ consisting of groupoid algebras $\kk \GG$ over $X$-groupoids $\GG$.
\end{enumerate}
\end{definition}

\begin{remark} \label{rem.xweak}
\begin{enumerate}[label=\normalfont(\alph*)]
\item Every $X$-decomposable $\kk$-algebra is an $X$-algebra, and every morphism of $X$-decomposable $\kk$-algebras is an $X$-algebra morphism. However, as exemplified in Remark~\ref{rem:localunits}, the converse is not necessarily true, since the idempotents of an $X$-algebra are not required to be central.
\item We have that $\XWHA$ is a subcategory of $\XAlg$.
 
\item By Proposition~\ref{prop:hs=ht}(a), a weak Hopf algebra $H$ satisfying $H_s = H_t$ is an $X$-weak Hopf algebra, where $X$ is the complete set of primitive idempotents of $H$.
 
\item By \cite[Proposition 2.3.3]{NV}, weak Hopf algebra morphisms preserve counital subalgebras. Hence, by considering all possible finite sets $X$, our results pertain to all weak Hopf algebras with commutative counital subalgebras, and all morphisms between such weak Hopf algebras.
 
\item We remark that the category $\mathsf{wha}$ considered in \cite{BGL} has a weaker notion of morphisms than those in $\XWHA$ here. The morphisms in $\mathsf{wha}$ need not be weak Hopf algebra morphisms (in particular, they are not necessarily $\kk$-algebra morphisms; see \cite[Theorem 4.12]{BGL}).

\item Similar to what was stated for $X$-groupoids in Remark~\ref{rem:permutation}, we could define a weaker notion of $X$-algebra morphism by allowing for permutations of the index set $X$. However, we opted not to do so in order to maintain simplicity in our proofs and because a suitable relabeling of the idempotents would ultimately reduce the situation to our case.
\end{enumerate}
\end{remark}

\begin{example}\label{ex:groupoidalgXWHA}
For an $X$-groupoid $\GG$, the groupoid algebra $\kk \GG$ belongs to $\XWHA$.
\end{example}

\begin{example} \label{ex:GGrep}
Let $\GG$ be an $X$-groupoid. Then, an $X$-decomposable vector space $V = \bigoplus_{x \in X} V_x$ is a representation of $\kk\GG$ if and only if the representation map $\kk\GG \to \End(V)$ is a morphism in $\XAlg$.
\end{example}

\begin{lemma}\label{lem:xalgmap}
	Let $A,B \in \XAlg$ with respective sets of idempotents $\{e_x^A\}_{x\in X}$ and $\{ e_x^B \}_{x\in X}$. If $f: A \to B$ is an $X$-algebra map and $a \in A$ is a local unit in $A$, then $f(a)$ is a local unit in $B$.
\end{lemma}

\begin{proof}
	Let $a \in e_y^A A e_x^A$ for some $x,y\in X$. Then, $f(a)=f(e_y^A a e_x^A) = e_y^B f(a) e_x^B \in e_y^B B e_x^B$. Moreover, $f(a) f(\widehat{a}) = f( a \widehat{a} ) = f(e_y^A) = e_y^B$, and similarly, $f(\widehat{a} ) f(a) = e_x^B$.
\end{proof}

%%%%%%%%%%%%%%%%%%%%%%%%%%%

\subsection{Module algebras over groupoid algebras}
\label{sec:gpdmodalg}

When $G$ is a group, the group algebra $\kk G$ is a Hopf algebra. The functor $\kk(-): \Group \to \Algcat$, which associates to the group $G$ its group algebra $\kk G$, is left adjoint to the functor $(-)^{\times}:\Algcat \to \Group$, which associates to a $\kk$-algebra $B$ its group of units $B^\times$. Since $\kk G$ is actually a Hopf algebra, we may also view the group algebra functor as a functor $\kk(-):\Group \to \Hopf$. In this case, it is left adjoint to the functor $\Gamma(-):\Hopf \to \Group$, which associates to the Hopf algebra $H$ its group of grouplike elements $\Gamma(H)$. We generalize these two adjunctions to the groupoid case, recovering the classical setting when $|X| = 1$.

\begin{theorem} \label{thm:repkGrp}
Let $X$ be a finite nonempty set.
	\begin{enumerate}[label=\normalfont(\alph*)]
		\item The following functors are well-defined: 
		\[
		\begin{array}{l}
			\kk(-): \XGrpd \longrightarrow \XAlg \qquad \text{(groupoid algebra);}\\ 
			(-)^{\times}_X: \XAlg \longrightarrow \XGrpd \qquad \text{(groupoid of local units).}
		\end{array}
		\]
		Moreover, $\kk(-) \dashv (-)^{\times}_X$, that is, for an $X$-groupoid $\GG$ and an $X$-algebra $B$, we have a bijection that is natural in each slot: 
		\[
	\Hom_{\XGrpd}(\GG, B^{\times}_X ) \; \cong \; \Hom_{\XAlg}(\Bbbk \GG, B).
		\]

		\item 
	 The following functors are well-defined:
		\[
		\begin{array}{l}
			\kk(-): \XGrpd \longrightarrow \XWHA \qquad \text{(groupoid algebra);}\\ 
			\Gamma(-): \XWHA \longrightarrow \XGrpd \qquad \text{(groupoid of grouplike elements).}
		\end{array}
		\]
		Moreover, $\kk(-) \dashv \Gamma(-)$, that is, for an $X$-groupoid $\GG$ and an $X$-weak Hopf algebra $H$, we have a bijection that is natural in each slot: 
		\[
	\Hom_{\XGrpd}(\GG, \Gamma(H)) \; \cong \; \Hom_{\XWHA}(\Bbbk \GG, H).		\]
	In particular, $\Gamma(\kk \GG) = \GG$. 
	\end{enumerate}
\end{theorem}

\begin{proof}
	(a) First, we prove that $\kk(-): \XGrpd \to \XAlg$ is indeed a functor. Let $\GG$ be an $X$-groupoid. As mentioned in Remark~\ref{rem:localunits} and Example~\ref{exa:localunits}(c), $\kk \GG$ is an $X$-algebra with idempotent set $\{e_x^{\GG}\}_{x\in X}$ given by the identity morphisms of $\GG$, so $\kk(-)$ sends objects to objects. Also, if $\phi: \GG \to \GG'$ is an $X$-groupoid morphism, then we can consider $\phi$ as a function between the sets $\GG_1$ and $\GG'_1$ so the linear extension $\kk \phi: \kk \GG \to \kk \GG'$ is well-defined. 
	It is straightforward to check that $\kk \phi$ is an $X$-algebra map. Hence, $\kk(-)$ also sends morphisms to morphisms. Since $\kk \phi = \phi$ when restricted to basis elements, $\kk(-)$ respects compositions and identities.

	Secondly, we check that $(-)^{\times}_X: \XAlg \to \XGrpd$ is a functor. By construction, the groupoid of local units is an $X$-groupoid, so $(-)^{\times}_X$ sends objects to objects. Moreover, if $\psi:A \to B$ is an $X$-algebra map, we can define the $X$-groupoid map $\psi^{\times}_X: A^{\times}_X \to B^{\times}_X$ where $a: x\to y$ maps to $\psi(a): x \to y$ (which is a local unit by Lemma~\ref{lem:xalgmap}). So $\psi^{\times}_X$ fixes $X$ and it is a functor due to the properties of $\psi$. Also, $(-)^{\times}_X$ respects compositions and identities.

	Now, for any $X$-groupoid $\GG$ and any $X$-algebra $B$, we want to display a bijection
	\begin{equation*}
		\Hom_{\XGrpd}(\GG, B^{\times}_X ) \cong \Hom_{\XAlg}(\Bbbk \GG, B).
	\end{equation*}
	 Given an $X$-groupoid morphism $\phi \in \Hom_{\XGrpd}(\GG, B^{\times}_X )$, we consider it as a function $\phi : \GG_1 \to (B^{\times}_X)_1$, and construct the linear extension $\kk\phi : \kk \GG \to B$.  Note that for $g\in \GG_1$, $\kk\phi(g) = \phi(g)$, which implies that $\kk\phi$ is indeed an $X$-algebra map. Such verification uses that $1_B = \sum_{x\in X} e_x^B$. Hence, we have constructed the assignment
\[
		\Phi_{\GG,B}:\Hom_{\XGrpd}(\GG, B^{\times}_X ) \longrightarrow \Hom_{\XAlg}(\Bbbk \GG, B), \quad
		\phi \mapsto \kk\phi.
\]
	On the other hand, given an $X$-algebra map $ \psi \in \Hom_{\XAlg}(\Bbbk \GG, B)$, consider the restriction $\psi|_{\GG} : \GG \to B^{\times}_X$ given by $\psi|_{\GG}(g: x \to y) := \psi(g): x \to y$. As above, this assignment is well-defined due to Lemma~\ref{lem:xalgmap}. The fact that $\psi|_{\GG}$ is a functor follows from $\psi$ being an $X$-algebra map. Hence, we have constructed the assignment
\[
\Psi_{\GG,B}:\Hom_{\XAlg}(\Bbbk \GG, B) \longrightarrow \Hom_{\XGrpd}(\GG, B^{\times}_X ), \quad \psi \mapsto \psi|_{\GG}.
\]
	The following calculations show that these assignments are mutually-inverse. If $g: x \to y$ in $\GG_1$, then
	\begin{gather*}
		[( \Psi_{\GG,B} \circ \Phi_{\GG,B} ) ( \phi )] (g) = \Psi_{\GG,B} (\kk\phi) (g) = \kk\phi|_{\GG} (g) =\phi(g),\\
		[( \Phi_{\GG,B} \circ \Psi_{\GG,B}) (\psi)] (g) = \Phi_{\GG,B} (\psi|_{\GG})(g) = (\psi|_{\GG})' (g) = \psi(g).
	\end{gather*}
Finally, the bijection is natural since for any $X$-groupoid morphism $\varphi: \GG' \to \GG$, any $X$-algebra map $f: B \to B'$, and all $\psi \in \Hom_{\XAlg}(\Bbbk \GG, B)$ one can check the following equality:
\begin{equation*}
(	\Phi_{\GG',B'} \circ \Hom_{\XAlg}(\kk \varphi,f )) (\psi) = (\Hom_{\XGrpd} ( \varphi, f^{\times}_X ) \circ \Phi_{\GG,B} )( \psi ).
\end{equation*}
Here, $\Hom_{\XAlg}(\kk \varphi,f ) : \Hom_{\XAlg}(\kk\GG, B) \to \Hom_{\XAlg}(\kk \GG', B') $ is given by composition, i.e., $\psi \mapsto f \circ \psi \circ \kk \varphi $; a similar notion holds for $\Hom_{\XGrpd} ( \varphi, f^{\times}_X ) $.

\noindent (b) First, we show that $\kk(-): \XGrpd \to \XWHA$ is a functor. If $\GG$ is an $X$-groupoid, then $\kk \GG$ is a weak Hopf algebra. The set $\{e^{\kk \GG}_x\}_{x \in X}$ of identity morphisms of $\GG$ are idempotent elements of $\kk \GG$ satisfying $1_{\kk \GG} = \sum_{x \in X} e_x^{\kk \GG}$ and $\Delta(e_x^{\kk \GG}) = e_x^{\kk \GG} \otimes e_x^{\kk \GG}$ for each $x \in X$, and so $\kk \GG$ is an object of $\XWHA$. If $\phi: \GG \to \GG'$ is an $X$-groupoid morphism, then by part (a), $\kk\phi: \kk \GG \to \kk \GG'$ is a map of $X$-algebras. Since the elements of $\GG$ form a $\kk$-basis of grouplike elements for $\kk \GG$, and $\kk\phi$ maps these elements to grouplike elements of $\kk \GG'$, we have that $\kk\phi$ is also a $\kk$-coalgebra map. Since the antipode of $\kk \GG$ is defined by $S_{\kk \GG}(g) = g\inv$ for all $g \in \GG$, therefore $S_{\kk \GG'} \circ \kk\phi = \kk \phi \circ S_{\kk \GG}$, so $\kk \phi$ is a morphism in $\XWHA$.

Now we show that $\Gamma(-): \XWHA \to \XGrpd$ is a functor. Note that $\Gamma(-)$ is the restriction of the functor $\mathsf{g}(-): \mathsf{wha} \to \Grpd$ in \cite[Theorem 8.4]{BGL} to $\XWHA$, where the category $\mathsf{wha}$ has a weaker notion of morphism; see Remark~\ref{rem.xweak}(e). 
So, we need only check that if $H \in \XWHA$ then $\Gamma(H) \in \XGrpd$. Indeed, if $H \in \XWHA$, then $\Gamma(H)$ is a groupoid with object set $\{e_x^H\}_{x \in X}$, and so $\Gamma(H) \in \XGrpd$. We remark that if $f: H \to H'$ is a morphism in $\XWHA$, then $\Gamma(f)$ is the restriction of $f$ to $\Gamma(H)$.

Finally, for any $\GG \in \XGrpd$ and $H \in \XWHA$, we wish to exhibit a bijection
\[
\Hom_{\XGrpd}(\GG, \Gamma(H)) \cong \Hom_{\XWHA} (\kk \GG, H).
\]
Given $\phi \in \Hom_{\XGrpd}(\GG, \Gamma(H))$, the $\kk$-linearization $\kk \phi: \kk \GG \to \kk \Gamma(H)$ is a morphism of $X$-weak Hopf algebras. Since composition in $\Gamma(H)$ was defined as restriction of multiplication in $H$, the map $\nu_H: \kk \Gamma(H) \to H$ induced by the inclusion $\Gamma(H) \subseteq H$ is a unital $\kk$-algebra map. Also, the elements of $\Gamma(H)$ are grouplike in $H$, so it is also a $\kk$-coalgebra map. Moreover, the inverse in $\Gamma(H)$ coincides with the antipode of $H$, so $\nu_H$ intertwines the antipodes of $\kk \Gamma(H)$ and $H$. Therefore, $\nu_H \circ \kk \phi \in \Hom_{\XWHA}(\kk \GG, H)$, and we have the assignment
\[
\Phi_{\GG, H}: \Hom_{\XGrpd}(\GG, \Gamma(H)) \longrightarrow \Hom_{\XWHA} (\kk \GG, H), \quad \phi \mapsto \nu_H \circ \kk \phi.
\]
Conversely, if $\psi \in \Hom_{\XWHA}(\kk \GG, H)$, then $\Gamma(\psi): \Gamma(\kk \GG) \to \Gamma(H)$ is an $X$-groupoid morphism. Note that $\Gamma(\kk \GG) = \GG$. Indeed, it is clear that the elements of $\GG$ are grouplike elements of $\kk \GG$. On the other hand, any element of $\kk \GG$ is of the form $a = \sum_{g \in \GG} \alpha_g g$, where
\[
a \otimes a = \textstyle \left(\sum_{g \in \GG} \alpha_g g\right) \otimes \left(\sum_{h \in \GG} \alpha_h h \right) = \sum_{g,h \in \GG} \alpha_g \alpha_h (g \otimes h)
\]
and $\Delta(a) = \sum_{g \in \GG} \alpha_g (g \otimes g)$. By the linear independence of grouplike elements, if $a \in \Gamma(\kk \GG)$ we have that $\alpha_g \alpha_h = \delta_{g,h}\alpha_g$, whence exactly one $\alpha_g = 1$ and the remaining coefficients are $0$. Hence, the only grouplike elements of $\kk \GG$ are the elements of $\GG$. We have therefore constructed an assignment
\[
\Psi_{\GG, H}:\Hom_{\XWHA} (\kk \GG, H) \longrightarrow \Hom_{\XGrpd}(\GG, \Gamma(H)), \quad \psi \mapsto \Gamma(\psi) \circ \eta_{\GG},
\]
where $\eta_{\GG}: \GG \to \Gamma(\kk \GG)$ is the identity map.

It is straightforward to check that $\Phi_{\GG, H}$ and $\Psi_{\GG, H}$ are mutually-inverse. Naturality holds by \cite[Theorem~8.4]{BGL}.
\end{proof}

\begin{corollary}\label{cor:repiso}
Let $\GG$ be an $X$-groupoid. Then, the categories $\GGrep$ and $\rep(\kk \GG)$ are monoidally isomorphic.
\end{corollary}

\begin{proof}
Let $V = \bigoplus_{x \in X} V_x$ be an $X$-decomposable vector space and $B = \End(V)$. Then, $B_{X}^{\times} = \GL_{X}(V)$, and from Example~\ref{ex:GGrep} and Theorem~\ref{thm:repkGrp}(a) we have
\begin{equation}\label{eq:repGG}
    \Hom_{\XGrpd}(\GG, \GL_{X}(V)) \cong \Hom_{\XAlg}(\Bbbk \GG, \End(V)).
\end{equation}
Consider the functor $F:\GGrep \to \rep(\kk \GG)$ that sends a $\GG$-representation $(V,\pi)$ to the representation $(V,\widetilde{\pi})$, where $\widetilde{\pi}$ is the $X$-algebra morphism corresponding to the $X$-groupoid morphism $\pi$ through the correspondence \eqref{eq:repGG}. Moreover, given $\GG$-representations $(V,\pi)$ and $(W,\eta)$, if $\varphi = \{\varphi_x:V_x \to W_x\}_{x\in X}$ is a morphism of $\GG$-representations, then define $F(\varphi)=\bigoplus_{x\in X} \varphi_x$. Clearly, $F$ is an isomorphism of categories.

Define $F_{V, W}:F(V)\otimes_{\rep(\kk \GG)} F(W)\to F(V\otimes_{\GGrep}W)$ as $F_{V, W}=\bigoplus_{x\in X} \id_{V_x\otimes W_x}$. The collection $\{F_{V,W}\}_{V,W\in\GGrep}$ forms a natural isomorphism. Additionally, consider the morphism $F_0:\unit_{\rep(\kk \GG)}\to F(\unit_{\GGrep})$ in $\rep(\kk \GG)$ given by $F_0=\bigoplus_{x\in X}\id_{\kk}$, which is invertible. Through straightforward verification, it can be checked that $F$ satisfies the associativity and unit constraints, proving it is a strong monoidal functor.
\end{proof}

This result, combined with Lemmas~\ref{lem:groupoidrep} and \ref{lem:kGGmod}, establishes the following sequence of monoidal isomorphisms between categories:
\[\rep(\kk \GG) \cong \rep(\GG) \cong \GGmod \cong \kk\GGmod.\]
For a group $G$, Corollary~\ref{cor:repiso} recovers the classical monoidal isomorphism between $\rep(G)$ and $\rep(\kk G)$, achieved by linearizing the representation map. Furthermore, it leads to the familiar monoidal isomorphisms: $\rep(\kk G) \cong \rep(G) \cong G\lmod \cong \kk G\lmod$.

Finally, using the adjunction in Theorem~\ref{thm:repkGrp}(b) along with Proposition~\ref{prop:groupoidalg}, we can identify the symmetry object object $\Sym_{\XGrAlg}(A)$ for an $X$-decomposable $\kk$-algebra $A$. Here, it is important to note that for an $X$-groupoid $\GG$, $\kk\GG$ is a weak Hopf algebra, so when we refer to a $\kk\GG$-module algebra, we mean a $\kk$-algebra satisfying condition \eqref{eq.Hmodalgebra} in Proposition~\ref{def:modulealgoverWHA}.

\begin{proposition} \label{prop:kGrpd}
Let $A = \bigoplus_{x\in X} A_x$ be an $X$-decomposable $\kk$-algebra and let $\GG$ be an $X$-groupoid. Then:
\begin{enumerate}[label=\normalfont(\alph*)]
\item The action of $\XAutA$ on $A$ can be extended linearly to make $A$ a $\kk(\XAutA)$-module algebra. For the action map $\ell_{\kk(\XAutA), A}: \kk (\XAutA) \otimes A \to A$, we use the notation $h \rt a := \ell_{\kk(\XAutA), A}(h \otimes a)$ for $h \in \kk(\XAutA)$ and $a \in A$.

\item Suppose that $A$ is a $\kk \GG$-module algebra via $\ell_{\kk \GG, A}$, where we write $h \cdot a$ for $\ell(h \otimes a)$ for $h \in \kk \GG$ and $a \in A$. Then there is a unique $X$-weak Hopf algebra morphism $\widetilde{\pi}: \kk\GG \to \kk(\XAutA)$ such that $h \cdot a = \widetilde{\pi}(h) \rt a$ for all $h \in \kk\GG$ and $a \in A$.

\item Every $X$-weak Hopf algebra morphism $\widetilde{\pi}: \kk \GG \to \kk(\XAutA)$ gives $A$ the structure of a $\kk\GG$-module algebra via $\ell_{\kk\GG,A}(h \otimes a) = \ell_{\kk(\XAutA), A}(\widetilde{\pi}(h) \otimes a)$ for all $h \in \kk \GG$ and $a \in A$ (that is, $h \cdot a = \widetilde{\pi}(h) \rt a$). 
\end{enumerate}
Hence $\Sym_{\XGrAlg}(A) = \kk(\XAutA)$.
\end{proposition}

\begin{proof}
(a) Since the isomorphism $\kk\GG\lmod \cong \GGmod$ of Lemma~\ref{lem:kGGmod} is monoidal for any groupoid $\GG$, the notions of a $\GG$-module algebra and a $\kk\GG$-module algebra are equivalent. Therefore, since $A$ is an $\XAutA$-module algebra (Proposition~\ref{prop:groupoidalg}), by linearizing this action, $A$ becomes a $\kk(\XAutA)$-module algebra.

(b, c) By the above and Proposition~\ref{prop:groupoidalg}, $A$ is a $\kk(\XAutA)$-module algebra if and only if the action of $\GG$ on $A$ factors through a unique $X$-groupoid morphism $\pi: \GG \to \XAutA$. Now, letting $H = \kk(\XAutA)$ in Theorem~\ref{thm:repkGrp}(b), we obtain a bijection: 
\[
\Hom_{\XGrpd}(\GG, \XAutA) \cong \Hom_{\XAlg}(\kk\GG, \kk(\XAutA)), \quad \pi \mapsto \widetilde{\pi}.
\]
Hence, there exists an $X$-groupoid morphism $\pi: \GG \to \XAutA$ if and only if there exists an $X$-weak Hopf algebra morphism $\widetilde{\pi}: \kk\GG \to \kk(\XAutA)$ (which is simply the $\kk$-linearization of $\pi$). 

Note that since $\XGrAlg$ is a full subcategory of $\XWHA$, the morphisms in $\XGrAlg$ are simply $X$-weak Hopf algebra morphisms, which, as per Theorem~\ref{thm:repkGrp}(b), are all induced by $X$-groupoid morphisms. Consequently, (a), (b), and (c) imply that $\Sym_{\XGrAlg}(A) = \kk(\XAutA)$.
\end{proof}

\begin{remark}
	Proposition~\ref{prop:kGrpd} has an alternative presentation appearing in \cite[Proposition~2.2]{PF13} using the language of partial actions. 
\end{remark}

 We recover the following well-known result for actions of group algebras, which is the special case of Proposition~\ref{prop:kGrpd} when $|X| = 1$.

\begin{remark} \label{rem:kGrp}
Let $A$ be a $\kk$-algebra. Then $A$ is a $\kk (\AutA)$-module algebra and for a group $G$, the following are equivalent.
	\begin{enumerate}[label=\normalfont(\alph*)]
		\item $A$ is a $\kk G$-module algebra.
		\item There exists a Hopf algebra morphism $\widetilde{\pi}: \kk G \to \kk(\AutA)$.
	\end{enumerate}
Hence $\Sym_{\GrAlg}(A) = \kk(\AutA)$.
\end{remark}

%%%%%%%%%%%%%%%%%%%%%%%%%%%%%%%%%%%%%%%%%%%%%%%%%%%%%%%%%%%%%%%%%%%%%

\section{Actions of $X$-Lie algebroids on algebras} \label{sec:Liealgebroids}

The definition of a Lie algebroid was introduced in \cite{pradines1967}, using the language of vector bundles over manifolds (see also \cite[Chapter~III]{MacK}). An equivalent, purely algebraic structure known as a Lie--Rinehart algebra, was introduced in \cite{rinehart1963}.
Although some Hopf-like structures have been defined for general Lie--Rinehart algebras by means of Hopf algebroids (see \cite{Saracco20,Saracco21}), the study of their actions on $\kk$-algebras is still an open problem. The focus of this section is to study actions of a special subclass of Lie algebroids (Definition~\ref{def:Liealgebroid}); the main result here is Proposition~\ref{prop:Liealgebroidalg}. We also extend this result for universal enveloping algebras of these Lie algebroids (Definition~\ref{def:envelopLiealgebroid}) in Proposition~\ref{prop:U(GLie)}.

As in previous sections, here $X$ denotes a finite non-empty set.

\begin{definition}\label{def:Liealgebroid}
An \emph{$X$-Lie algebroid $\GLie$} is a direct sum of vector spaces
\[
\GLie := \textstyle \bigoplus_{x \in X} \mf{g}_x
\]
where $\mf{g}_x$ has structure of Lie algebra for all $x\in X$. We regard $\GLie$ as having a partially defined bracket $[-,-]$, which is only defined on pairs of elements from the same component $\mf{g}_x$. Namely, $[a, b]=[a, b]_{\mf{g}_x}$ if $a,b \in \mf{g}_x$ for some $x\in X$, and is undefined otherwise.

Let $\GLie= \bigoplus_{x \in X} \mf{g}_x$ and $\mf{G'}= \bigoplus_{x \in X} \mf{g}'_x$ be two $X$-Lie algebroids. A linear map $\tau: \GLie \to \GLie'$ is an \emph{$X$-Lie algebroid morphism} if there exists a collection $\{ \tau_x: \mf{g}_x \to \mf{g}^{\prime}_x\}_{x\in X}$ of Lie algebra morphisms such that $\tau|_{\mf{g}_x} = \tau_x$ for all $x \in X$. Here, we write $\tau = \{\tau_{x}\}_{x \in X}$.
\end{definition}

As remarked in \cite[Section 3.2]{NDthesis}, an $X$-Lie algebroid is an special kind of Lie algebroid in the sense of \cite{MacK}. Furthermore, one could propose a more general definition of $X$-Lie algebroid morphism by allowing the index set $X$ to permute. However, as noted in previous remarks, for simplicity, we do not consider this case.

\begin{definition}[$\XLie$]
\label{def:XLie}
The $X$-Lie algebroids together with $X$-Lie algebroid morphisms form a category, which we denote by $\XLie$.
\end{definition} 

Now, we present some examples of $X$-Lie algebroids. Recall that for a given $\kk$-algebra $A$ we can always define on $A$ a Lie bracket given by $[a,b] = ab - ba$, for all $a,b\in A$. We denote this associated Lie algebra structure on $A$ by $\mathcal{L}(A)$. 

\begin{example}
    \label{ex:LieXAlg}
    Let $A$ be an $X$-algebra with given set $\{e_x^A\}_{x \in X}$ of nonzero orthogonal idempotents. For $x\in X$ we define $A_x:= e_x^A A e_x^A$, which is a $\kk$-algebra with the multiplication of $A$ and identity element $e_x^A$. Hence $\mathcal{L}_X(A):= \bigoplus_{x\in X} \mathcal{L}(A_x)$ is an $X$-Lie algebroid, called the \emph{$X$-Lie algebroid associated to $A$}. There are two particular examples of this construction.
    \begin{enumerate}[label=\normalfont(\alph*)]
    \item When $|X|=1$ then $\mathcal{L}_X(A)=\mathcal{L}(A)$.
    \item If $A = \bigoplus_{x \in X} A_x$ is an $X$-decomposable $\kk$-algebra, then $\mathcal{L}_X(A) = \bigoplus_{x \in X} \mathcal{L}(A_x)$ is simply the direct sum of the classical associated Lie algebra structures on each component $A_x$.
    \end{enumerate}
\end{example}

\begin{example}
\label{ex:primitiveWHA}
    Let $H$ be an $X$-weak Hopf algebra with given set $\{e_x^H\}_{x \in X}$ of nonzero orthogonal idempotents. For $x\in X$, we say that $h\in H$ is an \emph{$x$-primitive element} of $H$ if $\Delta(h)=e_x^H \otimes h + h \otimes e_x^H$. This implies $\varepsilon(h)=0$ and $S(h)=-h$. If we denote by $P_x(H)$ the set of $x$-primitive elements of $H$, then it is easy to check that $P_x(H)$ becomes a Lie algebra under the commutator bracket $[h,h']=hh'-h'h$, for all $h,h'\in P_x(H)$. Hence, $P_X(H):=\bigoplus_{x\in X} P_x(H)$ is an $X$-Lie algebroid, called the \emph{$X$-Lie algebroid of primitive elements of $H$}. When $|X|=1$ then $P_X(H)=P(H)$, the classical Lie algebra of primitive elements of $H$.
\end{example}

 Throughout this section, $\GLie= \bigoplus_{x \in X} \mf{g}_x$ denotes an $X$-Lie algebroid. Next, we briefly discuss modules and representations over $X$-Lie algebroids. 

\begin{definition}[$\GLie$-module]\label{def:Liealgebroidmod}
 An $X$-decomposable vector space $V=\bigoplus_{x \in X} V_x$ is said to be a left \emph{$\GLie$-module} if $V_x$ is a $\mf{g}_x$-module for each $x\in X$.
 
 Given two left $\GLie$-modules $V=\bigoplus_{x \in X} V_x$ and $W=\bigoplus_{x \in X} W_x$, a linear map $f: V \to W$ is a \emph{$\GLie$-module morphism} if there exists a collection $\{ f_x : V_x \to W_x\}_{x\in X}$ of maps such that $f|_{V_x}=f_x$ and $f_x$ is a $\mf{g}_x$-module morphism for each $x\in X$. Here, we write $f= (f_x)_{x \in X}$. 
 \end{definition}
 
The composition of two $\GLie$-module morphisms is defined component-wise (and is again a $\GLie$-module morphism). Given $\GLie$-representations $(V,\tau), (W, \omega), $

\begin{lemma}[$\GLie\lmod$] \label{lem:Gliemod}
 Let $\GLiemod$ be the category of $\GLie$-modules with $\GLie$-module morphisms. This category inherits a monoidal structure from the monoidal structures of each $\mf{g}_x \lmod$ as follows.
 \begin{itemize}
     \item If $V=\bigoplus_{x \in X} V_x,W=\bigoplus_{x \in X} W_x \in \GLiemod$, then
     \[V \otimes_{\GLiemod} W:= \bigoplus_{x \in X} (V_x \otimes_{\mf{g}_x\lmod} W_x) = \bigoplus_{x \in X} (V_x \otimes W_x),\]
     \item $\unit_{\GLiemod} = \bigoplus_{x \in X} \unit_{\mf{g}_x\lmod} = \bigoplus_{x \in X} \kk$.
 \end{itemize}
If $f=(f_x)_{x \in X} : V \to W$ and $f'=(f'_x)_{x \in X} : V' \to W'$ are two $\GLie$-module morphisms, then $f \otimes_{\GLiemod} f'  = (f_x \otimes f'_x)_{x\in X}$. The associativity constraint is that induced by the tensor product of components, and the left/right unital constraints are given by scalar multiplication.
\end{lemma}

\begin{remark}
\label{GLieboxtimes}
For a $X$-Lie algebroid $\GLie = \bigoplus_{x \in X} \mf{g}_x$ and an $X$-decomposable vector space $V = \bigoplus_{x \in X} V_x$, in the language of Remark~\ref{def:catHopflike}, we define
\[
\GLie \boxtimes_{\XLie} V = \{(p, v) \in \GLie \times V \mid p \in \mf{g}_x, v \in V_{x}\}.
\]
Then a left action of $\GLie$ on $V$ can be viewed as a map $\GLie \boxtimes_{\XLie} V \to V$.
\end{remark}

Now we introduce a generalization of the Lie algebra $\mf{gl}(V)$.

\begin{definition}[$\mf{GL}_X(V)$, $\mf{GL}_{(d_1, \dots, d_n)}(\kk)$]\label{def:GLLie1}
 Let $V=\bigoplus_{x\in X} V_x$ be an $X$-decomposable vector space. The \emph{$X$-general linear Lie algebroid} of $V$, denoted $\mf{GL}(V)$, is the $X$-Lie algebroid $\mf{GL}_X(V) := \bigoplus_{x\in X} \mf{gl}(V_x)$. If $X = \{1, \dots, n\}$ and $V_i$ has dimension $d_i$, then we also use the notation $\mf{GL}_{(d_1, \dots, d_n)}(\kk)$ where $d_1 \leq d_2 \leq \dots \leq d_n$ for $\mf{GL}_X(V)$.
\end{definition}

\begin{definition}[Representation of $\GLie$]\label{def:Liealgebroidrep}
 A \emph{representation} of $\GLie$ is an $X$-decomposable vector space $V= \bigoplus_{x\in X} V_x$ equipped with an $X$-Lie algebroid morphism $\tau: \GLie \to \mf{GL}_X(V)$, called the \emph{representation map} of $V$. We often denote this as $(V,\tau)$. 
 
 Given two $\GLie$-representations $(V,\tau),(W,\omega)$, a morphism of $\GLie$-representations is a family of linear maps $\varphi=\{ \varphi_x:V_x \to W_x \}$ such that $\varphi_x$ is a morphism of $g_x$-representations for every $x\in X$, that is, $\omega_x(v)\varphi_x = \varphi_x \tau_x(v)$ for every $v\in V_x$ and $x\in X$.
\end{definition}

\begin{lemma}[$\rep(\GLie)$]
\label{lem:GLierep}
Let $\rep(\GLie)$ be the category of $\GLie$-representations. This category admits a monoidal structure as follows.
\begin{itemize}
    \item If $(V,\tau),(W,\omega) \in \rep(\GLie)$, then $V \otimes_{\rep(\GLie)} W = \bigoplus_{x\in X} V_x \otimes W_x$; the representation map $\GLie \to \mf{GL}_X(V \otimes_{\rep(\GLie)} W)$ is given by $p \mapsto \tau_x(p) \otimes \omega_x(p)$ for every $p\in \mf{g}_x$ and $x\in X$.
    \item $\unit_{\rep(\GLie)} = \bigoplus_{x\in X} \kk$; the representation map $\GLie \to \mf{GL}_X(\unit_{\rep(\GLie)})$ is given by $p \mapsto \id_{\kk}$, for every $p\in \mf{g}_x$ and $x\in X$.
\end{itemize}
If $\varphi:V \to W$ and $\varphi':V'\to W'$ are morphisms of $\GLie$-representations, then $\varphi \otimes_{\rep(\GLie)} \varphi' = (\varphi_x \otimes \varphi'_x)_{x\in X}$. The associativity constraint is that induced by the tensor product of components, and the left/right unital constraints are given by scalar multiplication.
\end{lemma}

If $V$ is in $\GLiemod$, then $V = \bigoplus_{x \in X} V_x$ with each $V_x \in \mf{g}_x\lmod$, and so we get a Lie algebroid morphism $\tau_x: \mf{g}_x \to \mf{gl}_{x \in X}(V_x)$ which gives an $X$-Lie algebroid morphism $\GLie \to \mf{GL}_X(V)$. Conversely, if $V$ is in $\GLierep$, then the component $\tau_x$ of $\tau: \GLie \to \mf{GL}_X(V)$ gives $V_x$ the structure of a $\mf{g}_x$-module for all $x \in X$, which makes $V$ a $\GLie$-module. The details of how this leads to a monoidal isomorphism of categories can be worked out similarly to the groupoid setting, so we omit them here.

\begin{lemma} \label{lem:GLieequiv}
The categories $\GLiemod$ and $\rep(\GLie)$ are monoidally isomorphic.
\end{lemma}

Now, we change our focus to actions of $X$-Lie algebroids on $\kk$-algebras. 
Let $\GLie$ be an $X$-Lie algebroid. As we did for the category $\GGmod$ in Remark~\ref{rem:GGModuleMaps}, we would like to understand the connection between algebra objects in $\GLiemod$ and $\kk$-algebras that are $\GLie$-modules.

\begin{remark}\label{rem:XLieModuleMaps}
If $A = \bigoplus_{x\in X} A_x$ is an $X$-decomposable $\kk$-algebra, the multiplication map $m_A: A \otimes A \to A$ and unit map $u_A : \kk \to A$ immediately decompose into the respective set of multiplication maps $m_x: A_x \otimes A_x \to A_x$ and unit maps $u_x: \kk \to A_x$ of each $\kk$-algebra $A_x$, for all $x\in X$. If additionally $A$ is a $\GLie$-module, using the notation of Lemma~\ref{lem:Gliemod}, $m_A$ and $u_A$ induce $X$-decomposable linear maps $\underline{m}_A := (m_x)_{x\in X} : A \otimes_{\GLiemod} A \to A$ and $\underline{u}_A := (u_x)_{x\in X} : \unit_{\GLiemod} \to A$, which we call the \emph{monoidal multiplication} and \emph{monoidal unit} of $A$, respectively. It is clear that these maps satisfy the associativity and unital conditions. However, in general, these maps are not necessarily $\GLie$-module morphisms.

Conversely, if $A$ is $\GLie$-module algebra, then it comes equipped with maps in $\GLiemod$ $\underline{m}_A = (m_x)_{x \in X} A \otimes_{\GLiemod} A \to A$ and $\underline{u}_A = (u_x)_{x \in X}: \unit_{\GLiemod} \to A$. These maps extend naturally to maps $m_A: A \otimes A \to A$ (where if $a \in A_x$ and $b \in A_y$ for $x \neq y$, we define $m_A(a \otimes b) = 0$) and $u_A : \kk \to A$ (defined by $\kk\to \unit_{\GLiemod} \to A$ where $\kk \to \unit_{\GLiemod}$ maps $1_{\kk}$ to $(1,1, \cdots, 1)$).
\end{remark}

By definition, it follows immediately that the monoidal multiplication and monoidal unit maps are $\GLie$-module morphisms precisely when they make $A$ a $\GLie$-module algebra. Hence, we have proved the following result.

\begin{lemma}\label{lem:XLieModuleFormulas}
Let $A=\bigoplus_{x\in X} A_x$ be an $X$-decomposable $\kk$-algebra and let $\GLie$ be an $X$-Lie algebroid.
 Then the following statements are equivalent.
	\begin{enumerate}[font=\upshape]
		\item $A$ is a $\GLie$-module algebra, via the monoidal product $m: A \otimes_{\GLiemod} A \to A$ and monoidal unit $u: \unit_{\GLiemod} \to A$ of Remark~\ref{rem:XLieModuleMaps}.
		 
		\item $A$ is a $\GLie$-module, such that
 \begin{gather}
			p \cdot (ab) = a(p \cdot b) + (p \cdot a)b,\label{eq:XLiealg1}\\
			p \cdot 1_{x} = 0,\label{eq:XLiealg2}
		\end{gather}
		for all $p \in \mf{g}_x$ and $a,b\in A_{x}$.
	\end{enumerate}
\end{lemma}

By abuse of notation, we also call an $X$-decomposable algebra satisfying the conditions in Lemma~\ref{lem:XLieModuleFormulas}(b) a \emph{$\GLie$-module algebra}.
To proceed, we introduce a generalization of the Lie algebra $\Der(A)$ consisting of derivations of $A$ with commutator bracket.

\begin{notation}[$\Der_X(A)$]\label{not:GLLie2}
Let $A=\bigoplus_{x\in X} A_x$ be an $X$-decomposable $\kk$-algebra. We denote by $\Der_X(A)$ the $X$-Lie algebroid $\Der_X(A) = \bigoplus_{x\in X} \Der(A_x)$.
\end{notation}

Since each $\Der(A_x)$ is a Lie subalgebra of $\mf{gl}(A_x)$, we have that $\Der_X(A)$ is an $X$-Lie subalgebroid of $\mf{GL}(A)$.

\begin{proposition}\label{prop:Liealgebroidalg}
Let $A = \bigoplus_{x \in X} A_x$ be an $X$-decomposable $\kk$-algebra and let $\GLie = \bigoplus_{x \in X} \mf{g}_x$ be an $X$-Lie algebroid. Then:
\begin{enumerate}[font=\upshape] 
\item The natural action of $\Der(A_x)$ on $A_x$ for all $x \in X$ makes $A$ a $\Der_X(A)$-module algebra. We denote this action $\ell_{\Der_X(A), A}$ and write $p \rt a := \ell_{\Der_X(A), A}(p \boxtimes a)$ for $p \boxtimes a \in \Der_X(A) \boxtimes_{\XLie} A$.
 
\item Suppose that $A$ is a $\GLie$-module algebra via $\ell_{\GLie, A}$ and denote $p \cdot a := \ell_{\GLie, A}(p \boxtimes a)$ for $p \boxtimes a \in \GLie \boxtimes_{\XLie} A$. Then there is a unique $X$-Lie algebroid morphism $\tau: \GLie \to \Der_X(A)$ such that $p \cdot a = \tau(p) \rt a$ for all $p \boxtimes a \in \GLie \boxtimes_{\XLie} A$.

\item Every $X$-Lie algebroid morphism $\tau: \GLie \to \Der_X(A)$ gives $A$ the structure of a $\GLie$-module algebra via $\ell_{\GLie, A}(p \boxtimes a) = \ell_{\Der_X(A), A} (\tau(p) \boxtimes a)$ for all $p \boxtimes a \in \GLie \boxtimes_{\XLie} A$.
\end{enumerate}
Hence, $\Sym_{\XLie}(A) = \Der_X(A)$.
\end{proposition}

\begin{proof}
 (a) The natural action of $\Der(A_x)$ on $A_x$ is given by letting $p \rt a$ be the image of $a$ under the derivation $p$, for $a \in A_x$ and $p \in \Der(A_x)$. It is clear that $A$ is an $\Der_X(A)$-module and by Lemma~\ref{lem:XLieModuleFormulas}, it is clear that in fact $A$ is a $\Der_X(A)$-module algebra. 

 (b) If $A$ is a $\GLie$-module via $\cdot$, then there is an associated sub-$X$-Lie algebroid of $\mf{GL}_X(A)$ (see Lemma~\ref{lem:GLieequiv}). This induces a unique $X$-Lie algebroid morphism $\tau: \GLie \to \mf{GL}_X(A)$. Moreover, if $A$ is a $\GLie$-module algebra, then it satisfies \eqref{eq:XLiealg1} and \eqref{eq:XLiealg2}, which are equivalent to the map $\tau(p)$ being a derivation of $A_x$ for each $p \in \mf{g}_x$. Hence, the image of $\tau$ is contained in the sub-$X$-Lie algebroid $\Der_X(A)$ of $\mf{GL}_X(A)$. Corestricting $\tau$ to a map $\GLie \to \mf{GL}_X(A)$ gives the result.

 (c) If $\tau: \GLie \to \Der_X(A)$ is any $X$-Lie algebroid morphism, then we can define a map $\ell_{\GLie, A}: \GLie \boxtimes_{\XLie} A \to A$ by $\ell_{\GLie, A}(p \boxtimes a) = \ell_{\Der_X(A), A}(\tau(p), a)$ for all $p \boxtimes a \in \GLie \boxtimes_{\XLie} A$. Since $\tau$ is an $X$-Lie algebroid morphism, we see that if $p, q \in \Der(A_x)$ and $a \in A_x$, then
 \begin{align*}
 [p,q]\cdot a &= \tau([p,q])\rt a = [\tau(p), \tau(q)] \rt a = \tau(p) \rt (\tau(q) \rt a) -  \tau(q) \rt( \tau(p) \rt a) \\
 &= p \cdot ( q \cdot a) - q \cdot (p \cdot a),
 \end{align*}
 so $A$ is a $\GLie$-module. Further, for all $p \in \mf{g}_x$ and $a,b \in A_x$, we see
\[
p \cdot (ab) = \tau(p) \rt (ab) = a (\tau(p) \rt b) + (\tau(p) \rt a)b = a (p \cdot b) + (p \cdot a)b
\]
and
\[
p \cdot 1_x = \tau(p) \rt 1_x = 0.
\]
Hence, $A$ is a $\GLie$-module algebra.
\end{proof}

 We recover the following well-known result for actions of Lie algebras, which is the special case of Proposition~\ref{prop:Liealgebroidalg} when $|X| = 1$.

\begin{remark} \label{rem:Lie}
For a $\kk$-algebra $A$, $A$ is a $\Der(A)$-module algebra and for a Lie algebra $\mf{g}$, the following are equivalent.
 \begin{enumerate}[font=\upshape] 
  \item[(i)] $A$ is a $\mf{g}$-module algebra.
 
 \item[(ii)] There exists a Lie algebra morphism $\tau: \mf{g} \to \Der(A)$.
 \end{enumerate}
Hence, $\Sym_{\Lie}(A) = \Der(A)$.
\end{remark}

Recall that by Propostion~\ref{prop:weak-direct}(a), a direct sum of Hopf algebras has a canonical weak Hopf algebra structure. Following \cite[Section 3.2]{NDthesis} we recall a construction that generalizes the notion of the universal enveloping algebra of a Lie algebra to $X$-Lie algebroids.

\begin{definition}[$U_X(\GLie)$, $\XEnvLie$]
\label{def:envelopLiealgebroid}
The \emph{$X$-universal enveloping algebra} of $\GLie$ is the weak Hopf algebra $U_X(\GLie) := \bigoplus_{x \in X} U(\mf{g}_x)$, where each $U(\mf{g}_x)$ is the classical universal enveloping algebra of the Lie algebra $\mf{g}_x$.

The category $\XEnvLie$ of $X$-universal enveloping algebras of $X$-Lie algebroids is the full subcategory of $\XWHA$ consisting of $X$-universal enveloping algebras $U_X(\GLie)$ of $X$-Lie algebroids $\GLie$.
\end{definition}

\begin{remark}\label{rem:universal}
Notice that $U_X(\GLie)=\bigoplus_{x\in X} U(\mf{g}_x)$ is an $X$-decomposable algebra. In particular, it is an $X$-algebra, with idempotent set $\{1_x\}_{x\in X}$, where $1_x$ is the identity element of the universal enveloping algebra $U(\mf{g}_x)$, for each $x\in X$.
\end{remark}

Beyond its similarity to the classical case, the naming for $U_X(\GLie)$ reflects an actual universal property of this object, as demonstrated by the following result. Recall from Example~\ref{ex:LieXAlg}(b) that any $X$-decomposable $\kk$-algebra $A = \bigoplus_{x \in X} A_x$ has an associated $X$-Lie algebroid, denoted by $\mathcal{L}_X(A)$.

\begin{proposition}\label{prop:universal}
Let $A=\bigoplus_{x\in X} A_x$ be an $X$-decomposable $\kk$-algebra and let $\GLie$ be an $X$-Lie algebroid. For any $X$-Lie algebroid morphism $\tau=\{\tau_x\}_{x\in X}:\GLie \to \mathcal{L}_X(A)$ there exists an unique morphism of $X$-decomposable $\kk$-algebras $\tau'=\{\tau'_x\}_{x\in X}:U_X(\GLie) \to A$ such that $\tau'_x \circ \iota_x = \tau_x$, where $\iota_x : \mf{g}_x \to U(\mf{g}_x)$  is the canonical map for each $x\in X$.
\end{proposition}

\begin{proof}
For each $x \in X$, by definition, $\tau_x : \mf{g}_x \to \mathcal{L}(A_x)$ is a Lie algebra map. By the universal property of $U(\mf{g}_x)$, there exists a unique $\kk$-algebra morphism $\tau'_x : U(\mf{g}_x) \to A_x$ such that $\tau'_x \circ \iota_x = \tau_x$. Clearly, $\tau' = \bigoplus_{x \in X} \tau'_x$ is the desired morphism of $ X $-decomposable $\kk$-algebras.
\end{proof}

Consider the category $U_X(\GLie)\lmod$ of \emph{left $U_X(\GLie)$-modules}. Since $U_X(\GLie)$ is a weak Hopf algebra, it follows from Remark~\ref{Hmodboxtimes} that $U_X(\GLie)\lmod$ is a monoidal category. Moreover, given that for a Lie algebra $\mf{g}$ we have a monoidal equivalence between $\mf{g}\lmod$ and $U(\mf{g})\lmod$, the subsequent result becomes evident.

\begin{lemma} \label{lem:UGLiemod}
The categories $\GLie\lmod$ and $U_X(\GLie)\lmod$ are monoidally isomorphic.
\end{lemma}

Consider the following representation category for $U_X(\GLie)$.

\begin{definition}[$\rep(U_X(\GLie))$]
 Let $\rep(U_X(\GLie))$ be the category of {\it representations of $U_X(\GLie)$}, that is, objects are $X$-decomposable vector spaces $V = \bigoplus_{x \in X} V_x$ such that each $V_x$ is a representation of $U(\mf{g}_x)$.
\end{definition}

Observe that $\rep(\GLie) 
\cong 
\textstyle \bigoplus_{x\in X}\rep(\mf{g}_x) 
\cong 
\bigoplus_{x\in X}\rep(U(\mf{g}_x))
 \cong 
\rep(U_X(\GLie))$. Moreover, the monoidal structure of $\rep(\GLie)$ passes to $\rep(U_X(\GLie))$. Consequently, through Lemmas~\ref{lem:GLieequiv} and \ref{lem:UGLiemod}, a chain of monoidal isomorphisms between categories emerges:
\[\rep(U_X(\GLie)) \cong \rep(\GLie) \cong \GLiemod \cong U_X(\GLie)\lmod.\]
In the case when $|X|=1$, we naturally regain the classical monoidal isomorphisms for Lie algebras: $\rep(U(\mf{g})) \cong \rep(\mf{g}) \cong \mf{g}\lmod \cong U(\mf{g})\lmod$.

Recall that $\Lie$ denotes the category of Lie algebras, $\Algcat$ the category of $\kk$-algebras, and $\Hopf$ the category of Hopf algebras. It is well-known that the functor $U(-): \Lie \to \Algcat$, which associates to the Lie algebra $\mf{g}$ its universal enveloping algebra $U(\mf{g})$, is left adjoint to the functor $\mathcal{L}(-):\Algcat \to \Lie$, which associates to a $\kk$-algebra $B$ the Lie algebra $\mathcal{L}(B)$. Since $U(\mf{g})$ is actually a Hopf algebra, we may also view the universal enveloping algebra functor as a functor $U(-):\Lie \to \Hopf$. In this case, it is left adjoint to the functor $P(-):\Hopf \to \Lie$, which associates to the Hopf algebra $H$ its Lie algebra of primitive elements $P(H)$. Using the notation of Examples~\ref{ex:LieXAlg} and \ref{ex:primitiveWHA}, we generalize these two adjunctions to the $X$-Lie algebroid case, recovering the classical setting when $|X| = 1$.

\begin{theorem} \label{thm:adjLie}
Let $X$ be a finite nonempty set.
	\begin{enumerate}[label=\normalfont(\alph*)]
		\item The following functors are well-defined: 
		\[
		\begin{array}{l}
			U_X(-): \XLie \longrightarrow \XAlg \qquad \text{($X$-universal enveloping algebra);}\\ 
			\mathcal{L}_X(-): \XAlg \longrightarrow \XLie \qquad \text{(associated $X$-Lie algebroid).}
		\end{array}
		\]
		Moreover, $U_X(-) \dashv \mathcal{L}_X(-)$, that is, for an $X$-Lie algebroid $\GLie$ and an $X$-algebra $B$, we have a bijection that is natural in each slot: 
		\[
	\Hom_{\XLie}(\GLie, \mathcal{L}_X(B) ) \; \cong \; \Hom_{\XAlg}( U_X(\GLie) , B).
		\]
		
		\item 
	 The following functors are well-defined:
		\[\qquad \qquad
		\begin{array}{l}
			U_X(-): \XLie \longrightarrow \XWHA \qquad \text{($X$-universal enveloping algebra);}\\ 
			P_X(-): \XWHA \longrightarrow \XLie \qquad \text{($X$-Lie algebroid of primitive elements).}
		\end{array}
		\]
		Moreover, $U_X(-) \dashv P_X(-)$, that is, for an $X$-Lie algebroid $\GLie$ and an $X$-weak Hopf algebra $H$, we have a bijection that is natural in each slot: 
		\[
	\Hom_{\XLie}(\GLie, P_X(H)) \; \cong \; \Hom_{\XWHA}( U_X(\GLie), H).		\]
	In particular, $P_X(U_X(\GLie)) = \GLie$. 
	\end{enumerate}
\end{theorem}

\begin{proof}
    (a) First, we prove that $U_X(-): \XLie \to \XAlg$ is a functor. Let $\GLie = \bigoplus_{x\in X} \mf{g}_x$ and $\GLie' = \bigoplus_{x\in X} \mf{g}'_x$ be $X$-Lie algebroids. By Remark~\ref{rem:universal}, $U_X(-)$ sends objects to objects. Also, if $\tau= \{ \tau_x : \mf{g}_x \to \mf{g}'_x  \}_{x\in X}$ is an $X$-Lie algebroid morphism, using the classical adjunction for each $x\in X$, we can uniquely extend the Lie algebra morphism $\tau_x : \mf{g}_x \to \mf{g}'_x$ to an algebra morphism $U(\tau_x) : U(\mf{g}_x) \to U(\mf{g}'_x)$. Hence, $U_X(\tau):=\bigoplus_{x\in X} U(\tau_x) : U_X(\GLie) \to U_X(\GLie')$ is an $X$-algebra map. Hence, $U_X(-)$ also sends morphisms to morphisms. It is clear that $U_X(-)$ respects compositions and identities.

    Secondly, we check that $\mathcal{L}_X(-): \XAlg \to \XLie$ is a functor. From Example~\ref{ex:LieXAlg} it follows that $\mathcal{L}_X(-)$ sends objects to objects. Moreover, if $\psi:A \to B$ is an $X$-algebra map, we define $\psi_x: A_x \to B_x$ as the restriction $\psi_x:=\psi|_{A_x}$. Here, the $\kk$-algebras $A_x$ and $B_x$ are defined as in Example~\ref{ex:LieXAlg}. For each $x\in X$, it is clear that $\psi_x$ is a $\kk$-algebra map, which induces a Lie algebra map $\mathcal{L}(\psi_x): \mathcal{L}(A_x) \to \mathcal{L}(B_x)$. Therefore $\mathcal{L}_X(\psi):=\bigoplus_{x\in X} \mathcal{L}(\psi_x) : \mathcal{L}_X(A) \to \mathcal{L}_X(B)$ is an $X$-Lie algebroid morphism. Also, by construction, $\mathcal{L}_X(-)$ respects compositions and identities.

    Now, for any $X$-Lie algebroid $\GLie = \bigoplus_{x\in X} \mf{g}_x$ and any $X$-algebra $B$, we want to display a bijection
	\begin{equation*}
		\Hom_{\XLie}(\GLie, \mathcal{L}_X(B) )  \cong \Hom_{\XAlg}( U_X(\GLie) , B).
	\end{equation*}
	If $\phi: \GLie \to\mathcal{L}_X(B)$ is an $X$-Lie algebroid morphism of the form $\phi=\{ \phi_x: \mf{g}_x \to \mathcal{L}(B_x) \}_{x\in X}$, for each $x\in X$ we use the classical adjunction to construct $\kk$-algebra maps $\phi'_x: U(\mf{g}_x) \to B_x$. Since each $B_x$ is contained in $B$, we can construct $\phi':= \bigoplus_{x\in X} \phi'_x : U_X(\GLie) \to B$, which is clearly an $X$-algebra morphism. Hence, we have constructed the assignment
\[
		\Phi_{\GLie,B}: \Hom_{\XLie}(\GLie, \mathcal{L}_X(B) ) \longrightarrow \Hom_{\XAlg}( U_X(\GLie) , B), \quad
		\phi \mapsto \phi'.
\]
On the other hand, given an $X$-algebra map $\psi \in \Hom_{\XAlg}( U_X(\GLie) , B)$, for each $x\in X$ consider the restriction $\psi|_{U(\mf{g}_x)}: U(\mf{g}_x) \to B$, which is an algebra map. Since for any $p\in \mf{g}_x$ we have $\psi|_{U(\mf{g}_x)}(p) = \psi(1_x p 1_x) = e_x^B\psi(p) e_x^B \in B_x$, each restriction is actually of the form  $\psi|_{U(\mf{g}_x)}: U(\mf{g}_x) \to B_x$. As above, using the classical adjunction we have, for each $x\in X$, a Lie algebra map $\psi|_{\mf{g}_x}: \mf{g}_x \to \mathcal{L}(B_x)$. Then $\psi|_{\GLie}:= \{\psi|_{\mf{g}_x} \}_{x\in X}: \GLie \to \mathcal{L}_X(B)$ is an $X$-Lie algebroid map. Hence, we have constructed the assignment
\[
\Psi_{\GLie,B}: \Hom_{\XAlg}( U_X(\GLie) , B) \longrightarrow \Hom_{\XLie}(\GLie, \mathcal{L}_X(B) ), \quad \psi \mapsto \psi|_{\GLie}.
\]
The following calculations show that these assignments are mutually-inverse. If $p\in \mf{g}_x$, then
	\begin{gather*}
		[( \Psi_{\GLie,B} \circ \Phi_{\GLie,B} ) ( \phi )] (p) = \Psi_{\GLie,B} (\phi') (p) = \phi'|_{\GLie} (p) =\phi(p),\\
		[( \Phi_{\GLie,B} \circ \Psi_{\GLie,B}) (\psi)] (p) = \Phi_{\GLie,B} (\psi|_{\GLie})(p) = (\psi|_{\GLie})' (p) = \psi(p).
	\end{gather*}
Finally, the bijection is natural since for any $X$-Lie algebroid $\varphi: \GLie' \to \GLie$, any $X$-algebra map $f: B \to B'$, and all $\psi \in \Hom_{\XAlg}(U_X(\GLie), B)$ one can check the following equality:
\begin{equation*}
(	\Phi_{\GLie',B'} \circ \Hom_{\XAlg}(U_X( \varphi),f )) (\psi) = (\Hom_{\XLie} ( \varphi, \mathcal{L}_X(f) ) \circ \Phi_{\GLie,B} )( \psi ).
\end{equation*}
Here, $\Hom_{\XAlg}(U_X(\varphi),f ) : \Hom_{\XAlg}(U_X(\GLie), B) \to \Hom_{\XAlg}(U_X(\GLie'), B') $ is given by composition, i.e., $\psi \mapsto f \circ \psi \circ U_X( \varphi )$; a similar notion holds for $\Hom_{\XLie} ( \varphi, \mathcal{L}_X(f) ) $.

\noindent (b) First, we show that $U_X(-): \XLie \to \XWHA$ is a functor. If $\GLie=\bigoplus_{x\in X} \mf{g}_x$ is an $X$-Lie algebroid, then $U_X(\GLie)=\bigoplus_{x\in X} U(\mf{g}_x)$ is a weak Hopf algebra. The set $\{1_x\}_{x \in X}$ of identities $1_x \in U(\mf{g}_x)$ are idempotent elements of $U_X(\GLie)$ satisfying $1_{U_X(\GLie)} = \sum_{x \in X} 1_x$ and $\Delta(1_x) = 1_x \otimes 1_x$ for each $x \in X$, and so $U_X(\GLie)$ is an object of $\XWHA$. If $\tau = \{ \tau_x \}_{x\in X}: \GLie \to \GLie'$ is an $X$-Lie algebroid morphism, then by part (a), $U_X(\tau): U_X( \GLie) \to U_X( \GLie')$ is a map of $X$-algebras. Moreover, by the second classical adjunction each $U(\tau_x) : U(\mf{g}_x) \to U(\mf{g}_x')$ defined in part (a) is, in fact, a Hopf algebra morphism. So it is clear that $U_X(\tau)$ is a morphism in $\XWHA$.

Now, we show that $P_X(-): \XWHA \to \XLie$ is a functor. By Example~\ref{ex:primitiveWHA}, if $H$ is an $X$-weak Hopf algebra then $P_X(H) \in \XLie$. Moreover, if $f: H \to H'$ is a morphism in $\XWHA$, then $P_X(f)$ is simply the restriction of $f$ to $P_X(H)$. Using the notation of Example~\ref{ex:primitiveWHA}, it is clear for any $x\in X$ that if $h\in P_x(H)$ then $f(h) \in P_x(H')$, so $P_X(f)$ is indeed a morphism in $\XLie$. Also, by construction, $P_X(-)$ respects compositions and identities.

Finally, for any $\GLie \in \XGrpd$ and $H \in \XWHA$, we wish to exhibit a bijection
\[
\Hom_{\XLie}(\GLie, P_X(H)) \cong \Hom_{\XWHA} (U_X( \GLie), H).
\]
If $\phi : \GLie \to P_X(H) $ is an $X$-Lie algebroid morphism, then $U_X( \phi): U_X( \GLie) \to U_X( P_X(H))$ is a morphism of $X$-weak Hopf algebras. Since the Lie bracket in each component of $P_X(H)$ was defined as the commutator of $H$, the map $\eta_H: U_X( P_X(H)) \to H$ induced by the inclusion $P_X(H) \subseteq H$ is a unital $\kk$-algebra map. Also, the elements of $P_X(H)$ are $x$-primitive in $H$, for some $x\in X$, so $\eta_H$ is also a $\kk$-coalgebra map. Since taking opposites in $P_X(H)$ coincides with applying the antipode of $H$, $\eta_H$ intertwines the antipodes of $U_X(P_X(H))$ and $H$. Therefore, $\eta_H \circ U_X( \phi) \in \Hom_{\XWHA}(U_X(\GLie), H)$, and we have the assignment
\[
\Phi_{\GLie, H}: \Hom_{\XLie}(\GLie, P_X(H)) \longrightarrow \Hom_{\XWHA} (U_X(\GLie), H), \quad \phi \mapsto \eta_H \circ U_X(\phi).
\]
Conversely, if $\psi: U_X(\GLie) \to H$ is a map in $\XWHA$, then $P_X(\psi): P_X(U_X(\GLie)) \to P_X(H)$ is an $X$-Lie algebroid morphism. Note that $P_X(U_X (\GLie)) = \GLie$. Indeed, since we assume that $\kk$ has characteristic 0, for each $x\in X$ we have $P(U(\mf{g}_x))=\mf{g}_x$. We have therefore constructed an assignment
\[
\Psi_{\GLie, H}:\Hom_{\XWHA} (U_X(\GLie), H) \longrightarrow \Hom_{\XLie}(\GLie, P_X(H)), \quad \psi \mapsto P_X(\psi) \circ \nu_{\GLie},
\]
where $\nu_{\GLie}: \GLie \to P_X(U_X(\GLie))$ is the identity map. It is straightforward to check that $\Phi_{\GLie, H}$ and $\Psi_{\GLie, H}$ are mutually-inverse and the naturality of the adjunction.
\end{proof}

\begin{remark}\label{rem:Liealgebroid-bij}
    Let $A=\bigoplus_{x\in X} A_x$ be an $X$-decomposable $\kk$-algebra. In Theorem~\ref{thm:adjLie}(b), by taking $H=U_X(\Der_X(A))$, we have the bijection
\[
 \Hom_{\XLie}(\GLie, \Der_X(A)) \cong \Hom_{\XWHA}(U_X(\GLie), U_X(\Der_X(A)) ). 
\]
\end{remark}

Finally, we will construct the symmetry object $\Sym_{\XEnvLie}(A)$ over any $X$-decomposable $\kk$-algebra. Note that since $U_X(\GLie)$ is a weak Hopf algebra, by a $U_X(\GLie)$-module algebra, we mean a $\kk$-algebra satisfying condition \eqref{eq.Hmodalgebra} in Proposition~\ref{def:modulealgoverWHA}.

\begin{proposition}\label{prop:U(GLie)}
Let $A = \bigoplus_{x \in X} A_x$ be an $X$-decomposable $\kk$-algebra and let $\GLie$ be an $X$-Lie algebroid.
\begin{enumerate}[font=\upshape] 
\item The action of $U_X(\Der_X(A))$ on $A$ induced by the action of $\Der_X(A)$ on $A$ makes $A$ a $U_X(\Der_X(A))$-module algebra. For the action map $\ell_{U_X(\Der_X(A)), A}$, we use the notation $h \rt a$ for $\ell_{U_X(\Der_X(A)), A}(h \otimes a)$, where $h \in U_X(\Der_X(A))$ and $a \in A$.

\item Suppose that $A$ is a $U_X(\GLie)$-module algebra via $\ell_{U_X(\GLie), A}$, where we write $h \cdot a$ for $\ell(h \otimes a)$ for $h \in U_X(\GLie)$ and $a \in A$. Then there is a unique $X$-weak Hopf algebra morphism $\widetilde{\tau}: U_X(\GLie) \to U_X(\Der_X(A))$ such that $h \cdot a = \widetilde{\tau}(h) \rt a$ for all $h \in U_X(\GLie)$ and $a \in A$.

\item Every $X$-weak Hopf algebra morphism $\widetilde{\tau}: U_X(\GLie) \to U_X(\Der_X(A))$ gives $A$ the structure of a $U_X(\GLie)$-module algebra via $\ell_{U_X(\GLie),A}(h \otimes a) = \ell_{U_X(\Der_X(A)), A}(\widetilde{\tau}(h) \otimes a)$ for all $h \in U_X(\GLie)$ and $a \in A$ (that is, $h \cdot a = \widetilde{\tau}(h) \rt a$).
\end{enumerate}
Hence, $\Sym_{\XEnvLie}(A) = U_X(\Der_X(A))$.
\end{proposition}

\begin{proof}
(a) It is clear that each $A_x$ is a $\Der(A_x)$-module algebra. Hence, by the classical theory, each $A_x$ is a $U(\Der(A_x))$-module algebra. Hence, by Proposition~\ref{prop:hs=ht}(c), $A$ is a $U_X(\Der_X(A))$-module algebra.

(b, c) By the same argument as above, $A$ is a $U_X(\GLie)$-module algebra if and only if each $A_x$ is a $U(\mf{g}_x)$-module algebra if and only if each $A_x$ is a $\mf{g}_x$-module algebra. By Lemma~\ref{lem:XLieModuleFormulas}, this happens if and only if $A$ is a $\GLie$-module algebra. 

By Proposition~\ref{prop:Liealgebroidalg}, this happens if and only if the action of $\GLie$ factors through a unique $X$-Lie algebroid morphism $\tau: \GLie \to \Der_X(A)$. By Remark~\ref{rem:Liealgebroid-bij}, this if and only if there exists an $X$-weak Hopf algebra morphism $\widetilde{\tau}: U_X(\GLie) \to U_X(\Der_X(A))$.

Note that since $\XEnvLie$ is a full subcategory of $\XWHA$, the morphisms in $\XEnvLie$ are simply $X$-weak Hopf algebra morphisms. As a consequence, (a), (b), and (c) above imply that $\Sym_{\XEnvLie}(A) = U_X(\Der_X(A))$, as desired.
\end{proof}

 When $|X| = 1$, we recover the following well-known result for actions of enveloping algebras of Lie algebras.

\begin{remark} \label{rem:U(Lie)}
Let $A$ be a $\kk$-algebra. Then $A$ is a $U(\Der(A))$-module algebra and for a Lie algebra $\mf{g}$, the following are equivalent.
 \begin{enumerate}[font=\upshape] 
 
 \item[(i)] $A$ is a $U(\mf{g})$-module algebra.
 
 \item[(ii)] There exists a Hopf algebra morphism $\widetilde{\tau}: U(\mf{g}) \to U(\Der(A))$.
 \end{enumerate}
Hence, $\Sym_{\EnvLie}(A) = U(\Der(A))$.
\end{remark}

%%%%%%%%%%%%%%%%%%%%%%%%%%%%%%%%%%%%%%%%%%%%%%%%%%%%%%%%%%%%%%%%%%%%%
\section{Actions of cocommutative weak Hopf algebras on algebras} \label{sec:cocomweak}

In this section our goal is to study symmetries in the category $\CC = \XCocomWHA$ of cocommutative $X$-weak Hopf algebras (see Notation~\ref{not:cats}). The main result, which identifies the symmetry object $\Sym_{\CC}(A)$, is Theorem~\ref{thm:cocomweakHopf}.
First, we recall the definition of the smash product over a weak Hopf algebra, which generalizes the classical construction. Throughout, we use that the antipode $S$ of a weak Hopf algebra $H$ is bijective between the counital subalgebras $H_t$ and $H_s$ \cite[Proposition 2.3.2]{NV}.

\begin{definition}[{\cite[Section 4.2]{NV}}]\label{def:smashpro}
Let $H$ be a weak Hopf algebra and let $A$ be an $H$-module algebra. The \emph{smash product} algebra $A \# H$ is defined as the $\kk$-vector space $A \otimes_{H_t} H$, where $A$ is a right $H_t$-module via $a \cdot z = S^{-1}(z) \cdot a = a (z \cdot 1_A)$, for $a \in A$ and $z \in H_t$.
Multiplication in $A \# H$ is defined by
$(a \# h)(b \# g) = a(h_1 \cdot b) \# h_2 g.$ Here $a\#h$ denotes a coset representative of $a\otimes_{H_t} h$ for $a\in A$ and $h\in H$.
\end{definition}

To state Nikshych's generalization of the Cartier--Gabriel--Kostant--Milnor--Moore theorem \cite{NDthesis}, we first describe the weak Hopf algebra structure on this smash product.

\begin{definition} \label{def:XdecompHopf}
Let $H$ be a weak Hopf algebra. We say that $H$ is an \emph{$X$-decomposable weak Hopf algebra} if there exists a family $\{H_x\}_{x\in X}$ of weak Hopf algebras (some of which may be $0$) such that $H=\bigoplus_{x\in X} H_x$ as weak Hopf algebras.
\end{definition}

\begin{remark}
\begin{enumerate}
 \item Every $X$-decomposable weak Hopf algebra is an $X$-decomposable $\kk$-algebra (Definition~\ref{def:XdecompA}) whose $X$-decomposition also respects the weak Hopf structure.
 \item Not every $X$-weak Hopf algebra is an $X$-decomposable weak Hopf algebra. For example, for an $X$-groupoid $\GG$, the groupoid algebra $\kk \GG$ is in $\XWHA$ (Example~\ref{ex:groupoidalgXWHA}). However, if the idempotents $\{e_x\}_{x\in X}$ in $\kk\GG$ are not central (see Remark~\ref{rem:localid}), then $\kk\GG$ is not an $X$-decomposable algebra and thus not a $X$-decomposable weak Hopf algebra.
 \item Conversely, 
 not every $X$-decomposable weak Hopf algebra is as an $X$-weak Hopf algebra. For instance, if we have an $X$-decomposable weak Hopf algebra $H= \bigoplus_{x\in X} H_x$, where some of the $H_x$ are proper weak Hopf algebras, the idempotents $1_x$ in $H$ may not necessarily be grouplike elements as required by Definition~\ref{def:xxx}. (For example, the weak Hopf algebra from \cref{ex:hay} has non-grouplike idempotents for $N\geq 2$.) Consequently, $H$ would not satisfy the conditions to be classified as an $X$-weak Hopf algebra.
 
 \item For every $X$-Lie algebroid $\GLie= \bigoplus_{x \in X} \mf{g}_x$, the $X$-universal enveloping algebra $U_X(\GLie)$ is an $X$-decomposable weak Hopf algebra and also is an $X$-weak Hopf algebra since each $U(\mf{g}_x)$ is a Hopf algebra.
\end{enumerate}
\end{remark}

\begin{definition}[$\XAutH$]\label{def:autgroupoid3}
 
Let $H=\bigoplus_{x\in X} H_x$ be an $X$-decomposable weak Hopf algebra. We define $\XAutH$, the \emph{$X$-weak bialgebra automorphism groupoid} of $H$, as follows: 
\begin{itemize}
 \item the object set is $X$,
 %s are the weak Hopf algebras $\{H_x\}_{x\in X}$, 
 \item for any $x,y\in X$, $\Hom_{\XAutH}(H_x,H_y)$ is the space of weak bialgebra isomorphisms between the weak bialgebras $H_x$ and $H_y$.
\end{itemize}
\end{definition}

Clearly, $\XAutH$ is a subgroupoid of $\Aut_{\XAlg}(H)$, as defined in Definition~\ref{def:autgroupoid2}.

\begin{proposition}\label{prop:smashproduct}
Let $\GG$ be an $X$-groupoid and $H =\bigoplus_{x\in X} H_x$ be an $X$-decomposable weak Hopf algebra. If $H$ is a $\kk\GG$-module algebra via the action $\cdot$, then the following statements are equivalent.
\begin{enumerate}[font=\upshape]
\item For every $g\in \GG_1$ and $a\in H_{s(g)}$ the following relations hold:
\begin{gather}
 \Delta_H(g\cdot a) = g\cdot a_1 \otimes g\cdot a_2, \label{eq:smash1}\\
 \varepsilon_H(g\cdot a)=\varepsilon_{H_{s(g)}}(a).\label{eq:smash2}
\end{gather}
\item There exists an $X$-groupoid morphism $\pi: \GG \to \Aut_{\XWBA}(H)$.
\end{enumerate}
Moreover, in this case the smash product algebra $H \# \kk\GG$ becomes a weak Hopf algebra with operations given by:
\begin{equation}\label{eq:smashproduct}
 \Delta(a \# g) = a_1 \# g \; \otimes \; a_2 \# g, \quad \varepsilon(a \# g) = \varepsilon_H(a), \quad S(a \# g) = (1_H \# g^{-1})(S_H(a) \# 1_{\kk\GG}),
\end{equation}
for all $g\in\GG_1$ and $a\in H_{t(g)}$.
\end{proposition}

\begin{proof}
Since $H$ is a $\kk \GG$-module algebra, Proposition~\ref{prop:kGrpd} guarantees the existence of an $X$-groupoid morphism $\pi: \GG \to \Aut_{\XAlg}(H)$  such that $g\cdot a=\pi(g)(a)$ for all $g\in \GG_1$ and $a\in H_{s(g)}$; note that if $a\in H_x$ with $x\neq s(g)$ then $g\cdot a =0$ (see Remark \ref{rem:groupoidmod}). Equations \eqref{eq:smash1} and \eqref{eq:smash2} are equivalent to the statement that for all $g\in\GG_0$ the $\kk$-algebra map $\pi(g): H_{s(g)} \to H_{t(g)}$ is a $\kk$-coalgebra map. But this is the same as requiring the image of $\pi$ to be contained in the subgroupoid $\Aut_{\XWBA}(H)$. This proves that (a) and (b) are equivalent.

Now, it is clear that $H \# \kk\GG$ is both a $\kk$-algebra and a $\kk$-coalgebra, and the maps of \eqref{eq:smashproduct} are well-defined. Let $g,h \in \GG_1$ such that $t(h)=s(g)$. Suppose that $a\in H_{t(g)}$ and $b\in H_{t(h)}$, so we have
\begin{align*}
 \Delta((a\# g)(b\# h)) &= a_1 (g \cdot b)_1 \# gh \; \otimes \; a_2(g\cdot b)_2 \# gh \overset{\eqref{eq:smash1}}{=} a_1 (g \cdot b_1) \# gh \; \otimes \; a_2(g\cdot b_2) \# gh\\
 &= (a_1 \# g \; \otimes \; a_2 \# g)(b_1 \# h \; \otimes \; b_2 \# h) = \Delta(a\# g)\Delta(b\# h).
\end{align*}
Otherwise, both sides of the above equation are zero. Hence, $\Delta$ is multiplicative.  
If $g,h,l\in \GG_1$ satisfy $t(h)=s(g)$ and $t(l)=s(h)$, then for $a\in H_{t(g)}$, $b \in H_{t(h)}$, $c\in H_{t(l)}$ we have
\begin{align*}
 \varepsilon((a\# g)(b\# h)(c\# l)) &= \varepsilon( a(g\cdot b)(gh\cdot c) \# ghl ) = \varepsilon_H(a(g\cdot b)(gh\cdot c))\\
 &\overset{\eqref{eq:smash1}}{=} \varepsilon_H( a (g\cdot b_1) ) \varepsilon_H( (g\cdot b_2)(gh\cdot c )) = \varepsilon_H( a (g\cdot b_1) ) \varepsilon_H( g\cdot (b_2(h\cdot c)))\\
 &\overset{\eqref{eq:smash2}}{=} \varepsilon_H( a (g\cdot b_1) ) \varepsilon_{H_{s(g)}}( b_2(h\cdot c)) = \varepsilon( a(g\cdot b_1) \# gh)\varepsilon(b_2(h\cdot c) \# hl)\\
 &= \varepsilon((a\# g)(b\# h)_1)\varepsilon((b\# h)_2(c\# l)),
\end{align*}
and similarly, $\varepsilon((a\# g)(b\# h)(c\# l)) = \varepsilon((a\# g)(b\# h)_2)\varepsilon((b\# h)_1(c\# l))$,
so $\varepsilon$ is weakly multiplicative. Let $1:=1_H$, $1_x = 1_{H_x}$ for all $x\in X$, and $1_{\kk\GG}=\sum_{x\in X} e_x$. 
Then
\begin{align*}
 & (\Delta(1 \# 1_{\kk\GG}) \; \otimes \; 1 \# 1_{\kk\GG})(1 \# 1_{\kk\GG} \; \otimes \; \Delta(1 \# 1_{\kk\GG})) \\
 &= \left( \sum_{x\in X} \Delta(1 \# e_x) \; \otimes \; \sum_{y\in X} 1 \# e_y \right)\left( \sum_{z\in X} 1 \# e_z \; \otimes \; \sum_{w\in X} \Delta(1 \# e_w) \right) \\
 &= \left( \sum_{x,y\in X} 1_1 \# e_x \; \otimes \; 1_2 \# e_x \; \otimes \; 1 \# e_y \right)\left( \sum_{z,w\in X} 1 \# e_z \; \otimes \; 1_1' \# e_w \; \otimes \; 1_2' \# e_w \right)\\
 &= \sum_{x,y,z,w\in X} 1_1 \# e_x \cdot 1 \# e_z \; \otimes \; 1_2 \# e_x \cdot 1_1' \# e_w \; \otimes \; 1 \# e_y \cdot 1_2' \# e_w\\
 &= \sum_{x,y,z,w\in X} 1_1(e_x \cdot 1) \# e_xe_z \; \otimes \; 1_2(e_x \cdot 1_1') \# e_xe_w \; \otimes \; e_y \cdot 1_2' \# e_ye_w\\
 &\overset{(*)}{=} \sum_{x\in X} (1_x)_1 \# e_x \; \otimes \; (1_x)_2 (1_x')_1 \# e_x \; \otimes \; (1_x')_2 \# e_x\\
 &\overset{(**)}{=} \sum_{x\in X} (1_x)_1 \# e_x \; \otimes \; (1_x)_2 \# e_x \; \otimes \; (1_x)_3 \# e_x \\
 &\overset{(***)}{=} \sum_{x\in X} 1_1 \# e_x \; \otimes \; 1_2 \# e_x \; \otimes \; 1_3 \# e_x\\
 &= \sum_{x\in X} \Delta( 1_1 \# e_x ) \; \otimes \; 1_2\# e_x =\Delta^2(1 \# 1_{\kk\GG}).
\end{align*}
In $(*)$ we used that $e_x\cdot 1=1_x$, in $(**)$ that each $H_x$ is a weak Hopf algebra, while in $(***)$ that $(1_y)_1\#e_x=0$ if $x\neq y$; recall that $(\kk\GG)_t = \bigoplus_{x\in X} \kk e_x$. This proves that $\Delta$ is weakly comultiplicative and thus the smash product $H \# \kk\GG$ is a weak bialgebra. 

Next, we will use the identity $\Delta(1_H\#1_{\kk \GG})=\sum_{x\in X}(1_x)_1 \# e_x\otimes (1_x)_2\#e_x$ to prove the antipode condition. Indeed, if $g\in\GG_1$ and $a\in H_{t(g)}$, then
\begin{align*}
 S(a_1 \# g)(a_2 \# g) &= (1 \# g^{-1})(S_H(a_1)\# 1_{\kk\GG})(a_2 \# g) = g^{-1} \cdot (S_H (a_1)a_2 ) \# e_{s(g)}\\
 &= g^{-1} \cdot (\varepsilon_H)_s(a) \# e_{s(g)} = \sum_{x\in X}(g^{-1} \cdot (1_x)_1 \# e_{s(g)}) \varepsilon_H(a(1_x)_2)\\
 &\overset{(*)}{=}((1_{s(g)})_1\#e_{s(g)})\varepsilon_H(a(1_{t(g)})_2)\overset{(**)}{=}((1_{s(g)})_1\#e_{s(g)})\varepsilon_H(a(g\cdot(1_{s(g)})_2))\\
&\overset{(***)}{=}((1_{s(g)})_1\#e_{s(g)})\varepsilon_H(a(g\cdot (1_{s(g)})_2))\varepsilon_{\kk\GG}(ge_{s(g)})\\
&=((1_{s(g)})_1\#e_{s(g)})\varepsilon_{H\#\kk\GG}(a(g\cdot (1_{s(g)})_2)\#ge_{s(g)})\\
%&=(1_{s(g)}\#e_{s(g)})\varepsilon_{H\#\kk\GG}((a\#g)(1_{s(g)}\#e_{s(g)}))\\
&\overset{(****)}{=}\sum_{x\in X}((1_{x})_1\#e_{x})\varepsilon_{H\#\kk\GG}((a\#g)((1_{x})_2\#e_{x}))\\
&=\varepsilon_s(a\#g). 
\end{align*}

In $(*)$ we used that $g^{-1}\cdot 1_{x}=\delta_{x, t(g)}1_{s(g)}$, in $(**)$ that $g\cdot1_{s(g)}=1_{t(g)}$, in $(***)$ that $\varepsilon_{\kk\GG}(ge_{s(g)})=1$, while in $(****)$ that $(a\#g)((1_x)_2\#e_x)=\delta_{x,s(g)}a(g\cdot (1_{s(g)})_2)\#ge_{s(g)}$. 
Similarly, one can prove $(a_1 \# g)S(a_2 \# g) = \varepsilon_t(a\# g)$.  Finally, we have
\begin{align*}
 S( (a\# g)_1 ) (a \# g)_2 S( (a\# g)_3) &= (1 \# g^{-1})(S_H(a_1)\# 1_{\kk\GG})(a_2 \# g)(1 \# g^{-1})( S_H(a_3) \# 1_{\kk\GG})\\
 &= (g^{-1} \cdot (S_H(a_1)a_2) \# 1_{\kk\GG})(g^{-1}\cdot S_H(a_3) \# g^{-1})\\
 &= g^{-1} \cdot (S_H(a_1)a_2 S_H(a_3)) \# g^{-1} \overset{(*)}{=} g^{-1} \cdot S_H(a) \# g^{-1}\\
 &= (1 \# g^{-1})(S_H(a)\# 1_{\kk\GG}) = S(a\# g).
\end{align*}
Here, we used in $(*)$ that $S_H$ is an antipode for $H$. Therefore, $H \# \kk\GG$ is a weak Hopf algebra.
\end{proof}

In particular, if in the $X$-decomposition of $H$, each $H_x$ is a Hopf algebra, then the next result shows that $H \#\kk \GG$ is actually an $X$-weak Hopf algebra.

\begin{lemma} \label{lem:Hsbase}
Let $\GG$ be an $X$-groupoid and $W =\bigoplus_{x\in X} W_x$ be an $X$-decomposable weak Hopf algebra which is a $\kk\GG$-module algebra. Suppose further that for all $x \in X$, $W_x$ is actually a Hopf algebra. Then the smash product $K:=W \# \kk \GG$ is an $X$-weak Hopf algebra. 
 The complete set of grouplike orthogonal idempotents of $K$ is $\{f_x:=1_{W_{x}}\# e_x\}_{x \in X}$. Consequently,
 \begin{equation}\label{eq:vapt}
     (\varepsilon_K)_t(w_x\#g)=(\varepsilon_W)_t(w_x)\#(\varepsilon_{\kk \GG})_t(g)
 \end{equation}
 for any $w\in W_x$ and $g\in \GG$.
\end{lemma}

\begin{proof}
From Proposition~\ref{prop:smashproduct} it follows that $K$ is a weak Hopf algebra. Moreover, since $W$ is a $\kk \GG$-module algebra and $\kk \GG_0=\kk\{e_x\mid x\in X\}$, then $W$ also is a right $\kk \GG_0$-module via $p\cdot e_x=S^{-1}(e_x)\cdot p=p \,(e_x\cdot 1_H)$ for all $x\in \GG_0=X$ and $p\in W$; see Definition~\ref{def:smashpro}. Note that
\begin{equation*} 
\Delta(1_K) \;=\; \Delta(1_W \otimes_{\kk \GG_0} 1_{\kk \GG}) 
\;=\; \textstyle \sum_{x, y\in X} (1_{W_{x}}\otimes_{\kk \GG_0} e_y)\otimes (1_{W_{x}}\otimes_{\kk \GG_0} e_y).
\end{equation*}

For $x,y \in X$, we have
\begin{equation}
\begin{aligned}\label{eq:Hbase}
1_{W_x} \#e_y 
&=1_{W_x} \otimes_{\kk \GG_0} e_y
\overset{(*)}{=} 1_{W_x}\cdot e_y\otimes_{\kk \GG_0} e_y\\
&= 1_{W_{x}}(e_y\cdot 1_W)\otimes_{\kk \GG_0} e_y
\overset{\text{Rem.~\ref{rem:groupoidmod}}}{=}\delta_{x,\,y}1_{W_x} \otimes_{\kk \GG_0} e_x\\
&=\delta_{x,\,y}1_{W_x} \#e_x.
\end{aligned}
\end{equation}
In $(*)$ we used that $e_y^2=e_y$. Let $f_x:=1_{W_x}\otimes_{\kk \GG_0} e_x$, for each $x \in X.$ Therefore $1_K=\sum_{x\in X}f_x$ and thus $\Delta(1_K)=\sum_{x\in X}f_x\otimes f_x$.  Now the result follows from Proposition~\ref{prop:hs=ht}(a). 

Consequently, for any $w_x\in W_x$ and $g\in \GG$, since $e_y\cdot w_x=\delta_{x,y} w_x$ by Remark~\ref{rem:groupoidmod}, we have
\begin{equation*}
(\varepsilon_K)_t(w_x\#g) \overset{\eqref{eq:Hbase}}{=}\sum_{y\in X}\varepsilon_{W}(1_{W_y}\cdot (e_y\cdot w_x))1_{W_y}\#\varepsilon_{\kk\GG}(e_y\cdot g)e_y=(\varepsilon_W)_t(w_x)\#(\varepsilon_{\kk \GG})_t(g),
\end{equation*}
which proves \eqref{eq:vapt}.
\end{proof}

Two relevant examples of the smash product are the following.

\begin{example}\label{exa:smashproduct}
\begin{enumerate}[label=\normalfont(\alph*)]
    \item  Let $\GG$ be an $X$-groupoid acting by conjugation on an $X$-Lie algebroid $\GLie$. This induces an action of the groupoid algebra $\kk \GG$ on the $X$-universal enveloping algebra $U_{X}(\GLie)$ satisfying the conditions of Proposition~\ref{prop:smashproduct}(a). Thus $U_{X}(\GLie) \# \kk \GG$ is an $X$-weak Hopf algebra as a particular case of Lemma~\ref{lem:Hsbase}.
    \item Let $G$ be a group and $H$ be a Hopf algebra. If $H$ is a $\kk G$-module algebra, then Proposition~\ref{prop:smashproduct} translates into the following well-known result: if there is a group morphism from $G$ to the group of bialgebra automorphisms of $H$, $\pi: G \to \Aut_{\Bialg}(H)$, or equivalently, if the map $\pi_g: H \to H$ given by $a\mapsto g\cdot a$, for  any $g\in G$, is a coalgebra morphism, then the smash product algebra $H \# \kk G$ has structure of Hopf algebra with operations as in \eqref{eq:smashproduct}.
\end{enumerate}
\end{example}

Next, we recall Nikshych's generalization of the Cartier--Gabriel--Kostant--Milnor--Moore theorem to the weak Hopf setting.

\begin{theorem}[{\cite[Theorem 3.2.4]{NDthesis}}]
\label{thm:weakCGKMM}
 Any cocommutative weak Hopf algebra $H$ is isomorphic to $U_X(\GLie) \# \kk \GG$ as weak Hopf algebras, for some $X$-Lie algebroid $\GLie$ and $X$-groupoid $\GG$. 
\end{theorem}

Now, we characterize the morphisms in $\XCocomWHA$.

\begin{definition} \label{def:linearhom}
Let $R$ and $S$ be two $\kk$-algebras and let $\beta: R \to S$ be a $\kk$-algebra morphism. Let $M \in R\lmod$ and $N \in S\lmod$. We call a group morphism
$\alpha: M \to N$ a {\it $\beta$-linear morphism} if $\alpha(r \cdot m) = \beta(r) \cdot \alpha(m)$ for all $m \in M$ and $r \in R$.
\end{definition}

\begin{lemma} \label{lem:new1}
Let $H = U_X(\GLie) \# \kk\GG$ and $H' = U_X(\GLie') \# \kk \GG'$, where $\GG$ and $\GG'$ are $X$-groupoids, $\GLie$ and $\GLie'$ are $X$-Lie algebroids and $U_X(\GLie)$ is a $\GG$-module algebra and $U_X(\GLie')$ is a $\GG'$-module algebra.
\begin{enumerate}
    \item Suppose that $\pi: \GG \to \GG'$ is an $X$-groupoid morphism and $\tau: \GLie \to \GLie'$ is an $X$-Lie algebroid morphism such that for the corresponding $X$-weak Hopf algebra morphisms $\widetilde{\pi}: \kk \GG \to \kk \GG'$ and $\widetilde{\tau}: U_X(\GLie) \to U_X(\GLie')$, $\widetilde{\tau}$ is a $\widetilde{\pi}$-linear morphism. Then there is an induced $X$-weak Hopf algebra morphism $\phi: H \to H'$. 
    \item Every $X$-weak Hopf algebra morphism from $H$ to $H'$ arises in this way.
\end{enumerate}
\end{lemma}
\begin{proof}
(a) Assume that $\GLie=\bigoplus_{x\in X}\mf{g}_x$ and $\GLie'=\bigoplus_{x\in X}\mf{g}'_x$. For two given $X$-weak Hopf algebra morphisms $\widetilde{\pi}$, $\widetilde{\tau}$ in (a), we can define a map 
\begin{equation}\label{eq:lem6.15}
   \phi: H\longrightarrow H', \quad p_x \# g \mapsto \widetilde{\tau}_x(p_x) \# \widetilde{\pi}(g), 
\end{equation}
where $p_x\in \mf{g}_x, g\in \GG$, for any $x\in X$.
First, we will show that $\phi$ is well-defined. Suppose that $p\otimes_{\kk \GG_0}g=q\otimes_{\kk \GG_0}h$ in $U_X(\GLie)\#\kk \GG$.
Then there exists $\beta \in \kk \GG_0$ such that $g = \beta h$ and hence $p = q \cdot \beta\inv=S\inv(\beta\inv)\cdot q$ by Definition~\ref{def:smashpro}. Since $\widetilde{\tau}$ is a $\widetilde{\pi}$-linear morphism, we have $\widetilde{\tau}(g\cdot p)=\widetilde{\pi}(g)\cdot \widetilde{\tau}(p)$ for $p\in \GLie$ and $g\in \GG$.
Then 
\begin{align*}
\widetilde{\tau}(p) \# \widetilde{\pi}(g)& = \widetilde{\tau}(S\inv(\beta \inv)\cdot q) \# \widetilde{\pi}( \beta h) =(\widetilde{\pi}(S\inv(\beta\inv))\cdot\widetilde{\tau}(q))\#\widetilde{\pi}(\beta h)\\
&=(\widetilde{\tau}(q)\cdot S(\widetilde{\pi}(S\inv(\beta\inv))))\#\widetilde{\pi}(\beta h)
=\widetilde{\tau}(q)\#\widetilde{\pi}(\beta\inv)\widetilde{\pi}(\beta h)\\
&=\ \widetilde{\tau}(q) \# \widetilde{\pi}(h).
\end{align*}
The second-to-last equality holds since $\widetilde{\pi}$ is a weak Hopf algebra morphism. 

By \eqref{eq:smashproduct} and the fact that $\widetilde{\pi}$ and $\widetilde{\tau}$ are $\kk$-coalgebra maps, we conclude that $\phi$ is a $\kk$-coalgebra map. Moreover, $\phi$ is unital as $\widetilde{\pi}$ and $\widetilde{\tau}$ are unital. Next we show that $\phi$ is multiplicative:
\begin{align*}
\phi((p \# g)( q \# h)) &= \phi(p(g \cdot q) \; \# \; gh)) 
= \widetilde{\tau}(p(g \cdot q))\; \# \; \widetilde{\pi}(gh) \\
&= \widetilde{\tau}(p) \widetilde{\tau}(g \cdot q)\; \# \; \widetilde{\pi}(g) \widetilde{\pi}(h) 
\overset{(*)}{=}\widetilde{\tau}(p) (\widetilde{\pi}(g) \cdot \widetilde{\tau}(q)) \; \# \; \widetilde{\pi}(g) \widetilde{\pi}(h) \\
&\overset{(**)}{=}(\widetilde{\tau}(p)\; \# \; \widetilde{\pi}(g))
(\widetilde{\tau}(q) \; \# \; \widetilde{\pi}(h))
=\phi(p \# g)\phi(q \# h), 
%= 
\end{align*}
for $g,h \in \GG$ and $p,q \in \GLie$. Here, in $(\ast)$ we used that $\widetilde{\tau}$ is a $\widetilde{\pi}$-linear morphism, and in $(\ast \ast)$ that $\widetilde{\tau}$ is a $\kk$-coalgebra map and Definition~\ref{def:smashpro}. Also, for $g \in \GG$ and $p \in \GLie$ we have
\begin{align*}
   ( S_{H'} \circ \phi ) (p \# g) &= (1_{U_X(\GLie')} \# \widetilde{\tau}(g)^{-1})(S_{U_X(\GLie')}(\widetilde{\pi}(p)) \# 1_{\kk\GG'}) \\
   & \overset{(*)}{=} (\widetilde{\pi}(1_{U_X(\GLie)}) \# \widetilde{\tau}(g^{-1}))( \widetilde{\pi}(S_{U_X(\GLie)}(p)) \# \widetilde{\tau} (1_{\kk\GG}))\\
   &= (\phi \circ S_H ) (p \# g),
\end{align*}
which proves that $\phi$ is a weak Hopf algebra morphism. Here, in $(*)$ we used that both $\widetilde{\pi}$ and $\widetilde{\tau}$ are weak Hopf algebra morphisms. Lastly, since for each $x\in X$ we have
\[
\phi(1_{U(\mf{g}_x)}\# e_x) = \widetilde{\tau}(1_{U(\mf g_x)})\#\widetilde{\pi}(e_x)=1_{U(\mf g'_x)} \# e_x,
\]
then $\phi$ preserves the set $X$ by Lemma~\ref{lem:Hsbase}. Therefore $\phi$ is an $X$-weak Hopf algebra morphism.

(b) Let $\psi: H \to H'$ be an $X$-weak Hopf algebra morphism. Then the natural restrictions $\bar{\pi}: \kk \GG \to H'$ and $\bar{\tau}: U_X(\GLie) \to H'$ of $\psi$ are both $X$-weak Hopf algebra morphisms. By Theorems~\ref{thm:repkGrp} and \ref{thm:adjLie}, there exist an $X$-groupoid morphism $\pi=j_1\circ \bar{\pi} \circ i_1: \GG \to \GG' $ and an $X$-Lie algebroid morphism $\tau=j_2\circ \bar{\tau} \circ i_2: \GLie  \to \GLie'$, where $\GG\overset{i_1}{\to} H\overset{\psi}{\to} H'\overset{j_1}{\to} \GG'$ and $\GLie\overset{i_2}{\to} H\overset{\psi}{\to} H'\overset{j_2}{\to} \GLie'$ are the natural inclusions and projections. In other words, $\pi(g)=\psi(1_{U_X(\GLie)} \# g) \in \GG'_1$, for all $g\in \GG_1$, and $\tau(p)=\psi(p \# 1_{\kk\GG}) \in \GLie'$, for all $p\in\GLie$. But, again by Theorems~\ref{thm:repkGrp} and \ref{thm:adjLie}, there are induced $X$-weak Hopf algebra morphisms $\widetilde{\pi}: \kk \GG \to \kk \GG'$ and $\widetilde{\tau}: U_X(\GLie) \to U_X(\GLie')$, given by $\widetilde{\pi}=\kk(\pi)$ and $\widetilde{\tau}=U_X(\tau)$. Since for any $g\in \GG_1$ and $p\in\mf{g}_{s(g)}$ we have
\begin{align*}
    \widetilde{\tau}(g\cdot p) &= \psi(g\cdot p \# 1_{\kk \GG}) = \psi(S^{-1}_H(p)gp \# 1_{\kk\GG}) = \psi(S^{-1}_{H}(p) \# 1_{\kk \GG}) \psi(1_{U_X(\GLie)} \# g) \psi (p \# 1_{\kk \GG})\\
    &= \widetilde{\tau}(S^{-1}_{H}(p))\widetilde{\pi}(g) \widetilde{\tau}(p) =   S^{-1}_{H'}(\widetilde{\tau}(p))\widetilde{\pi}(g) \widetilde{\tau}(p) = \widetilde{\pi}(g)\cdot \widetilde{\tau}(p),
\end{align*}
it follows that $\widetilde{\tau}$ is $\widetilde{\pi}$-linear. Using part (a), there exists an induced $X$-weak Hopf algebra morphism $\phi: H \to H'$ given by \eqref{eq:lem6.15}. By construction, it is clear that $\psi=\phi$.
\end{proof}

\begin{theorem}
\label{thm:cocomweakHopf}
Let $A = \bigoplus_{x \in X} A_x$ be an $X$-decomposable $\kk$-algebra, let \[K = U_X(\Der_X(A)) \# \kk(\XAutA),\] and let $H$ be a cocommutative $X$-weak Hopf algebra.
Write $H = U_X(\GLie) \# \kk \GG$ for some $X$-Lie algebroid $\GLie$ and some $X$-groupoid $\GG$. Then:
\begin{enumerate}[font=\upshape] 
\item The natural actions of $U_X(\Der_X(A))$ and $\kk(\XAutA)$ on $A$ give $A$ the structure of a $K$-module algebra. We denote this action $\ell_{K, A}$ and write $k \rt a := \ell_{K, A}(k \otimes a)$ where $k \otimes a \in K \otimes A$.
 
\item Suppose that $A$ is an $H$-module algebra via $\ell_{H, A}$ and denote $h \cdot a := \ell_{H, A}(h \otimes a)$ for $h \otimes a \in H \otimes A$. Then there is a unique $X$-weak Hopf algebra morphism $\phi: H \to K$ such that $h \cdot a = \phi(h) \rt a$ for all $h \otimes a \in H \otimes A$.

\item Every $X$-weak Hopf algebra morphism $\phi: H \to K$ gives $A$ the structure of an $H$-module algebra via $\ell_{H, A}(h \otimes a) = \ell_{K, A} (\phi(h) \otimes a)$ for all $h \otimes a \in H \otimes A$.
\end{enumerate}

Hence, $\Sym_{\XCocomWHA}(A) = K$.
\end{theorem}

\begin{proof}
It follows from \Cref{prop:smashproduct} that $U_X(\Der_X(A))\# \kk(\XAutA)$ is a cocommutative weak Hopf algebra.

 (a) By the assumption, there are two action maps $\ell_{\kk(\XAutA), A}$ in Proposition~\ref{prop:kGrpd} and $\ell_{U_X(\Der_X(A)), A}$ in Proposition~\ref{prop:U(GLie)}. Hence, one can define an action map $\ell_{K,A}: K\otimes A\to A$ via 
 \[
 k\rt a:=\ell_{K,A}((\delta\#\alpha)\otimes a)=\delta\rt(\alpha\rt a), \qquad k=\delta\# \alpha \in K, \quad a,b\in A. 
 \]
Furthermore, 
\begin{align*}
    k\rt (ab)&=\delta\rt(\alpha\rt (ab))=\delta\rt((\alpha_1\rt a)(\alpha_2\rt b))\\
    &=(\delta_1\rt (\alpha_1\rt a))(\delta_2\rt(\alpha_2\rt b))=(k_1\rt a)(k_2\rt b),
\end{align*}
and
\[k\rt 1_A=\delta\rt(\alpha\rt 1_A)=
((\varepsilon_{U_X(\Der_X(A)})_t(\delta)\#(\varepsilon_{\kk(\XAutA)})_t(\alpha))\cdot 1_A\overset{\text{Lem.~\ref{lem:Hsbase}}}{=}\varepsilon_K(k)\cdot 1_A.\]
Therefore, $A$ is a $K$-module algebra.

(b) Since $A$ is an $H$-module algebra,
we can define a $\kk\GG$-module algebra structure on $A$ via $g \cdot a:= \ell_{H, A}(1_{U_X(\GLie)} \# g \otimes a)$ for $g \in \GG$ and $a \in A$. It is clear that $A$ is a $\kk\GG$-module. By \eqref{eq:Hbase}, we have that
\begin{equation}
\label{eq:HHbase}
1_{U(\mf g_x)}\# g= 1_{U(\mf g_x)}\# e_{t(g)}g=\delta_{x,\,t(g)}1_{U(\mf g_x)}\#g, 
\end{equation} 
and therefore
\begin{equation}
\label{eq:HHbase2}
\begin{aligned}
 \Delta(1_{U_X(\GLie)}\#g)& = \sum_{y\in X} (1_{U(\mf g_y)}\#g)\otimes (1_{U(\mf g_y)}\#g) \overset{\eqref{eq:HHbase}}{=}(1_{U(\mf g_{t(g)})}\#g)\otimes(1_{U(\mf g_{t(g)})}\#g)\\
 &\overset{\eqref{eq:HHbase}}{=}(1_{U_X(\GLie)}\#g)\otimes (1_{U_X(\GLie)}\#g).
\end{aligned}
\end{equation} 
Notice that
\begin{align*}
g \cdot (ab) &= \ell_{H, A}(1_{U_X(\GLie)} \# g \otimes ab)=(1_{U_X(\GLie)} \# g)\cdot (ab)\\
&\overset{\eqref{eq:HHbase2}}{=}\ell_{H, A}(1_{U_X(\GLie)} \# g \otimes a)\cdot \ell_{H, A}(1_{U_X(\GLie)} \# g \otimes b)\\
&=(g\cdot a)(g\cdot b).
\end{align*}
Furthermore, 
\begin{align*}
g \cdot 1_A &=\ell_{H, A}(1_{U_X(\GLie)} \# g\otimes 1_A)=(1_{U_X(\GLie)} \# g)\cdot 1_A\\
&= \ept(1_{U_X(\GLie)} \# g) \cdot1_A\overset{\text{Lem.~\ref{lem:Hsbase}}}{=}(1_{U_X(\GLie)} \# (\varepsilon_{\kk \GG})_t(g))\cdot 1_A\\
&=\ell_{H, A}(1_{U_X(\GLie)} \# (\varepsilon_{\kk \GG})_t(g)\otimes 1_A)\\
&=(\varepsilon_{\kk \GG})_t(g)\cdot 1_A.
\end{align*}

Next, define a $\GLie$-action on $A$ via $p_x\cdot b:= \ell_{H, A}(p_x \# 1_{\kk G} \otimes b)$ for all $p_x\in \mf{g}_x$ with $b\in A$. By a similar argument to \eqref{eq:Hbase}, we show $p_x\# e_y=\delta_{x,y}p_x\#e_x$ and so $(p_x\#1_{\kk\GG})=(p_x\# e_x)$ for $p_x\in {\mf g}_x$. It is straightforward to check that 
\begin{align*}
p_x\cdot(ab)
 &=(p_x\cdot a)(1_{U(\mf{g}_x)}\cdot b)+(1_{U(\mf{g}_x)}\cdot a)(p_x\cdot b),\\
p_x \cdot 1_A &=0= \ept(p_x) \cdot 1_A,
\end{align*}
since
\begin{equation*}
    \ept(p_y)=(p_y)_1 S((p_y)_2) = 1_{U(\mf g_y)}S(p_y)+p_yS(1_{U(\mf g_y)})=-p_y + p_y=0,
\end{equation*}
for $p_y\in \mf g_y$ and $a, b\in A$.
By Propositions~\ref{prop:Liealgebroidalg} and \ref{prop:U(GLie)}, $A$ is a $U_X(\GLie)$-module algebra. Consequently, there exist $X$-weak Hopf algebra maps 
\begin{equation*}
    \widetilde{\pi}:=\nu_{\kk(\XAutA)}\kk(\pi): \kk \GG \to \kk(\XAutA) \quad \text{and} \quad \widetilde{\tau}: U_X(\GLie) \to U_X(\Der_X(A))
\end{equation*}
by Propositions~\ref{prop:kGrpd}(b) and \ref{prop:U(GLie)}(b), respectively. Similar to the argument used to prove Lemma~\ref{lem:new1}(b), we can show that $\widetilde{\tau}$ is $\widetilde{\pi}$-linear.
It follows from Lemma~\ref{lem:new1}(a) and Propositions~\ref{prop:kGrpd}(c) and \ref{prop:U(GLie)}(c) that there exists a weak Hopf algebra morphism from $K$ to $H$ such that $h\cdot a=\phi(h)\cdot a$ for $h\in H$.
 
(c) For any given weak Hopf algebra morphism $\phi: H\to K$, we will show that $A$ is an $H$-module algebra via $h \cdot a:= \ell_{H, A}(h\otimes a)=\ell_{K, A}(\phi(h) \otimes a)=\phi(h)\cdot a$, for all $h \in H$ and $a \in A$. Since $\phi$ is a weak Hopf algebra morphism,
\[h\cdot (ab)=\phi(h)\cdot (ab)=(\phi(h_1)\cdot a)(\phi(h_2)\cdot b)=(h_1\cdot a)(h_2\cdot b).\]
It follows that
\begin{equation}
\label{eq:eptt}
\phi((1_H)_1)\otimes \phi((1_H)_2)
=\Delta(\phi(1_H))=\Delta(1_K)= (1_K)_1\otimes (1_K)_2.
\end{equation}
Thus,
\begin{align*}
h \cdot 1_A 
&=\phi(h)\cdot 1_A =(\varepsilon_K)_t(\phi(h)) \cdot1_A\\
&\overset{\eqref{eq:eptt}}{=}
\left(\sum \varepsilon_K(\phi((1_H)_1)\phi(h))\phi((1_H)_2)\right) \cdot1_A 
=(\varepsilon_K \phi\otimes \phi)((1_H)_1h\otimes (1_H)_2) \cdot 1_A\\
&\overset{(\ast)}{=}(\varepsilon_H\otimes \phi)((1_H)_1h\otimes (1_H)_2) \cdot1_A
= \phi\left((\varepsilon_H)_t(h)\right)\cdot 1_A\\
&\overset{(\ast\ast)}{=}(\varepsilon_H)_t(h)\cdot 1_A
\end{align*}
as desired. Here, in $(\ast)$ we used that $\phi$ is a counital map, and in $(\ast \ast)$ the assumption. 

Therefore,  $\Sym_{\XCocomWHA}(A) = K$ as desired.
\end{proof}

 We make the following remark, which is the special case of \cref{thm:cocomweakHopf} when $|X|=1$.

\begin{remark}\label{CocomHopf}
Let $A$ be a $\kk$-algebra. Let $K = U(\Der(A)) \# \kk(\AutA)$. Then $A$ is a $K$-module algebra and for a cocommutative Hopf algebra $H$, the following are equivalent.
\begin{enumerate}[font=\upshape]
 \item[(i)] $A$ is an $H$-module algebra.

 \item[(ii)] There exists a Hopf algebra morphism $\phi: H \to K$ giving $A$ the structure of an $H$-module algebra via $\ell_{H, A}(h\otimes a)=\ell_{K, A}(\phi(h)\otimes a)$ for all $h\otimes a\in H\otimes A$. 
\end{enumerate}
Hence, $\Sym_{\CocomHopf}(A) = K$.
\end{remark}

As mentioned in the introduction, one potential research direction is to extend these results to the non-cocommutative case. For instance, there has been much recent activity on partial actions of Hopf-like structures in this case (see, e.g., \cite{FMF,FMS,MPS}).

%%%%%%%%%%%%%%%%%%%%%%%%%%%%%%%%%%%%%%%
%%%%%%%%%%%%%%%%%%%%%%%%%%%%%%%%%%%%%%%
%%%%%%%%%%%%%%%%%%%%%%%%%%%%%%%%%%%%%%%
%%%%%%%%%%%%%%%%%%%%%%%%%%%%%%%%%%%%%%%%%%%%%%%%%%%%%%%%%%%%%%%%%%%%%
	
\section{Examples: Actions on polynomial algebras} \label{sec:examples}
	
In this section we illustrate the results in this paper for \emph{graded} actions on polynomial algebras by general linear Hopf-like structures. 
As a warm up, we begin by studying the (classical) indecomposable case before considering the $X$-decomposable case.
	
\subsection{Indecomposable module algebra case} \label{sec:examples-indecomp}

Let $A:=\kk[x_1, \dots, x_n]$ be a polynomial algebra, which is isomorphic to the symmetric algebra $S(V)$ on an $n$-dimensional vector space $V$. We consider the following group and Lie algebra, respectively,
\[
T_n:= \AutA \; \; \; \qquad \text{ and } \; \; \; \qquad
W_n := \Der(A).
\]
When $n \geq 3$, the group $T_n$ contains an automorphism of wild type \cite{SU} and so is not fully understood. Hence, we restrict our attention to the subgroup of \emph{graded} automorphisms of $A$, which we identify with the general linear group:
\[
\GL_n(\kk) = \GL(V) =: \Aut_{\textnormal{GrAlg}}(A).
\]

On the other hand, $W_n$ is well-known to be the infinite-dimensional Lie algebra consisting of derivations of the form $f_1(\underline{x}) \frac{\partial}{\partial x_1} + \cdots + f_n(\underline{x}) \frac{\partial}{\partial x_n}$, for $f_i(\underline{x}) \in A$ (see, e.g., \cite[Section~1.2]{Bahturin}). The general linear Lie algebra is a Lie subalgebra of $W_n$ given as follows:
\[
\mf{gl}_n(\kk) = \text{span}_{\kk}\left\{\textstyle x_i \frac{\partial}{\partial x_j}\right\}_{i,j = 1, \dots, n} =: \Der_{\textnormal{Lin}}(A).
\]
Moreover, $\GL_n(\kk)$ acts on $\mf{gl}_n(\kk)$ by conjugation after identifying $x_i \frac{\partial}{\partial x_j}$ with the elementary matrix $E_{i,j}$. With the maps,
\[
\pi: \GL_n(\kk) \longrightarrow T_n \; \; \text{(inclusion of groups)}, \quad 
\tau: \mf{gl}_n(\kk) \longrightarrow W_n \; \; \text{(inclusion of Lie algebras)},
\]
we obtain the Hopf algebra maps, 
\[
\widetilde{\pi}: \kk\GL_n(\kk) \longrightarrow \kk T_n \; \; \; \qquad \text{ and } \qquad \; \; \; 
\widetilde{\tau}: U(\mf{gl}_n(\kk)) \longrightarrow U(W_n)
\]
and $\widetilde{\tau}$ is $\widetilde{\pi}$-linear. Hence by Remarks~\ref{rem:group}, \ref{rem:kGrp}, \ref{rem:Lie}, \ref{rem:U(Lie)}, and \ref{CocomHopf}, we obtain the following result.
\begin{proposition}
The polynomial algebra $\kk[x_1, \dots, x_n]$ is a module algebra over the following general linear Hopf-like structures:
$$\GL_n(\kk), \qquad \mf{gl}_n(\kk), \qquad \kk\GL_n(\kk), \qquad U(\mf{gl}_n(\kk)), \qquad U(\mf{gl}_n(\kk)) \# \kk\GL_n(\kk).$$
\end{proposition}
	
\subsection{$X$-decomposable module algebra case} \label{sec:examples-Xdecomp} 

For $x \in X$, let $A_x:=S(V_x)$ be the symmetric algebra on a finite-dimensional vector space $V_x$. Let $n_x$ denote $\dim(V_x)$. Take $V:=\bigoplus_{x \in X} V_x$ to be the corresponding $X$-decomposable vector space, and $A := \bigoplus_{x \in X} A_x$ to be the corresponding $X$-decomposable $\kk$-algebra. Take
\[
T_X := \XAutA,
\]
to be the groupoid from Definition~\ref{def:autgroupoid2}, i.e., the objects of $T_X$ are the elements of $X$ and for $x,y \in X$, the morphisms $\Hom_{T_X}(x,y)$ are the unital $\kk$-algebra isomorphisms $A_x \to A_y$. 

Since $\GL(V_x)$ is a subgroupoid of the groupoid $\GL_X(V)$ from Definition~\ref{def:autgroupoid1}, the result below is straightforward, where the second statement follows from Theorem~\ref{thm:repkGrp}(b) (or Proposition~\ref{prop:kGrpd}).

\begin{lemma} \label{lem:poly-grpd}
There is an $X$-groupoid morphism 
\[
\pi: \GL_X(V) \longrightarrow T_X,
\]
given by $\pi(V_x) = S(V_x)$ for all $x \in X$. This also yields an $X$-weak Hopf algebra morphism:
\[
\widetilde{\pi}: \kk \GL_X(V) \longrightarrow \kk T_X. 
\]
\end{lemma}

On the other hand, take
\[
W_X:= \Der_X(A):= \textstyle \bigoplus_{x \in X} \Der(A_x),
\]
to be the $X$-Lie algebroid from Definition~\ref{def:Liealgebroid} and Notation~\ref{not:GLLie2}. The next result is straightforward, with the second statement following from Proposition~\ref{prop:U(GLie)}.

\begin{lemma} \label{lem:poly-Liealgd}
There is an $X$-Lie algebroid morphism 
\[
\tau: \mf{GL}_X(V) \to \Der_X(A),
\]
given by inclusion. This also yields an $X$-weak Hopf algebra morphism:
\[
\widetilde{\tau}:  U_X(\mf{GL}_X(V)) \longrightarrow U_X( \Der_X(A)).
\]
\end{lemma}

Working component-wise, we also have the following result.

\begin{lemma} \label{lem:poly-linearity}
Given $\widetilde{\pi}$ from Lemma~\ref{lem:poly-grpd}, we have that $\widetilde{\tau}$ from Lemma~\ref{lem:poly-Liealgd} is $\widetilde{\pi}$-linear.   
\end{lemma}

Finally, Lemmas~\ref{lem:poly-grpd}, \ref{lem:poly-Liealgd}, and~\ref{lem:poly-linearity} yield the following consequences of Propositions~\ref{prop:groupoidalg}, \ref{prop:kGrpd}, and \ref{prop:Liealgebroidalg} and Theorem~\ref{thm:cocomweakHopf}.

\begin{proposition}
For an $X$-decomposable vector space $V:=\bigoplus_{x \in X} V_x$, the $X$-decomposable $\kk$-algebra $\bigoplus_{x \in X} S(V_x)$ is a module algebra over the following general linear $X$-Hopf-like structures:
\[
\GL_X(V), \quad\; \mf{GL}_X(V), \quad \; \kk\GL_X(V), \quad \; U_X(\mf{GL}_X(V)), \quad \; \bigoplus_{x \in X} U_X(\mf{GL}_X(V)) \# \kk\GL_X(V).
\]
\end{proposition}

\section*{Declarations}
 
Ethical Approval: not applicable.

 \medskip 
Funding:
\begin{itemize}
    \item  The first author was partially supported by the research fund of the Department of Mathematics, Universidad Nacional de Colombia - Sede Bogot\'a, Colombia, HERMES code 52464, and the Fulbright Visiting Student Researcher Program.

    \item  The fourth author was partially supported by an AMS--Simons Travel Grant and Simons Foundation grant \#961085.
\end{itemize}

 \medskip
Availability of data and materials: not applicable.

%%%%%%%%%%%%%%%%%%%%%%%%%%%%%%%%%%%%%%%%%%%%%%%%%%%%%%%%%%%%%%%%%%%%%
%%%%%%%%%%%%%%%%%%%%%%%%%%%%%%%%%%%%%%%%%%%%%%%%%%%%%%%%%%%%%%%%%%%%%	%%%%%%%%%%%%%%%%%%%%%%%%%%%%%%%%%%%%%%%%%%%%%%%%%%%%%%%%%%%%%%%%%%%%%

\section*{Acknowledgements} 
We thank Chelsea Walton for many insightful conversations and for her valuable comments on our manuscript. We also extend our gratitude to the anonymous referee for their feedback and helpful corrections, which greatly improved the overall quality of the paper. Finally, the first, second, and fourth authors would like to congratulate the third author on the birth of her child, who happily (and indirectly) contributed to the writing of this paper.

\bibliographystyle{alpha}
\bibliography{biblio}

\begin{thebibliography}{HWWW23}

\bibitem[Bah21]{Bahturin}
Yuri Bahturin.
\newblock {\em Identical relations in {L}ie algebras}, volume~68 of {\em De
  Gruyter Expositions in Mathematics}.
\newblock De Gruyter, Berlin, second edition, 2021.

\bibitem[BCJ11]{BCJ}
Gabriella B{\"o}hm, Stefaan Caenepeel, and Kris Janssen.
\newblock Weak bialgebras and monoidal categories.
\newblock {\em Comm. Algebra}, 39(12):4584--4607, 2011.

\bibitem[BGTLC14]{BGL}
Gabriella B{\"o}hm, Jos\'e{} G\'omez-Torrecillas, and Esperanza
  L\'opez-Centella.
\newblock On the category of weak bialgebras.
\newblock {\em J. Algebra}, 399:801--844, 2014.

\bibitem[BHS11]{BHS2011}
Ronald Brown, Philip~J. Higgins, and Rafael Sivera.
\newblock {\em Nonabelian algebraic topology}, volume~15 of {\em EMS Tracts in
  Mathematics}.
\newblock European Mathematical Society (EMS), Z\"urich, 2011.
\newblock Filtered spaces, crossed complexes, cubical homotopy groupoids, With
  contributions by Christopher D. Wensley and Sergei V. Soloviev.

\bibitem[BNS99]{BNS}
Gabriella B{\"o}hm, Florian Nill, and Korn\'el Szlach\'anyi.
\newblock Weak {H}opf algebras. {I}. {I}ntegral theory and {$C^*$}-structure.
\newblock {\em J. Algebra}, 221(2):385--438, 1999.

\bibitem[BP12]{BP2012}
Dirceu Bagio and Antonio Paques.
\newblock Partial groupoid actions: globalization, {M}orita theory, and
  {G}alois theory.
\newblock {\em Comm. Algebra}, 40(10):3658--3678, 2012.

\bibitem[BS00]{BSK}
Gabriella B{\"o}hm and Korn\'el Szlach\'anyi.
\newblock Weak {H}opf algebras. {II}. {R}epresentation theory, dimensions, and
  the {M}arkov trace.
\newblock {\em J. Algebra}, 233(1):156--212, 2000.

\bibitem[CDG00]{C-DG}
S.~Caenepeel and E.~De~Groot.
\newblock Modules over weak entwining structures.
\newblock In {\em New trends in {H}opf algebra theory ({L}a {F}alda, 1999)},
  volume 267 of {\em Contemp. Math.}, pages 31--54. Amer. Math. Soc.,
  Providence, RI, 2000.

\bibitem[CEW15]{CEW1}
Juan Cuadra, Pavel Etingof, and Chelsea Walton.
\newblock Semisimple {H}opf actions on {W}eyl algebras.
\newblock {\em Adv. Math.}, 282:47--55, 2015.

\bibitem[CEW16]{CEW2}
Juan Cuadra, Pavel Etingof, and Chelsea Walton.
\newblock Finite dimensional {H}opf actions on {W}eyl algebras.
\newblock {\em Adv. Math.}, 302:25--39, 2016.

\bibitem[CKWZ16]{CKWZ3}
Kenneth Chan, Ellen Kirkman, Chelsea Walton, and James~J. Zhang.
\newblock Quantum binary polyhedral groups and their actions on quantum planes.
\newblock {\em J. Reine Angew. Math.}, 719:211--252, 2016.

\bibitem[CL09]{CF2009}
Dongming Cheng and Fang Li.
\newblock The structure of weak {H}opf algebras corresponding to {$U_q({\rm
  sl}_2)$}.
\newblock {\em Comm. Algebra}, 37(3):729--742, 2009.

\bibitem[CM84]{CM1984}
Miriam Cohen and Susan Montgomery.
\newblock Group-graded rings, smash products, and group actions.
\newblock {\em Trans. Amer. Math. Soc.}, 282(1):237--258, 1984.

\bibitem[CR24]{CR}
Fabio Calder\'{o}n and Armando Reyes.
\newblock On the (partial) representation category of weak {H}opf algebras.
\newblock {\em \textnormal{{I}n preparation}}, 2024.

\bibitem[EGNO15]{EGNO}
Pavel Etingof, Shlomo Gelaki, Dmitri Nikshych, and Victor Ostrik.
\newblock {\em Tensor categories}, volume 205 of {\em Mathematical Surveys and
  Monographs}.
\newblock American Mathematical Society, Providence, RI, 2015.

\bibitem[EW14]{EW1}
Pavel Etingof and Chelsea Walton.
\newblock Semisimple {H}opf actions on commutative domains.
\newblock {\em Adv. Math.}, 251:47--61, 2014.

\bibitem[EW16]{EW2}
Pavel Etingof and Chelsea Walton.
\newblock Finite dimensional {H}opf actions on algebraic quantizations.
\newblock {\em Algebra Number Theory}, 10(10):2287--2310, 2016.

\bibitem[EW17]{EW3}
Pavel Etingof and Chelsea Walton.
\newblock Finite dimensional {H}opf actions on deformation quantizations.
\newblock {\em Proc. Amer. Math. Soc.}, 145(5):1917--1925, 2017.

\bibitem[FMF22]{FMF}
Eneilson Fontes, Grasiela Martini, and Graziela Fonseca.
\newblock Partial actions of weak {H}opf algebras on coalgebras.
\newblock {\em J. Algebra Appl.}, 21(1):Paper No. 2250012, 35, 2022.

\bibitem[FMS24]{FMS}
Graziela Fonseca, Grasiela Martini, and Leonardo Silva.
\newblock Partial (co)actions of {T}aft and {N}ichols {H}opf algebras on
  algebras.
\newblock {\em J. Pure Appl. Algebra}, 228(1):Paper No. 107455, 25, 2024.

\bibitem[Hay99]{Hayashi99b}
Takahiro Hayashi.
\newblock Face algebras and unitarity of {${\rm SU}(N)_L$}-{TQFT}.
\newblock {\em Comm. Math. Phys.}, 203(1):211--247, 1999.

\bibitem[HWWW23]{HWWW20}
Hongdi Huang, Chelsea Walton, Elizabeth Wicks, and Robert Won.
\newblock Universal quantum semigroupoids.
\newblock {\em J. Pure Appl. Algebra}, 227(2):Paper No. 107193, 34, 2023.

\bibitem[IR19]{IRPaper19}
Alberto Ibort and Miguel Rodr{\'{\i}}guez.
\newblock On the structure of finite groupoids and their representations.
\newblock {\em Symmetry}, 11(3):414, 2019.

\bibitem[IRg20]{IR19}
Alberto Ibort and Miguel~A. Rodr\'i{}~guez.
\newblock {\em An introduction to groups, groupoids and their representations}.
\newblock CRC Press, Boca Raton, FL, 2020.

\bibitem[Kir16]{K}
Ellen~E. Kirkman.
\newblock Invariant theory of {A}rtin-{S}chelter regular algebras: a survey.
\newblock In {\em Recent developments in representation theory}, volume 673 of
  {\em Contemp. Math.}, pages 25--50. Amer. Math. Soc., Providence, RI, 2016.

\bibitem[Mac87]{MacK}
K.~Mackenzie.
\newblock {\em Lie groupoids and {L}ie algebroids in differential geometry},
  volume 124 of {\em London Mathematical Society Lecture Note Series}.
\newblock Cambridge University Press, Cambridge, 1987.

\bibitem[Mon93]{Mo}
Susan Montgomery.
\newblock {\em Hopf algebras and their actions on rings}, volume~82 of {\em
  CBMS Regional Conference Series in Mathematics}.
\newblock Conference Board of the Mathematical Sciences, Washington, DC; by the
  American Mathematical Society, Providence, RI, 1993.

\bibitem[MPS23]{MPS}
Grasiela Martini, Antonio Paques, and Leonardo~Duarte Silva.
\newblock Partial actions of a {H}opf algebra on its base field and the
  corresponding partial smash product algebra.
\newblock {\em J. Algebra Appl.}, 22(6):Paper No. 2350140, 29, 2023.

\bibitem[Nik01]{NDthesis}
Dmitri Nikshych.
\newblock {\em Quantum groupoids, their representation categories, symmetries
  of von {N}eumann factors, and dynamical quantum groups}.
\newblock ProQuest LLC, Ann Arbor, MI, 2001.
\newblock Thesis (Ph.D.)--University of California, Los Angeles.

\bibitem[Nik02]{Nik02}
Dmitri Nikshych.
\newblock On the structure of weak {H}opf algebras.
\newblock {\em Adv. Math.}, 170(2):257--286, 2002.

\bibitem[Nik04]{Nik04}
Dmitri Nikshych.
\newblock Semisimple weak {H}opf algebras.
\newblock {\em J. Algebra}, 275(2):639--667, 2004.

\bibitem[NV02]{NV}
Dmitri Nikshych and Leonid Vainerman.
\newblock Finite quantum groupoids and their applications.
\newblock In {\em New directions in {H}opf algebras}, volume~43 of {\em Math.
  Sci. Res. Inst. Publ.}, pages 211--262. Cambridge Univ. Press, Cambridge,
  2002.

\bibitem[PF14]{PF13}
Antonio Paques and Daiana Fl\^ores.
\newblock Duality for groupoid (co)actions.
\newblock {\em Comm. Algebra}, 42(2):637--663, 2014.

\bibitem[Pra67]{pradines1967}
Jean Pradines.
\newblock Th\'eorie de {L}ie pour les groupo\"ides diff\'erentiables. {C}alcul
  diff\'erenetiel dans la cat\'egorie des groupo\"ides infinit\'esimaux.
\newblock {\em C. R. Acad. Sci. Paris S\'er. A-B}, 264:A245--A248, 1967.

\bibitem[PT14]{PT2014}
Antonio Paques and Tha\'isa Tamusiunas.
\newblock A {G}alois-{G}rothendieck-type correspondence for groupoid actions.
\newblock {\em Algebra Discrete Math.}, 17(1):80--97, 2014.

\bibitem[Rin63]{rinehart1963}
George~S. Rinehart.
\newblock Differential forms on general commutative algebras.
\newblock {\em Trans. Amer. Math. Soc.}, 108:195--222, 1963.

\bibitem[Sar22a]{Saracco20}
Paolo Saracco.
\newblock On anchored {L}ie algebras and the {C}onnes-{M}oscovici bialgebroid
  construction.
\newblock {\em J. Noncommut. Geom.}, 16(3):1007--1053, 2022.

\bibitem[Sar22b]{Saracco21}
Paolo Saracco.
\newblock Universal enveloping algebras of {L}ie-{R}inehart algebras as a left
  adjoint functor.
\newblock {\em Mediterr. J. Math.}, 19(2):Paper No. 92, 19, 2022.

\bibitem[SU04]{SU}
Ivan~P. Shestakov and Ualbai~U. Umirbaev.
\newblock The tame and the wild automorphisms of polynomial rings in three
  variables.
\newblock {\em J. Amer. Math. Soc.}, 17(1):197--227, 2004.

\bibitem[WWW22]{WWW}
Chelsea Walton, Elizabeth Wicks, and Robert Won.
\newblock Algebraic structures in comodule categories over weak bialgebras.
\newblock {\em Comm. Algebra}, 50(7):2877--2910, 2022.

\end{thebibliography}
\end{document}